\documentclass[11pt]{amsart}
\usepackage{amsmath}
\usepackage{amssymb}
\usepackage{amsthm}
\usepackage{amscd}
\usepackage{fontenc}
\usepackage{url}
\usepackage{hyperref}
\usepackage[all]{xy}
\usepackage{graphicx}
\usepackage{color}
\usepackage[shortlabels]{enumitem}
\usepackage[normalem]{ulem}
 
\usepackage{tikz}
 \usetikzlibrary{arrows}

\numberwithin{equation}{section}

\newcounter{AbcT}

\newcommand{\me}[1]{\textcolor{black}{#1}}
\newcommand{\alb}[1]{\textcolor{black}{#1}}

\newtheorem {AbcTheorem} [AbcT]{Theorem}

\newtheorem {Theorem}    {Theorem}[section]
\newtheorem* {Lemma9.3}    {Proposition 10.2}

\newtheorem* {Proposition10.3}    {Proposition 10.3}

\newtheorem {Question}    {Question}
\newtheorem {Problem}    [Question]{Problem}

\newtheorem {Lemma}      [Theorem]    {Lemma}
\newtheorem {Corollary}   [Theorem] {Corollary}
\newtheorem {Proposition}[Theorem]    {Proposition}

\newtheorem {Observation}[Theorem]    {Observation}

\theoremstyle{remark}

\newtheorem {Definition}  [Theorem]  {\bf{Definition}}
\newtheorem {Remark}  [Theorem]	{\bf{Remark}}
\newtheorem {Notation}  [Theorem]  {\bf{Notation}}

\newcounter{DM@bibnum}

\newcommand  {\abs}[1] {{#1}}
\newcommand  {\apf}[1] {{\mathcal{#1}}}

\newcommand {\R} {{\mathbb R}}

\newcommand {\Z} {{\mathbb Z}}

\newcommand{\la}{\langle}
\newcommand{\ra}{\rangle}

\def\IA{{\rm IA}}
\def\SL{{\rm SL}}
\def\PSL{{\rm PSL}}
\def\Sym{{\rm Sym}}
\def\Alt{{\rm Alt}}
\def\GL{{\rm GL}}

\def\Aut{{\rm Aut}}
\def\uAut{\underline{\Aut}}
\def\SAut{{\rm SAut}}

\def\SFAut{{\rm SAut}_{\dbF_p}}
\def\Inn{{\rm Inn}}

\def\Ker{{\rm Ker\,}}

\def\rk{{\rm rk\,}}

\def\QZ{{\rm QZ}}
\def\VZ{{\rm VZ}}
\def\Gama{{G}}
\def\iotta{{i}}

\def\NSL_2{{\mathcal N SL_2}}

\def\Comm{{\rm Comm}}
\def\SComm{{\rm SComm}}
\def\AComm{{\rm AComm}}

\def\prof{{\rm prof}}

\def\Lg{{\mathrm{Comm}}}
\def\VZ{{\rm VZ}}
\newcommand{\mon}{\ensuremath{\mathrm{Mon}}}%

\def\CpC{{\mathcal{C}}}
\def\SCpC{{\mathrm{S}\mathcal{C}}}

\newcommand {\absF}[1] {F_{#1}}
\newcommand {\propF}[1] {\mathbf{F}}
\newcommand {\propU}[1] {\mathbf{U}}
\newcommand {\propV}[1] {\mathbf{V}}
\newcommand {\propW}[1] {\mathbf{W}}
\newcommand {\propZ}[1] {\mathbf{Z}}

\def\eps{\varepsilon}

\def\lam{\lambda}            
                
\def\phi{\varphi}

\def\calA{{\mathcal A}}

\def\calF{{\mathcal F}}

\def\calS{{\mathcal S}}

\def\calU{{\mathcal U}}

\def\hbar{\bar h}

\def\dbF{{\mathbb F}}

\def\dbN{{\mathbb N}}

\def\dbQ{{\mathbb Q}}
\def\dbR{{\mathbb R}}

\def\dbZ{{\mathbb Z}}

\def\skv{{\vskip .1cm}}

\begin{document}

\title{On commensurators of free groups and free pro-$p$ groups}

\author[Barnea]{Yiftach Barnea}
\address{Yiftach Barnea, Department of Mathematics, Royal Holloway, University of London, Egham, Surrey TW20 0EX, UK}
\email{y.barnea@rhul.ac.uk}

\author[Ershov]{Mikhail Ershov}
\address{Mikhail Ershov, Department of Mathematics, University of Virginia, 141 Cabell Drive, Charlottesville, VA 22903, USA}
\email{ershov@virginia.edu}

\author[Le Boudec]{Adrien Le Boudec}
\address{Adrien Le Boudec, CNRS, UMPA - ENS Lyon, 46 all\'ee d'Italie, 69364 Lyon, France}
\email{adrien.le-boudec@ens-lyon.fr}

\author[Reid]{Colin D. Reid}
\address{Colin D. Reid, Western Sydney University, School of Computer, Data and Mathematical Sciences,
Penrith, NSW 2751, Australia}
\email{colinreid29@gmail.com}

\author[Vannacci]{Matteo Vannacci}
\address{Matteo Vannacci, Dipartimento di Matematica `Ulisse Dini', Universit\`a degli Studi di Firenze, Viale
Morgagni, 67/a - 50134 Firenze, Italy
}
\email{matteo.vannacci@unifi.it}

\author[Weigel]{Thomas Weigel}
\address{Thomas Weigel, Dipartimento di Matematica e Applicazioni, Universit\`a degli Studi di Milano-Bicocca,
Via Roberto Cozzi, 55 - 20126, Milano, Italy}
\email{thomas.weigel@unimib.it}


\begin{abstract}
We study the commensurators of free groups and free pro-$p$ groups, as well as certain subgroups of these. We prove that the commensurator  $\Comm(F)$ of a non-abelian free group of finite rank $F$ is not virtually simple, answering a question of Lubotzky. On the other hand, we exhibit a family of easy-to-define finitely generated subgroups of $\Comm(F)$ and show that some groups in this family are simple.

For a prime $p$, we also consider the $p$-commensurator $\Comm_p(F)$, which is the commensurator of $F$ viewed as a group with pro-$p$ topology. By contrast with $\Comm(F)$, we prove that $\Comm_p(F)$ has a simple subgroup of index at most $2$. Further, while the isomorphism class of $\Comm(F)$ does not depend on the rank of $F$, we prove that the isomorphism class of $\Comm_p(F)$ depends on the rank of $F$ and determine the exact dependency.

If $\propF{p}$ is the pro-$p$ completion of $F$ (which is a free pro-$p$ group), $\Comm(\propF{p})$ is a totally disconnected
locally compact (tdlc) group containing $\propF{p}$ as an open subgroup. We use $\Comm_p(F)$ to construct an abstractly
simple subgroup of $\Comm(\propF{p})$ containing $\propF{p}$ as well as a family of non-discrete tdlc groups which are compactly generated and simple. 
\end{abstract}

\maketitle

\setcounter{tocdepth}{1}
\tableofcontents

\section{Introduction} \label{sec-intro}

\subsection{Motivation and overview}
Let $\Gama$ be an abstract group. The commensurator of $\Gama$ is a natural variation of the automorphism group $\Aut(\Gama)$. Instead of automorphisms, one considers {\it virtual isomorphisms of $\Gama$}, that is, isomorphisms between finite index subgroups of $\Gama$.
While in general one cannot compose two virtual isomorphisms of $\Gama$, the problem goes away if we replace virtual isomorphisms by their equivalence classes where two virtual isomorphisms are equivalent if they coincide on a finite index subgroup.
The equivalence classes of virtual isomorphisms of $\Gama$ are called {\it commensurations of $\Gama$}. They form a group (with respect to composition) called the {\it commensurator of $\Gama$} and denoted by $\Comm(\Gama)$.

Historically, the groups $\Comm(\Gama)$ were first introduced in connection with relative commensurators. By definition, if $\Gama$ is a subgroup of a group $L$, the relative commensurator of $\Gama$ in $L$ is the subgroup $\Comm_{L}(\Gama)$ consisting of all $g \in L$ such that $\Gama$ and $g \Gama g^{-1}$ are commensurable, and in such situation there is a natural homomorphism $\Comm_{L}(\Gama)\to \Comm(\Gama)$. Relative commensurators play an important role in the study of discrete subgroups of Lie groups. A key theorem in that realm is Margulis arithmeticity criterion which asserts that if $\Gama$ is an irreducible lattice in a connected semi-simple Lie group $L$  with trivial center and no compact factor, then $\Gama$ is arithmetic if and only if $\Comm_{L}(\Gama)$ is a dense subgroup of $L$ \cite{Margulis-book}. 
\footnote{\me{If, in addition, $L$ is not isomorphic to $\PSL_2(\mathbb R)$, the natural map
$\Comm_{L}(\Gama)\to \Comm(\Gama)$ has finite index image by Mostow rigidity theorem.}}

If $G$ is a profinite group, one defines $\Comm(G)$ in the same way replacing finite index subgroups by open subgroups and requiring that virtual isomorphisms are continuous. Commensurators of profinite groups play a key role in understanding the connection between the structure of totally disconnected locally compact groups (tdlc hereafter) and their compact open subgroups (which are profinite). Indeed, if $L$ is a tdlc group, the compact open
 subgroups of $L$ form a base of neighborhoods of the identity, and for any such subgroup $G$ there is a natural homomorphism
$L\to \Comm(G)$ induced by conjugation. We refer the reader to subsection~\ref{sec:tdlcintro} for details. 

\skv

If $G$ is an abstract (resp. profinite) group, there is a natural homomorphism $\iotta: G \to \Comm(G)$ which sends
each $g\in G$ to the class of the corresponding inner automorphism. Its kernel is the set of all $g\in G$ which centralize a finite index (resp. open) subgroup of $G$.
There are many important cases where $\Comm(G)$ is equal to $\iotta(G)$ or at least contains $\iotta(G)$ as a subgroup of finite index. 
Examples \alb{include non-arithmetic lattices in connected simple Lie groups not locally isomorphic to $\SL_2(\dbR)$ \cite{Margulis-book},}  mapping class groups of surfaces \cite{Ivanov-MCG}, outer automorphism groups of free groups \cite{FarbHandel, Horbez-Wade} and automorphism groups of free groups \cite{Bridson-Wade}. 
\skv

In this paper we will focus on groups which are on the opposite end of the spectrum in terms of their commensurator size --
non-abelian free groups and free pro-$p$ groups of finite rank. If $\calF$ is one of those groups, the automorphism group $\Aut(\calF)$ is already much larger than $\calF$. And the commensurator $\Comm(\calF)$ is substantially larger than $\Aut(\calF)$ because it contains subgroups isomorphic to $\Aut(\calF')$ for free (abstract or pro-$p$) groups $\calF'$ of arbitrarily large rank. Further, $\calF$ contains plenty of isomorphic finite index subgroups (as any two subgroups of $\calF$ of the same finite index are isomorphic), and $\Comm(\calF)$ contains an isomorphism between any two such subgroups. One testament to the complexity of the structure of $\Comm(F)$
for a non-abelian free group $F$ is a theorem of Bering and Studenmund \cite{Bering-Studenmund} asserting that every countable locally finite group embeds into $\Comm(F)$.

\skv
It is well known that for any finitely generated residually finite group $\Gama$, its automorphism group $\Aut(\Gama)$ is residually finite. By contrast, there are classical situations where the commensurator is close to being simple. For instance, the latter happens if 
$\Gama=\mathbb Z^d$, in which case $\Comm(\Gama)\cong \GL_d(\dbQ)$. 
Other more involved examples are the groups \me{$\Gama=L_{\dbZ}$
where $L$ is a connected semi-simple group defined over $\dbQ$ and $L_\Z$ is its group of integer points.
Under mild assumptions on such $L$, the relative commensurator $\Comm_L(L_\Z)$ is equal to $L_\dbQ$ \cite{Borel66},  which is often close to being a simple group \cite{Tits64}.}
One of the main goals of this paper is to determine whether free groups enjoy a similar behavior.
More precisely, we will study the simplicity question for the commensurator $\Comm(\calF)$ as well as for certain subgroups of $\Comm(\calF)$, where as above  $\calF$ is either a free group or a free pro-$p$ group.

\subsection{Commensurators of free groups} \label{subsec-intro-Comm(F)}

Here and throughout the paper, we use the notation $F$ to denote a non-abelian free group of finite rank. Since all non-abelian free groups of finite rank are virtually isomorphic to each other, the isomorphism class of $\Comm(F)$ does not depend on the rank of $F$. 
To the best of our knowledge, the group $\Comm(F)$ was first explicitly considered in a 2010 paper of 
Bartholdi and Bogopolski~\cite{Barth-Bogo} where in particular it was proved that $\Comm(F)$ is not finitely generated.
However, there is another group related to $\Comm(F)$ whose study goes back to at least early 1990s -- the relative commensurator of $F$ inside the automorphism group of the Cayley tree $T$ of  $F$. The group $F$ sits as a uniform lattice in $\Aut(T)$, and its relative commensurator 
$\Comm_{\Aut(T)}(F)$ plays a key role in the general theory of tree lattices -- see \cite{BK}, \cite{LMZ} and \cite{Ca} for many important results and questions about this group. In particular, it is asked in \cite[Remarks~2.12(i)]{LMZ} whether a certain index two subgroup of $\Comm_{\Aut(T)}(F)$ is simple.

The question whether the full commensurator $\Comm(F)$ is simple appears as 
Problem~20.7.2 in \cite{Ca-Mo-Problems}, where it is attributed to Lubotzky. Some evidence towards a positive answer was obtained by
Caprace in \cite[Theorem A.2]{Ca}. A group $\Gama$ is called \textit{monolithic} if the intersection of all non-trivial normal subgroups of $\Gama$ is non-trivial; if this is the case, this intersection is called the monolith of $\Gama$ and denoted $\mon(\Gama)$. Caprace proved that $\Comm(\abs{F})$ is monolithic, $\mon(\Comm(\abs{F}))$ contains a finite index subgroup of $F$, and $\mon(\Comm(\abs{F}))$ is simple~\cite[Theorem A.2]{Ca}. 

Our first main theorem (Theorem~\ref{thmA} below) implies that while $\mon(\Comm(\abs{F}))$ is indeed a very large subgroup of $\Comm(\abs{F})$,
it is not the entire $\Comm(\abs{F})$ or even a finite index subgroup of it. In particular, the answer to 
the aforementioned question of Lubotzky is negative.

\begin{Definition}\rm Given an abstract group $G$, we define $\AComm(G)$ to be the subgroup of $\Comm(G)$ generated by 
the canonical images of the groups $\Aut(H)$ where $H$ ranges over all finite index subgroups of $G$. 
\end{Definition}

As we will see, the subgroup $\AComm(G)$ is normal in $\Comm(G)$ under very mild assumptions on $G$ (see Lemma~\ref{lemma:Aut_1}(b)) -- in particular, this is the case  if $G$ is finitely generated. 
Further, under some natural conditions on $G$ which hold, for instance, if $G$ is free,
the canonical map $\Aut(H)\to\Comm(G)$ is injective for any finite index subgroup $H$ (see Lemmas~\ref{lem-BEW}~and~\ref{lem-unique-root}). In this
case we will identify $\Aut(H)$ with its image in $\Comm(G)$.
\skv

We are now ready to state our first main theorem.

\begin{AbcTheorem}
\label{thmA}
Let $F$ be a non-abelian free group of finite rank. The following hold: 

\begin{enumerate}	
\item \label{item-AComm-simple} $\AComm(F)$ is simple and is equal to the monolith of $\Comm(F)$.
	\item \label{item-AComm-inf-ind} $\AComm(F)$ has infinite index in $\Comm(F)$. In particular, $\Comm(F)$ is not virtually simple. 
\end{enumerate} 

\end{AbcTheorem}

Part (\ref{item-AComm-simple}) strengthens the result of Caprace stated above. 
Our proof of (\ref{item-AComm-simple}) does not rely on \cite[Theorem A.2]{Ca} and uses a completely different approach based on the free factor theorem of M.~Hall. 
We will deduce part (\ref{item-AComm-inf-ind}) from a more general result applicable to other
classes of groups, including non-abelian surface groups -- see Theorems~\ref{thm:nonsimple}~and~\ref{thm:CommACommquotient}.

\subsection{New finitely generated simple groups}

As a general rule, constructing simple groups is a more challenging task if the groups are required to be finitely generated. In this paper 
we will use commensurators to exhibit a new family of finitely generated infinite simple groups. The construction is entirely elementary.

\begin{Definition} \label{defi-An-intro}
	Let $F$ be a free group and $ m\geq 2$. Let $A_{m}(F)$ be the subgroup of $\Comm(F)$ generated by the subgroups $\Aut(H)$ where $H$ ranges over normal subgroups of $F$ such that $F/H$ is cyclic of order $m$. 
\end{Definition}

By definition $A_{m}(F)$ is a subgroup of $\AComm(F)$. By a classical theorem of Nielsen, the automorphism group of a finitely generated free group is finitely generated. Since $F$ has only a finite number of subgroups of a given index, it follows that $A_m(F)$ is a finitely generated group. In Section~\ref{sec-finite-generated} we will define a certain subgroup $S_m(F)$ of $A_m(F)$ which contains $F$ and has index at most $2$ (so $S_m(F)$ is still finitely generated).

\begin{AbcTheorem} \label{thm-intro-fg-simple}
	Let $m \geq 2$, and let $F_d$ be a free group of rank $d \geq 2$. Then there exist infinitely many values of $d$ such that the finitely generated group $S_{m}(F_d)$ is simple.
\end{AbcTheorem}

We refer the reader to Theorem~\ref{thm-fg-simple-general} for a more precise statement. The restriction on the rank in Theorem \ref{thm-intro-fg-simple} appears to be a technical limitation, and we believe that the statement is possibly true for every $m \geq 2$ and every $d \geq 2$ -- see Remark \ref{remark-all-ranks}. 

 As we already mentioned, the group $\Comm(F)$ is not finitely generated \cite{Barth-Bogo}. The subgroup $\AComm(F)$ is not finitely generated either \cite[Theorem A.2]{Ca}. The first examples of infinite finitely generated simple groups embeddable in $\Comm(F)$ have been constructed by Burger~and~Mozes \cite{BM00b}. These groups act properly and cocompactly on the product of two trees, in such a way that the action on each factor is not proper. Equivalently, they are cocompact irreducible lattices in the automorphism group of the product of two trees. We refer to the recent survey \cite{Caprace-survey} for an extensive discussion on these groups. To the best of our knowledge, prior to this work, irreducible lattices in products of trees was the only source of examples of finitely generated infinite simple subgroups of $\Comm(F)$. The groups $S_m(F)$ are of very different nature than irreducible lattices in product of trees. For instance, they cannot act properly and cocompactly by isometries on a complete CAT(0) space -- see Remark  \ref{remark-not-CAT(0)}. 

\subsection{On the $p$-commensurator of free groups}

As before, let $F$ be a non-abelian free group of finite rank. If $p$ is a prime number, the pro-$p$ topology on $F$ is the topology for which a base of open neighborhoods of an element $x \in F$ is given by cosets of the form $xN$ where $N$ ranges over the normal subgroups of $F$ of finite $p$-power index. We will say that a subgroup of $F$ is $p$-open if it is open in the pro-$p$ topology. The $p$-commensurator of $F$, denoted $\Comm_p(F)$, is defined similarly to the commensurator of $F$, except that finite index subgroups are replaced by $p$-open subgroups 
(see Definition~\ref{def-Comm-Commp}). There is a natural  homomorphism $\Comm_p(F) \to \Comm(F)$, and this homomorphism is injective 
(Lemma~\ref{lem-germ-compare-topologies}). It will often be convenient to identify $\Comm_p(F)$ with its image in $\Comm(F)$. 

The group $\Comm_p(F)$ was considered by
Bou-Rabee and Young in \cite{BouRabee-Young} in connection with the study of representations of Baumslag--Solitar groups in $\Comm(F)$. The paper~\cite{BouRabee-Young} also posed several interesting questions very relevant to our work.

\skv

We will establish several results showing that in many ways $\Comm_p(F)$ behaves differently from $\Comm(F)$ and that $\Comm_p(F)$ may be easier to understand than $\Comm(F)$. First, similarly to the case of $\Comm(F)$, we can consider the subgroup of $\Comm_p(F)$ generated by all subgroups $\Aut(U)$ where $U$ now ranges over all $p$-open subgroups of $F$. But in contrast with Theorem~\ref{thmA}(\ref{item-AComm-inf-ind}), we will show that this subgroup is actually equal to $\Comm_p(F)$ -- see Proposition \ref{prop-Cp-generation}. 
Morever, we will show that the $p$-commensurator $\Comm_p(F)$ is virtually simple:

\begin{AbcTheorem}
	\label{ThmC} Let $F$ be a non-abelian free group of finite rank. Then $\Comm_p(F)$ contains a simple subgroup of index at most two.
\end{AbcTheorem}

We refer the reader to Theorem~\ref{thm-simplicity-Commp} and discussion preceding it 
for a more general statement and the definition of the simple group from Theorem~\ref{ThmC}.

\skv
The second key difference that we establish between $\Comm(F)$ and $\Comm_p(F)$ is the dependency on the rank of the free group. As mentioned earlier, the isomorphism class of $\Comm(F)$ does not depend on the rank of $F$ since finitely generated non-abelian free groups are all virtually isomorphic to each other. The situation for the $p$-commensurator appears to be more subtle. It is a consequence of  the Nielsen-Schreier formula that two free groups $F_k$ and $F_\ell$ of ranks $k,\ell \geq 2$ sit as $p$-open subgroups in a common free group $F$ if and only if $(k-1) / (\ell -1) = p^s$ for some $s \in \Z$. When this condition holds, it is not hard to see that the embeddings of $F_k$ and $F_\ell$ as $p$-open subgroups of $F$ induce isomorphisms $\Comm_p(F_k) \to \Comm_p(F)$ and $\Comm_p(F_\ell) \to \Comm_p(F)$, so that the groups $\Comm_p(F_k)$ and $\Comm_p(F_\ell)$ are isomorphic. Theorem~\ref{thm-intro-depend-parameters} below asserts that

\begin{itemize}
\item[(i)] distinct primes always give rise to non-isomorphic groups $\Comm_p(F)$ and 
\item[(ii)] that when $p$ is fixed, the above condition on $k$ and $\ell$ actually characterizes when $\Comm_p(F_k)$ and $\Comm_p(F_\ell)$ are isomorphic. 
\end{itemize}
This theorem solves a question raised in \cite{BouRabee-Young}.

\begin{AbcTheorem} 
\label{thm-intro-depend-parameters}
	Let $p,q$ be prime numbers, let $k,\ell \geq 2$, and let $F_k$ and $F_{\ell}$ be free groups of ranks $k$ and $\ell$,
respectively. Then the groups $\Comm_p(F_k)$ and $\Comm_q(F_\ell)$ are isomorphic if and only if $p=q$ and there exists $s \in \dbZ$ such that $\frac{k-1}{\ell -1} = p^s$. 
\end{AbcTheorem}

Our strategy for proving Theorem  \ref{thm-intro-depend-parameters} also provides information on automorphisms of the $p$-commensurator. We prove the following: 

\begin{AbcTheorem}\label{thm-intro-trivial_outer}
Let $F$ be a non-abelian free group of finite rank. Then the group $\Comm_p(F)$ has trivial outer automorphism group.
\end{AbcTheorem}

\subsection{Commensurators of free pro-$p$ groups}

Again let $p$ be a prime number, and let $\propF{p}$ be a non-abelian free pro-$p$ group of finite rank. 
If $F$ is a free group of the same rank, we can consider $\propF{p}$ 
as the pro-$p$ completion of $F$ \cite[Proposition~3.3.6]{RZ-book}. 
Since $F$ is residually-$p$, the canonical map $F \to \propF{p}$ is injective, so we can view $F$ as a subgroup of $\propF{p}$.
There is also a natural bijection between $p$-open subgroups of $F$ and open subgroups of $\propF{p}$, which maps a $p$-open subgroup of $F$ to its closure in $\propF{p}$ (the inverse map is given by the intersection with $F$). Using this bijection, it is routine to extend
the embedding $F \to \propF{p}$ to an injective  homomorphism $\Comm_p(F)\to \Comm(\propF{p})$, and thus we can also
view $\Comm_p(F)$ as a subgroup of $\Comm(\propF{p})$. We note that however there is no natural homomorphism $\Comm(F)\to \Comm(\propF{p})$.

It is natural to expect some similarities between $\Comm_p(F)$ and $\Comm(\propF{p})$, and as the first indication of this,  
we will show that $\Comm(\propF{p})$ is generated by the subgroups $\Aut(\propU{p})$ where $\propU{p}$ ranges over open subgroups of $\propF{p}$ -- see Proposition \ref{Comm=AComm}.  
We will also build on Theorem~\ref{ThmC} to construct
an analogous abstractly simple subgroup of $\Comm(\propF{p})$ --
see Theorem~\ref{Theorem-intro-pro-p} below.

The group $\Comm(\propF{p})$ admits a  tdlc group topology, for which $\propF{p}$ is a compact open subgroup of $\Comm(\propF{p})$, and such that the induced topology on $\propF{p}$ is the original topology on $\propF{p}$. 
 Let $\CpC_p(\absF{})$ be the closure of $ \Comm_p(\absF{})$  in $\Comm(\propF{p})$. Since $\Comm_p(\absF{})$ contains $F$ and
$F$ is dense in $\propF{p}$,  the group $\CpC_p(F)$ is equal to the subgroup of $ \Comm_p(\absF{})$ generated by $\propF{p}$  and $\Comm_p(F)$. In particular, $\CpC_p(F)$ is open in $\Comm(\propF{p})$.

\begin{AbcTheorem}\label{Theorem-intro-pro-p}
	Let $F$ be a non-abelian free group of finite rank and let $\propF{p}$ be its pro-$p$
	completion. Then the open subgroup $\CpC_p(F)$ of $\Comm(\propF{p})$ has a subgroup of index at most two that is abstractly simple.
\end{AbcTheorem}

Similarly to the case of the $p$-commensurator, we will prove a slightly stronger statement -- see Theorem~\ref{thm-completion-abs-simple}.
\skv

We note that the group $\CpC_p(F)$ is much smaller than the entire commensurator $\Comm(\propF{p})$. Indeed, since
$\Comm_p(F)$ is a countable group and $\Comm_p(F)$ commensurates $\propF{p}$, the subgroup $\propF{p}$ has countable index in $\CpC_p(\absF{})=\la \propF{p}, \Comm_p(F)\ra$, \me{so in particular, $\CpC_p(\absF{})$ is second countable.}
On the other hand, already the index $[\Aut(\propF{p}):\propF{p}]$ is uncountable.

The question whether $\Comm(\propF{p})$ is abstractly simple or at least topologically simple (with respect to the above tdlc topology) is mentioned in \cite[p. 349]{Ca-Mo-Problems}. We were unable to answer either question. Nevertheless, we will show that $\Comm(\propF{p})$ is monolithic and its monolith is abstractly simple:

\begin{AbcTheorem}\label{Theorem-intro-Comm-pro-p}
Let $\propF{p}$ be a non-abelian free pro-$p$ group of finite rank. Then the group $\Comm(\propF{p})$ is monolithic, $\mon(\Comm(\propF{p}))$ contains $\propF{p}$, and $\mon(\Comm(\propF{p}))$ is abstractly simple.
\end{AbcTheorem}

Unlike the case of free (abstract) groups, we do not have a very transparent description of $\mon(\Comm(\propF{p}))$, but we will 
obtain additional information about $\mon(\Comm(\propF{p}))$, beyond the statement of Theorem~\ref{Theorem-intro-Comm-pro-p} -- 
see Proposition~\ref{Comm-prop-contains-A-closure}. In particular, we will prove that the index  
$[\mon(\Comm(\propF{p})): \propF{p}]$ is uncountable. \me{Since $\propF{p}$ is open in  $\mon(\Comm(\propF{p}))$,
it follows that $\mon(\Comm(\propF{p}))$ is not $\sigma$-compact and (being locally compact) not second countable}.

\subsection{New compactly generated simple tdlc groups and embedding free pro-$p$ groups into such groups}
\label{sec:tdlcintro}

A prominent theme in the theory of tdlc groups is understanding the relationship between the global properties of a tdlc group $L$ and its local properties, that is, properties of its compact open subgroups -- for some important results see \cite{BM00a,Willis-co-S,BEW,CapDeM,Cap-Stul,CRW-Part2}. 
The connection between these global and local properties is particularly tight under the extra hypotheses that $L$ is compactly generated and topologically simple -- see for instance \cite[Theorem 4.8]{BEW}, \cite[Theorem 5.3]{CRW-Part2}, \cite[Theorem J]{CRW-Part2}, \cite[Theorem F]{Cap-LB-19}. While the topological simplicity assumption can very often be relaxed, the compact generation assumption is usually crucial. 

The following question originally arose from \cite{BEW} and is recorded in \cite[Problem~20.5.2]{Ca-Mo-Problems} and part 2 of \cite[Problem~19.10]{Kou19}.

\begin{Question}
\label{qCM}
Let $\propF{p}$ be a non-abelian free pro-$p$ group of finite rank. Does there exist a compactly generated topologically simple group $L$ which contains $\propF{p}$ as an open subgroup? 
\end{Question}

As we already explained, for any group $L$ answering Question~\ref{qCM} there is a homomorphism $\iota_L: L\to \Comm(\propF{p})$
whose image contains $\propF{p}$ (so in particular is open); further, since $\propF{p}$ has trivial quasi-center and $L$ is topologically simple, $\iota_L$ must be an embedding. Thus, when attempting to solve Question~\ref{qCM} in the positive, there is no loss of generality in restricting to (compactly generated and topologically simple) subgroups $L$ of $\Comm(\propF{p})$ containing $\propF{p}$.

Let us now make a small digression and
recall that for any $m\geq 2$,
the finitely generated group $A_m(F)$ of $\Comm(F)$ from Definition~\ref{defi-An-intro} has a subgroup $S_m(F)$ of index at most $2$ which is simple for some values of $\rk(F)$ by Theorem \ref{thm-intro-fg-simple}. When $m=p$ is prime, the group $A_p(F)$ is actually contained
in $\Comm_p(F)$. So there exists a finitely generated simple subgroup of $\Comm_p(F)$ containing $F$. This result can be viewed as a positive answer to an analogue of Question~\ref{qCM} for free groups.

From this perspective, the closure $L$ of $S_p(F)$ in $\Comm(\propF{p})$, which is also the subgroup of $\CpC_p(F)$ generated by $\propF{p}$ and $S_p(F)$, appears as a natural candidate for solving Question~\ref{qCM}. We elaborate on this idea in Section \ref{sec-compact-generated} and consider a certain ascending sequence of compactly generated open subgroups $(L_n)_{n\geq 1}$ of $\CpC_p(F)$ starting with $L_1=L$. We did not manage to determine whether $L_n$ is simple for all $n$ or at least for all sufficiently large $n$, but we will show that the groups $L_n$
are at least not too far from being simple.

\begin{AbcTheorem}\label{thm:maintdlc}
Let $\absF{}$ be a free group of rank $d\ge 3$ and let $\propF{p}$ be its pro-$p$ completion.  Then there is an ascending sequence $(L_n)_{n \ge 1}$ of compactly generated open subgroups of $\CpC_p(\absF{})$ containing $\propF{p}$ 
with the following properties.
 \begin{itemize}
 \item[(i)] The union $\cup_{n\in\mathbb N} L_n$ has index at most two in $\CpC_p(F)$. 
\item[(ii)] Let $K_n$ denote the normal core of $\propF{p}$ in $L_n$. Then the sequence $(K_n)_{n \ge 1}$ is descending and has trivial intersection.
 \item[(iii)] The group $L_n/K_n$ admits a non-discrete abstractly simple quotient $Q_n$ which has an open subgroup isomorphic to 
$F(n)=  \propF{p}/K_n$. Moreover, the kernel of the projection $L_n/K_n\to Q_n$ is discrete.
 \item[(iv)] The pro-$p$ group $F(n) =  \propF{p}/K_n$ has the following features in common with non-abelian free pro-$p$ groups:
         \begin{itemize}
           \item every $d$-generator $p$-group of order at most $p^{n+1}$ occurs as a quotient of $F(n)$;
           \item for all $1 \le j \le n+1$ any two subgroups of $F(n)$ of index $p^j$ are isomorphic.
         \end{itemize}
    \item[(v)] For each $n$ the following are equivalent:
\begin{itemize}
\item $L_n$ is abstractly simple;
\item  $K_n=\{1\}$; 
\item $F(n)$ is free pro-$p$.
\end{itemize}
\item[(vi)] If $\CpC_p(\absF{})$ contains any  topologically simple compactly generated open subgroup, then $L_n$ is abstractly simple for all sufficiently large $n$.
\end{itemize}
\end{AbcTheorem}

We refer the reader to Section \ref{sec-compact-generated} for details. We do not know if the group $F(n)$ is free pro-$p$ for sufficiently large $n$. Of course, if the latter is the case, Question~\ref{qCM} would be solved in the positive. Regardless of whether $F(n)$ is free pro-$p$ or not, we believe that the groups $(Q_n)$ yield a new family of compactly generated simple tdlc groups.

\vskip .2cm
{\bf Acknowledgments}. We are grateful to Pierre-Emmanuel Caprace for illuminating discussions and providing useful references, and to Alex Lubotzky for telling us about the history of the notion of commensurator. \me{We are also grateful to the anonymous referees for carefully reading the manuscript and providing many useful suggestions.}

M.V.\ is funded by the Italian program Rita Levi Montalcini for young researchers (Edition 2021) and he was supported by the Spanish Government grant PID2020-117281GB-I00, partly by the European Regional Development Fund (ERDF), the MICIU /AEI /10.13039/501100011033 / UE grant PCI2024-155053-2. C.R.\ was supported by the welcome program of the ENS de Lyon Research Fund.

\bigskip

\paragraph{\bf Notations.} 
Throughout the paper we will adopt the following notational convention. First, $p$ will denote a fixed prime. To make the notations transparent, we will always denote free groups by capital Latin letters in a Roman font  and free pro-$p$ groups by capital Latin letters in a boldface font, with $F$ and $\propF{p}$ being the default choices. Occasionally we will make statements which covers both free pro-$p$ and free groups, and in that case we will use letters in a calligraphic font (usually $\apf{F}$).

\section{Basic results on commensurators} \label{sec-prelim-general}

\subsection{The commensurator of a topological group}\label{subsec-profinite-commens}

We start by defining the commensurator $\Comm(G)$ for a topological group $G$. 

\begin{Definition} \label{defin-local-iso}
Let $G$ be a topological group.
	\begin{itemize}
		\item A {\it virtual isomorphism of $G$} is a topological group isomorphism $f : U \to V $ where $U,V$ are both finite index open subgroups of $G$. We also say that the triple $(f,U,V)$ is a virtual isomorphism.
		\item Given two virtual isomorphisms $f_i : U_i \to V_i $, $i=1,2$, their composition is the virtual isomorphism
		\[
		f_1 \circ f_2: f^{-1}_2(V_2 \cap U_1) \to f_1(V_2 \cap U_1 )\mbox{ defined by } x \mapsto f_1(f_2(x)).
		\]
		\item Two virtual isomorphisms are declared equivalent if they coincide on some finite index open subgroup. 
		\item The equivalence class of a virtual isomorphism $f$ will be denoted by $[f]$.
	\end{itemize}
\end{Definition}

Virtual isomorphisms modulo this equivalence form a group with operation defined by $[f_1][f_2] = [f_1 \circ f_2]$. This group is called
the {\it commensurator of $G$} and will be denoted by $\Lg(G)$.  

If $H$ is a finite index open subgroup of $G$, then every virtual isomorphism of $H$ is a virtual isomorphism of $G$, and the induced homomorphism $\Lg(H) \to \Lg(G)$ is an isomorphism, so we can identify $\Lg(H)$ with $\Lg(G)$. More generally, if two topological groups $G_1$ and $G_2$ are \emph{virtually isomorphic}\footnote{We warn the reader that what we call 'virtually isomorphic' is sometimes called 'abstractly commensurable' or just 'commensurable'. Here we will use the word `commensurable' only in the situations when both groups are subgroups of the same group -- see Definition~\ref{def-commensurated}.}, 
meaning that they admit topologically isomorphic finite index open subgroups, then $\Lg(G_1)$ and $\Lg(G_2)$ are isomorphic.

\begin{Remark} \label{rem:openfiniteindex}
In this paper we will only consider the groups $\Comm(G)$ in the case where all open subgroups of $G$ are of finite index.
\end{Remark}

We will use the following terminology. 
	
\begin{Definition}\label{def:quasivirtualcenter}
$\empty$

\begin{itemize}
\item[(a)] Let $\Gama$ be an abstract group. The {\it virtual center} $\VZ(\Gama)$ is the set of elements $g\in \Gama$ which centralize a finite index subgroup of $G$ \footnote{Also called FC-center. Here we adopt the terminology `virtual center' as we think of this set as the elements that act trivially by conjugation on a finite index subgroup, rather than the elements whose conjugacy class is finite.
}. 
\item[(b)] Now let $G$ be a topological group. The {\it quasi-center} $\QZ(G)$ is the set of elements $g \in G$ which centralize an open subgroup of $G$.
\end{itemize}
\end{Definition}
Clearly, $\VZ(\Gama)$ is a characteristic subgroup of an abstract group $\Gama$. Likewise $\QZ(G)$ is a topologically characteristic subgroup of a topological group $G$ (which need not be closed). For a general topological group $G$ neither of the subgroups $\VZ(G)$ and $\QZ(G)$ must contain the other, but for instance $\QZ(G)=\VZ(G)$ if either $G$ is an abstract group with profinite topology or a profinite group. 

\skv
Let $G$ be a topological group. Given $g \in G$, we denote by $\iotta_g: G \rightarrow G$ the left conjugation by $g$, that is,
$\iotta_g(x)=gxg^{-1}$. The map $\iotta: G \to \Lg(G)$ given by $\iotta(g)= [\iotta_g]$ is a group homomorphism. In general, the kernel
of $\iotta$ is $\VZ(G)\cap \QZ(G)$, the intersection of the virtual center and the quasi-center. However, we will only consider
$\Lg(G)$ when every open subgroup of $G$ has finite index in $G$, in which case $\QZ(G)\subseteq \VZ(G)$ and therefore
$$\Ker(\iotta)=\QZ(G).$$
For the rest of the paper we make the standing assumption that when $\QZ(G)$ is trivial, we will 
usually identify $G$ with $\iotta(G)$ and thus view $G$ as a subgroup of its commensurator $\Lg(G)$. 

\begin{Definition} \label{def-commensurated}
		Two subgroups $H,K$ of a group $G$ will be called  \textit{commensurable} if their intersection $H \cap K$ has finite index in both $H$ and $K$. A subgroup $H$ of $G$ is  \textit{commensurated} in $G$ if any two $G$-conjugates of $H$ are commensurable. 
\end{Definition}

We will be frequently using this terminology when the group $G$ is equipped with a topology. We stress that in such situations we are not requiring any extra hypotheses on $H,K$ and $H \cap K$ with respect to the ambient topology on $G$.

\begin{Lemma} \label{lem-G-commens-CommG}
Let $G$ be a topological group and $f: U \to V$ a virtual isomorphism of $G$. 
Then we have $[f] \iotta(U) [f]^{-1} = \iotta(V)$. In particular, $i(G)$ is commensurated in $\Lg(G)$.
\end{Lemma}

\begin{proof}
The equality $[f] \iotta(U) [f]^{-1} = \iotta(V)$ holds since for every $u\in U$ we have $f\iotta_u f^{-1}=\iotta_{f(u)}$ as maps
from $V$ to $V$.
That $\iotta(G)$ is commensurated follows because $\iotta(U)$ and $\iotta(V)$ both have finite index in $\iotta(G)$. 
\end{proof}

There is a natural way to equip $\Lg(G)$ with a group topology, which relies on the following classical principle (see e.g. \cite[III.2 Proposition 1]{Bourbaki_III_1-4}):

\begin{Proposition}
	\label{prop:topbase}
	Let $\Gama$ be a group and  $\mathcal{O}$ a collection of subgroups of $\Gama$ such that
	\begin{enumerate}
		\item any finite intersection of elements of $\mathcal{O}$ contains an element of $\mathcal{O}$;
		\item for every $g\in \Gama$ and $O_1\in \mathcal{O}$, there is $O_2\in \mathcal{O}$ such that $O_2 \subseteq gO_1g^{-1}$. 
	\end{enumerate}
	Then there is a unique group topology $\tau$ on $\Gama$ such that $\mathcal{O}$ is a base of neighborhoods of the identity for $\tau$. Moreover, $\tau$ is Hausdorff if and only if the intersection of all subgroups in $\mathcal O$ is trivial.
\end{Proposition}

\begin{Proposition}\label{prop-topo-germs}
	Let $G$ be a topological group.  Then $\Lg(G)$ admits a unique group topology $\tau$ such that there is a base $\mathcal{O}$ of neighborhoods of the identity formed by subgroups of the form $\iotta(U)$, for $U$ a finite index open subgroup of $G$.  With respect to $\tau$, the homomorphism $\iotta: G \to \Lg(G)$ is continuous and open.
\end{Proposition}

\begin{proof}
To establish the existence and uniqueness of $\tau$ it is enough to check that $\mathcal{O}$ has property (2) from Proposition~\ref{prop:topbase}. Given $g \in \Lg(G)$ and an open subgroup $W$ in $G$, we want to show that $g\iotta(W)g^{-1}$ contains some element of $\mathcal{O}$. Choose a virtual isomorphism $f: U \to V$ of $G$ such that $g = [f]$ and $U \subseteq W$. 
Then $g\,\iotta(W)g^{-1} \supseteq g\,\iotta(U)g^{-1} = \iotta(V)$ where the equality holds by Lemma \ref{lem-G-commens-CommG}.  
\end{proof}

 For the rest of the paper we always consider $\Lg(G)$  being equipped with the topology from Proposition~\ref{prop-topo-germs}, unless explicitly mentioned otherwise.\footnote{This topology is called the strong topology in \cite{BEW}, but we will not be using that terminology.}
\skv

The commensurator $\Lg(G)$ satisfies the following universal property. For simplicity, we state it in the special case where $G$ has trivial quasi-center. Recall that in this case the map $\iotta:G\to \Lg(G)$ is injective and we identify $\iotta(G)$ with $G$.

\begin{Proposition}[Universal property of $\Lg(G)$] \label{prop-univ-Comm}
Let $G$ be a topological group with trivial quasi-center. Then $G$ is an open commensurated subgroup of $\Lg(G)$. Conversely, if $L$ is a topological group into which $G$ embeds as an open commensurated subgroup, then there is a  homomorphism $\psi: L \to \Lg(G)$. Moreover, $\psi$ is  continuous, $\psi$ has open image, and $\ker(\psi) = \QZ(L)$, which is a discrete normal subgroup of $L$. 
\me
{\begin{equation}
\label{eq:univproperty}
\xymatrix{
G\ar[r]\ar[rd]^{\iotta}&L\ar@{-->}[d]^{\psi}\\
&\Comm(G)
}
\end{equation}
}
\end{Proposition}

\begin{proof}
The subgroup $G$ is open in $\Lg(G)$ by definition and is commensurated by Lemma~\ref{lem-G-commens-CommG}.
Now suppose $L$ is a topological group into which $G$  embeds as an open commensurated subgroup. For  every $h \in L$, the subgroups $h^{-1}Gh \cap G$ and $G \cap hGh^{-1}$ are finite index open subgroups of $G$, and the conjugation by $h$ induces a virtual isomorphism $h^{-1}Gh \cap G \to G \cap hGh^{-1}$ of $G$. This defines a homomorphism $\psi : L \to \Lg(G)$, whose restriction to $G$ is equal to $\iotta: G \to \Lg(G)$. Since $G$ is open in $L$, $\psi $ is continuous, and  $\psi(L)$ is open since it contains $i(G)$. An element of $h \in L$ belongs to $\ker(\psi)$ if and only if there exists an open subgroup of $G$ with which $h$ commutes, which precisely means that $h \in \QZ(L)$. Since $\QZ(L) \cap G = \QZ(G) = 1$, the group $\QZ(L) $ is necessarily discrete. 
\end{proof}

Given a finite index open subgroup $U$ of a topological group $G$, any automorphism of $U$ can be viewed as a virtual isomorphism of $G$.  Thus we have a natural homomorphism $\iota_U: \Aut(U) \rightarrow \Lg(G)$; we denote its image by $\uAut(U)$.

\begin{Definition}
We will say that $G$ is {\it $\Aut$-$\Comm$ injective} if the map $\iota_U$ is injective for every open subgroup $U$ of $G$.
\end{Definition}
  
If $G$ is $\Aut$-$\Comm$ injective, we will simplify notation again and write $\Aut(U)$ instead of $\uAut(U)$, thereby
identifying $\Aut(U)$ with a subgroup of $\Lg(G)$.
The following two simple observations provide two natural conditions on $G$, either of which implies that $G$ is $\Aut$-$\Comm$ injective
and in fact has a slightly stronger property.

\begin{Lemma}
	\label{lem-BEW}
	Let $G$ be a topological group with  trivial quasi-center. Suppose $f: U \to V$ is an isomorphism between two open subgroups of $G$ such that $f$ is the identity on some open subgroup of $G$. Then $U=V$ and $f$ is the identity map. In particular, $G$ is $\Aut$-$\Comm$ injective.
\end{Lemma}

\begin{proof}
When $G$ is profinite, this was proved in  \cite[Proposition~2.6]{BEW}. The proof for general topological groups is identical.
\end{proof}

\begin{Lemma} \label{lem-unique-root}
	Let $G$ be a topological group with the unique root property (if $g^n=h^n$ for some $g,h\in G$ and $n \geq 1$, then $g=h$). Suppose $f: U \to V$ is an isomorphism between two finite index open subgroups of $G$ such that $f$ is the identity on some open subgroup of $G$. Then $U=V$ and $f$ is the identity map. In particular, $G$ is $\Aut$-$\Comm$ injective.
\end{Lemma}

\begin{proof}
This was proved for abstract groups in  \cite{Barth-Bogo}. The same proof applies to topological groups without any changes.
\end{proof}

\begin{Remark}\label{rem:valueofvirtulauto}
If $G$ \me{is $\Aut$-$\Comm$ injective}, the following more general property automatically holds: if $f: U \to V$ is a virtual isomorphism of $G$ and $f': U' \to V'$ is a virtual isomorphism of $G$ that is equivalent to $f$, then the value of $f'$ at any point is determined by $f$; in particular, $f(g) = f'(g)$ for all $g \in U \cap U'$.
\end{Remark}

Our next lemma provides some useful information on the subgroup structure of $\Lg(G)$ in the case where $G$ has trivial quasi-center. Of particular importance to us is part (\ref{item-Comm-local-normal}) which will be the starting observation of many of our simplicity theorems.

\begin{Lemma} \label{lem-Comm-trivialQZ}
	Let $G$ be a topological group with trivial quasi-center. The following hold: 
\begin{enumerate}
		\item  \label{item-Comm-trivial-QZ} Any open subgroup of $\Lg(G)$ has trivial  quasi-center.
		\item \label{item-Comm-local-normal} Let $N$ be a non-trivial subgroup  of $\Lg(G)$ whose normalizer in $\Lg(G)$ is open. 
Then $N \cap G$ is non-trivial. 
	\end{enumerate} 
\end{Lemma}

\begin{proof}
(\ref{item-Comm-trivial-QZ}) The entire $\Lg(G)$ has trivial quasi-center by \cite[Proposition 3.2(c)]{BEW}. On the other hand,
if $U$ is an open subgroup of a topological group $L$, then $\QZ(U)=\QZ(L)\cap U$, so (\ref{item-Comm-trivial-QZ}) follows.
\skv

 (\ref{item-Comm-local-normal}) In this proof, to avoid confusion, we will not identify $G$ with $\iotta(G)$, so our goal is to show
 that $N\cap \iotta(G)$ is non-trivial. Let $[\psi]$ be a non-trivial element of $N$. By our hypotheses there exists an open subgroup $U$ of $G$ such that
$\psi$ is defined on $U$ and $\iotta(U)$ normalizes $N$. Since $\QZ(G)=\{1\}$, the restriction of $\psi$ to $U$ is non-trivial
by Lemma~\ref{lem-BEW}, so we can find $g\in U$ such that $\psi(g)\neq g$.
 Then $\iotta_{\psi(g)g^{-1}}=\iotta_{\psi(g)}\iotta_g^{-1}=\psi \iotta_g \psi^{-1}\iotta_g^{-1}=\psi\cdot (\iotta_g \psi^{-1}\iotta_g^{-1})$ is a non-trivial element which lies in both $N$ and $\iotta(G)$.
\end{proof}

Now suppose that $\tau_1$ and $\tau_2$ are two group topologies on $G$ such that $\tau_1 \subseteq \tau_2$. Every virtual isomorphism $f$ of
the topological group $(G,\tau_1)$ is also a virtual isomorphism of $(G,\tau_2)$, and $f_1 \sim f_2$ for $\tau_1$ 
implies $f_1 \sim f_2$ for $\tau_2$. Thus there is an induced map $\Lg((G,\tau_1)) \to \Lg((G,\tau_2))$, which is clearly a  homomorphism.

\begin{Lemma} \label{lem-germ-compare-topologies}
Suppose that $(G,\tau_2)$ has trivial quasi-center. Then the homomorphism $\Lg((G,\tau_1)) \to \Lg((G,\tau_2))$ is injective.
\end{Lemma}

\begin{proof}
	If $f : U  \to V$ is a virtual isomorphism of $(G,\tau_1)$ whose image in $\Lg((G,\tau_2))$ is trivial, then Lemma \ref{lem-BEW} asserts 
that $U=V$ and $f$ is the identity map, so $f$ has trivial image in $\Lg((G,\tau_1))$ as well. 
\end{proof}

\subsection{The commensurator and the $p$-commensurator of an abstract group}
Let $\Gama$ be an abstract group. We start by recalling the definitions of the profinite and pro-$p$ topologies on $\Gama$. 
Each of these turns $\Gama$ into a topological group. 

\begin{itemize}
\item The profinite topology $\tau_{\prof}$ on $\Gama$  is defined by declaring that the finite index normal subgroups of $\Gama$ form a base of neighborhoods of the identity. Since every finite index subgroup contains a finite index normal subgroup, the open subgroups
in the profinite topology are precisely the finite index subgroups.

\item If $p$  is a prime number, the  pro-$p$ topology $\tau_p$ on $\Gama$ is defined by declaring that the normal subgroups of $\Gama$ of finite $p$-power index form a base of neighborhoods of the identity. We say that a subgroup of $\Gama$ is {\it $p$-open} or {\it $p$-closed} if it is open or closed in the pro-$p$ topology, respectively. It is well known that $p$-open subgroups are precisely subnormal subgroups of $p$-power index.
\end{itemize}

\begin{Definition}  \label{def-Comm-Commp}
	Let $\Gama$ be an abstract group.
\begin{itemize}
\item The {\it commensurator} of $\Gama$ denoted by $\Comm(\Gama)$ is the commensurator of the topological group $(\Gama, \tau_{prof})$, that is, $\Comm(\Gama)= \Lg((\Gama,\tau_{\prof}))$. 
\item The {\it $p$-commensurator} of $\Gama$ denoted by $\Comm_p(\Gama)$ is the commensurator of the topological group $(\Gama, \tau_{p})$, that is,  $\Comm_p(\Gama)= \Lg((\Gama,\tau_p))$.
\end{itemize}
\end{Definition}

	More explicitly, \begin{itemize}
		\item 	elements of $\Comm(\Gama)$ are  isomorphisms between finite index subgroups of $\Gama$,  modulo being equal on a finite index subgroup. The kernel of $i: \Gama \to \Comm(\Gama)$ is $\QZ((\Gama,\tau_{prof}))=\VZ(\Gama)$, the virtual center of $\Gama$. 

		\item elements of $\Comm_p(\Gama)$ are isomorphisms between $p$-open subgroups of $\Gama$,  modulo being equal on a $p$-open subgroup. The kernel of $i: \Gama \to \Comm_p(\Gama)$ is the set of elements which centralize a $p$-open subgroup of $\Gama$. 
	\end{itemize}
	
In what follows we always endow the groups $\Comm(\Gama)$ and $\Comm_p(\Gama)$ with the topology from Proposition \ref{prop-topo-germs}.  By definition the pro-$p$ topology is coarser than the profinite topology, so applying Lemma \ref{lem-germ-compare-topologies} 
to the topological group $(\Gama,\tau_{prof})$ we obtain the following:

\begin{Lemma}
	Suppose that an abstract group $\Gama$ has trivial virtual center. Then the natural homomorphism $\Comm_p(\Gama) \to \Comm(\Gama)$ is injective. 
\end{Lemma}

Again let  $\Gama$ be an abstract group and $p$ a prime number. We denote by $\widehat{\Gama}_p$  the pro-$p$ completion of $\Gama$; thus we have a  completion map $\Gama \rightarrow \widehat{\Gama}_p$, which is injective precisely when $\Gama$ is residually-$p$ 
(by definition this holds if elements of $\Gama$ can be separated by homomorphisms to finite $p$-groups).

 Part (3) of the following lemma establishes a natural map from the $p$-commensurator of an abstract group to the commensurator of its pro-$p$ completion, {which will be used repeatedly in the paper.} 
 
\begin{Lemma} \label{lem-localiso-extends-completion}
Let $\Gama$ be an abstract group.
\begin{enumerate}
\item  If $H$ is a $p$-open subgroup of $\Gama$, then the restriction to $H$ of the pro-$p$ topology on $\Gama$ is equal to the pro-$p$ topology on $H$. 
\item Given an injective  homomorphism $\theta: H \to \Gama$ with $p$-open image,  $\theta$ extends to a continuous injective open homomorphism $\widehat{\theta}_p:  \widehat{H}_p \to  \widehat{\Gama}_p$. In particular, if $H$ is a $p$-open subgroup of $\Gama$ and $\theta: H \to \Gama$ is the inclusion map, then we may regard $\widehat{H}_p$ as an open subgroup of $\widehat{\Gama}_p$ via the map $\widehat{\theta}_p$.

\item  There is a continuous  homomorphism $\phi: \Comm_p(\Gama) \to \Comm(\widehat{\Gama}_p)$ given by $\phi([f]) = [\widehat{f}_p]$.
\end{enumerate}
\end{Lemma}

\begin{proof}
For (1) see \cite[Lemma 3.1.4(a)]{RZ-book}. Part (2) is a standard construction -- see \cite[\S~3.2]{RZ-book}. The fact that $\widehat{\theta}_p$ is injective follows from part (1) -- see \cite[Lemma 3.2.6]{RZ-book}.  \skv

(3) First let us explain why the map $\phi$ given by $\phi([f]) = [\widehat{f}_p]$ is well defined. For this,
it is enough to note that when $f: H \rightarrow K$ is a virtual isomorphism of $\Gama$ in the pro-$p$ topology and $g$ is the restriction of $f$ to some $p$-open subgroup $L$ of $H$, then $\widehat{g}_p$ is given by the restriction of $\widehat{f}_p$ to $\widehat{L}_p$ (regarding $\widehat{L}_p$ as an open subgroup of $\widehat{H}_p$, as in (2)).  

Next we show that $\phi$ is a  homomorphism.
Given $f,g \in \Comm_p(\Gama)$, there is an identity neighborhood $\propU{p}$ in $\widehat{\Gama}_p$ on which the composition $\widehat{f}_p\widehat{g}^{-1}_p$ is defined, and for all $x \in \Gama \cap \propU{p}$ we have $\widehat{f}_p\widehat{g}^{-1}_p(x) = \widehat{(fg^{-1})}_p(x)$; by continuity, the same equation holds for all $x \in \propU{p}$, and hence $[\widehat{f}_p][\widehat{g}_p]^{-1} = [\widehat{(fg^{-1})}_p]$, showing that $\phi$ is a  homomorphism.  

Finally, note that $\{ \widehat{H}_p : H \text{ is a $p$-open subgroup of }G\}$ is a base of neighborhoods of the identity in $\Comm(\widehat{\Gama}_p)$, and for each $p$-open subgroup $H$ of $\Gama$, the preimage $\phi^{-1}(\widehat{H}_p)$ is an open subgroup of $\Comm_p(\Gama)$ since it contains $H$. It follows that $\phi$ is continuous.
\end{proof}

In particular, Lemma~\ref{lem-localiso-extends-completion}(1) ensures that for a $p$-open subgroup $H$ of $\Gama$, we have $\Comm_p(H) = \Comm_p(\Gama)$. The same is of course true for the profinite topology: if $H$ is a finite index subgroup of $\Gama$, then the profinite topology on $H$ agrees with the restriction to $H$ of the profinite topology on $\Gama$.  We will use these facts frequently without further mention.
 
 We remark that an equivalent way to say that two subgroups $H, K$ of a group $\Gama$ are commensurable (Definition \ref{def-commensurated}) is that $H \cap K$ is open in the profinite topology on $H$ and on $K$. Note that this reformulation involves the profinite topologies on $H$ and $K$, and not the profinite topology on $G$. We make an analogous definition for the pro-$p$ topology. 
 
 \begin{Definition}
	Two subgroups $H, K$ of a group $\Gama$ are {\it $p$-commensurable} if $H \cap K$ is a $p$-open subgroup in  $H$ and $K$. This is an equivalence relation. A subgroup $L$ of $\Gama$ is {\it $p$-commensurated} if all the $\Gama$-conjugates of $L$ are $p$-commensurable with each other. 
\end{Definition}
 
\begin{Lemma} \label{lem-im-p-commens}
The image $i(\Gama)$ of $\Gama$ in $\Comm_p(\Gama)$ is a $p$-commensurated subgroup of $\Comm_p(\Gama)$. 
\end{Lemma}

\begin{proof}
Let $f: U \to V$ be an isomorphism between two $p$-open subgroups of $G$. By Lemma \ref{lem-G-commens-CommG} we have $[f] \iotta(U) [f]^{-1} = \iotta(V)$, and the subgroups $\iotta(U)$ are $\iotta(V)$ are $p$-commensurable since they are both $p$-open in $i(\Gama)$. 
\end{proof}

\subsection{Subgroups of the commensurator generated by automorphisms} \label{subsec-def-AComm}

Let $G$ be a topological group. Recall that for any finite index open subgroup $U$ of $G$ we denote by $\uAut(U)$
the canonical image of $\Aut(U)$ in $\Lg(G)$. We define $\AComm(G)$ to be the subgroup of $\Comm(G)$ generated
by all subgroups $\uAut(U)$.

If $\Gama$ is an abstract group,  we define 
$$\AComm(\Gama)=\AComm((\Gama,\tau_{prof})) \mbox{ and }\AComm_p(\Gama)=\AComm((\Gama,\tau_p)).$$
More explicitly, 
\begin{itemize}
\item $\AComm(\Gama)$ is the subgroup of $\Comm(\Gama)$ generated by $\uAut(H)$ where $H$ ranges over all finite index subgroups of $\Gama$;
\item $\AComm_p(\Gama)$ is the subgroup of $\Comm_p(\Gama)$ generated by $\uAut(H)$ where $H$ ranges over all $p$-open subgroups of $\Gama$.
\end{itemize}

As we will show shortly (see Lemma~\ref{lemma:Aut_1} below and the remark after it), under mild additional assumptions on $G$, the group $\AComm(G)$ is normal in $\Comm(G)$ and only depends on the commensurability class of $G$.

\begin{Definition}
	A topological group $G$ is {\it characteristically based} if every finite index open subgroup of $G$ contains a finite index open (topologically) characteristic subgroup.  We say that $G$ is {\it hereditarily  characteristically based} (h.c.b.) if every finite index open subgroup of $G$ is  characteristically based.
\end{Definition}

\begin{Lemma}
\label{lem:hcb}
Any topological group that is topologically finitely generated is h.c.b.. In particular, the following groups are h.c.b.:
	\begin{enumerate}
		\item A finitely generated abstract group equipped with the profinite topology or pro-$p$ topology.
		\item A finitely generated profinite group.\footnote{As we will explain in Section~\ref{sec-prelim-freegroups},
by a finitely generated profinite group, we mean a profinite group which is
topologically finitely generated.}
	\end{enumerate}
\end{Lemma}

\begin{proof}
Let $G$ be a topologically finitely generated group. By a standard argument $G$ contains only finitely many open subgroups of any given finite index. The intersection $G_n$ of all open subgroups of index $n$ is then a characteristic open subgroup, and by definition every finite index open subgroup contains $G_n$ for some $n$. Thus $G$ is characteristically based.
Further, by Schreier's subgroup lemma every finite index open subgroup of $G$ satisfies the same hypothesis, so in fact $G$ is h.c.b.
\end{proof}

\begin{Lemma}
	\label{lemma:Aut_1}
Let $G$ be a h.c.b. topological group. The following hold:
	\begin{itemize}
		\item[(a)] Let $U$ be any finite index open subgroup of $G$. Then $\AComm(U)=\AComm(G)$.
		\item[(b)] For any finite index open subgroup $U$ of $G$ and any $f\in \Comm(G)$ there exists a finite index open subgroup $Z$ of $G$ such that $$f\uAut(U)f^{-1}\subseteq \uAut(Z).$$ In particular, 
		$\AComm(G)$ is normal in $\Comm(G)$. 
	\end{itemize}

In particular, these conditions hold when $G$ is a finitely generated abstract group equipped with the profinite topology or pro-$p$ topology or when $G$ is a finitely generated profinite group.
\end{Lemma}

\begin{proof} 
(a) The inclusion $\AComm(U)\subseteq \AComm(G)$ holds simply because every finite index open subgroup of $U$ is also a finite index open subgroup of $G$. To prove the reverse inclusion we need to show that for every finite index open subgroup $V$ of $G$, the group
$\uAut(V)$ lies in $\AComm(U)$. Since $G$ is h.c.b., it has a finite index open subgroup $W$ such that $W\subseteq U\cap V$ and
$W$ is characteristic in $V$. The latter implies that $\uAut(V)\subseteq \uAut(W)$, and since $W\subseteq U$, we have 
$\uAut(W)\subseteq \AComm(U)$, so $\uAut(V)\subseteq \AComm(U)$, as desired.
\skv

(b) Take any finite index open subgroup $U$ of $G$ and $[\phi]\in \Comm(G)$. Since $G$ is h.c.b., it has a finite index open subgroup $W$ 
which is characteristic in $U$ and such that $\phi$ is defined on $W$. Let $Z=\phi(W)$. Then for any $\psi\in \Aut(U)$ we have
$$\phi\,\psi\phi^{-1}(Z)=(\phi\,\psi)(W)=\phi(W)=Z,$$
so $[\phi]\,\uAut(U)[\phi]^{-1}\subseteq \uAut(Z)$, as desired.
\end{proof}
\me{
\begin{Remark} Both parts of Lemma~\ref{lemma:Aut_1} were previously established in the Ph.D. thesis of Chistopher Odden
\cite[Propositions~1.4.2~and~1.4.3]{oddenthesis} in the special case where $G$ is an orientable surface group of genus $>1$;
however, the proof in the general case is essentially the same. We are grateful to the referee for pointing out this reference.
\end{Remark}
}

\section{Preliminaries on free groups and their automorphisms.} \label{sec-prelim-freegroups}

This section is organized as follows. We will start with some basic properties of the automorphism groups
of free groups in subsection~\ref{subsec-free-groups-gen-aut}, followed by the discussion of some results on finite index subgroups
in free groups and their generating sets in subsection~\ref{subsect-free-abstract-subgroups}. After reviewing a few general facts about generation in profinite groups in 
subsection~\ref{subsec-prelim-pro-p}, in subsection~\ref{subsec-prelim-freeprop} we will state the analogues of some of the results from 
subsection~\ref{subsect-free-abstract-subgroups} for free pro-$p$ groups.  Finally, in subsection~\ref{subsec-prelim-auto} we will introduce a natural topology on the automorphism groups of profinite groups.
The discussion of the automorphism groups of free pro-$p$ groups will be deferred till subsection~\ref{subsec-generation-commfreeprop}.

\subsection{The subgroup $\SAut(F)$ and Nielsen generators} \label{subsec-free-groups-gen-aut}

Let $\absF{}$ be a free group of rank $d>1$. The abelianization map $\absF{}\to \absF{}/[\absF{},\absF{}]\cong \dbZ^d$ induces a homomorphism $\pi:\Aut(\absF{})\to \Aut(\dbZ^d)\cong\GL_d(\dbZ)$, which is known to be surjective \cite[Proposition 4.4]{LS}. The special automorphism group $\SAut(\absF{})$ is defined by $$\SAut(\absF{})=\pi^{-1}(\SL_d(\dbZ)).$$
Since $\pi$ is surjective, we have $[\Aut(\absF{}):\SAut(\absF{})]=[\GL_d(\dbZ)):\SL_d(\dbZ))]=2$. 
\me{Note that the isomorphism between $F/[F,F]$ and $\dbZ^d$ (and hence the map
$\pi$) depends on the choice of a free generating set $X$ of $F$. However, $\pi$ is well defined up to conjugation; in particular, the subgroup $\SAut(F)$ does not depend on $X$.}

It is well known that a subset $X$ of $F$ is a free generating set if and only
if $X$ generates $F$ and $|X|=d$. For brevity, we will refer to free generating sets as {\it bases} of $F$. Let $X = \left\lbrace x_1,\ldots,x_d \right\rbrace $ be a fixed basis of $\absF{}$. For $1\leq i\neq j\leq d$ define
$R_{ij},L_{ij}\in \Aut(F)$ by $R_{ij}(x_i)=x_i x_j$ and $R_{ij}(x_k)=x_k$ for $k \neq i$; and $L_{ij}(x_i)=x_j x_i$ and $L_{ij}(x_k)=x_k$ for $k \neq i$. The maps $R_{ij}$ and $L_{ij}$ are called {\it Nielsen transformations}. Clearly, $R_{ij},L_{ij}\in \SAut(F)$, and
it is a classical theorem of Nielsen~\cite{Ni} (see also \cite[\S~I.4]{LS})  that $\SAut(\absF{})$ is generated by the maps $R_{ij}$ and $L_{ij}$.

Every permutation of $X$ defines an element of $\Aut(F)$. We denote by $\Sym(X)$ the corresponding subgroup of $\Aut(F)$ and by 
$\Alt(X) \subseteq \Sym(X)$ the subgroup consisting of even permutations. 

The following lemma collects some basic properties of $\SAut(\absF{})$.

\begin{Lemma} \me{Let $X$ be a basis of $\absF{}$.}
	\label{lem:absperfect}
	The following hold:	
	\begin{itemize}
		\item[(a)] $\Sym(X) \cap \SAut(\absF{}) = \Alt(X)$ . 
		\item[(b)] Assume that $\rk(F)\geq 3$. Then $\SAut(\absF{})$ is perfect and $[\Aut(\absF{}),\Aut(\absF{})]=\SAut(\absF{})$. 
	\end{itemize}
\end{Lemma}
\begin{proof} 
	(a) holds since the image
	of an element $\sigma$ of $ \Sym(X)$ in $\Aut(\absF{}/[\absF{},\absF{}])\cong \GL_d(\dbZ)$ is the permutation matrix $P_{\sigma}$ and
	$\det(P_{\sigma})=\mathrm{sgn}(\sigma)$.
	\skv
	(b) Let $d=\rk(\absF{})$. Since $d\geq 3$, for any indices $1\leq i\neq j\leq d$ there exists $1\leq m\leq d$ with $m\neq i,j$,
	and by direct computation $R_{ij}=[R_{im},R_{mj}]$ and $L_{ij}=[L_{im},L_{mj}]$. Since $\SAut(F)$ is generated by $R_{ij}$ and $L_{ij}$,
	it follows that $\SAut(F)$ is perfect.
	Finally, $[\Aut(\absF{}),\Aut(\absF{})]\subseteq \SAut(\absF{})$ simply because $[\GL_d(\dbZ),\GL_d(\dbZ)]\subseteq \SL_d(\dbZ)$.
\end{proof}

\subsection{On the subgroup structure of free groups}\label{subsect-free-abstract-subgroups}

Let $F$ be a non-abelian free group of finite rank. 

\skv
\paragraph{\bf Normal subgroups with finite cyclic quotients.} Throughout the paper we will frequently deal with automorphisms of finite index subgroups of $F$. The following situation  will play a {special} role. 

\begin{Definition}
Let  $m \geq 2$. We denote by $SCQ(F,m)$ the set of normal subgroups $H$ of $F$ such that $F/H$ is cyclic of order $m$.
\end{Definition}

For a general $m$, among the normal subgroups of index $m$, the ones in $SCQ(F,m)$ are very special. However when $m=p$ is prime, every normal subgroup of index $p$ is in $SCQ(F,p)$.

	\begin{Definition}
	Let  $m \geq 2$. Let $X$ be a basis of $F$ and let $x \in X$. We denote by $ F(X,x,m) $ the subgroup of $\absF{}$ normally generated by $x^m$ and $X \setminus \left\lbrace  x\right\rbrace $. Equivalently, $F(X,x,m)$ is the unique normal subgroup of index $m$ in $F$	
which contains $X\setminus\{x\}$.
	\end{Definition}
	
The following fact is well known, but we are not aware of an explicit reference in the literature, so we will provide a proof.

\begin{Lemma} \label{lem-description-SCQ}
Fix $m \geq 2$. Then  any subgroup in $SCQ(F,m)$ is equal
to $F(X,x,m)$ for some basis $X$ of $F$ and $x\in X$. Moreover, $\SAut(F)$ acts transitively on $SCQ(F,m)$.
\end{Lemma}

\begin{proof}
Let $H$ be any subgroup in $SCQ(F,m)$ and choose an epimorphism $\phi:F\to \dbZ/m\dbZ$ with $\Ker(\phi)=H$.
Given an ordered basis $X=(x_1,\ldots, x_d)$, let $t(X)\in (\dbZ/m\dbZ)^d$ be the vector
$(\phi(x_1),\ldots, \phi(x_d))$. Then for any indices $i\neq j$, the vector $t(R_{ij}^{\pm 1}X)$ is obtained
from $t(X)$ by adding (resp. subtracting) the $j^{\rm th}$ coordinate to (resp. from) the $i^{\rm th}$ coordinate.

Using such operations, any nonzero vector in $(\dbZ/m\dbZ)^d$ can be reduced to a vector with only one nonzero coordinate
(equal to the greatest common divisor of the coordinates of the original vector).
Therefore, starting with any basis $X_0$ and applying suitable Nielsen maps to it, we obtain a basis
$X$ such that $t(X)$ has at most one (in fact, exactly one) nonzero coordinate; equivalently,
$H=\Ker(\phi)$ contains $X\setminus\{x\}$. Since $H$ is normal of index $m$, it follows that $H=F(X,x,m)$, so we proved
the first assertion of Lemma~\ref{lem-description-SCQ}.
\skv

To deduce the second assertion from the first one, note that $\Aut(F)$ acts transitively on the set of ordered bases, so
for any bases $X,X'$ of $F$ and elements $x\in X$ and $x'\in X'$ there exists $\phi\in \Aut(F)$ which sends
$F(X,x,m)$ to $F(X',x',m)$. Moreover, if $\eps$ is the automorphism of $F$ which sends $x$ to $x^{-1}$ and
fixes other elements of $X$, then $\phi\, \eps$ also sends $F(X,x,m)$ to $F(X',x',m)$, and one of the maps $\phi$ and $\phi\,\eps$
lies in $\SAut(F)$.
\end{proof}

\skv
\paragraph{\bf The Schreier index formula.}
\skv

Any subgroup of $\absF{}$ is free, and the rank of a finite index subgroup $H$
of $\absF{}$ can be computed by the Schreier index formula:
$$\rk(H)=1+[\absF{}:H]\cdot\rk(\absF{}).$$
Moreover, suppose we are given a basis $X$ for $\absF{}$ and 
a right Schreier transversal $T$ for $H$ and $\absF{}$ with respect to $X$
(that is, a set of right coset representatives closed under taking suffixes,
where elements of $\absF{}$ are viewed as reduced words in $X\sqcup X^{-1}$).
Then $H$ has a basis consisting of non-identity elements of the form
$$\{tx{\overline{tx}\,}^{-1}: t\in T, x\in X\}\eqno (***)$$ where $\overline g$ is the unique
element of $T$ with $H\overline{g}=Hg$. 

\begin{Definition}
We adopt the following convention for commutators: $[x,y] = xyx^{-1}y^{-1}$, and a commutator of length more than two should be read as a {\it left-normed commutator}:
$[x_1,\dots,x_{n+1}] = [[x_1,\dots,x_n],x_{n+1}]$.
\end{Definition}

The following will be used repeatedly, sometimes without further mention, later in the paper.

\begin{Lemma}
	\label{lem:Schreier}
	 Let $X$ be a basis of $\absF{}$, let $x \in X$, and let $m \geq 2$. Then each of the following sets is a free generating set for $ F(X,x,m)$:
	\begin{itemize}
		\item[(a)] $Y=\{x^m\}\cup\{ x^j y x^{-j}: y\in X\setminus\{x\}, 0\leq j\leq m-1\}.$
		\item[(b)] $Z=\{x^m\}\cup\{[y,{}_j x]: y\in X\setminus\{x\}, 0\leq j\leq m-1\}$
		where $[y,{}_j x]$ is the left-normed commutator 
		$[y,\underbrace{x,\ldots,x}_{j\mbox{ times} }]$.
	\end{itemize}
\end{Lemma}

\begin{proof} (a) $Y$ is precisely the generating set given by (***) if we set
	$T=\{1,x,\ldots, x^{m-1}\}$. 
	
	(b) Since $|Z|=|Y|$, it suffices to show that $Z$ is a generating set for $ F(X,x,m)$.
	To deduce this from (a) it suffices to show that 
	$\la y,x y x^{-1},\ldots, x^k y x^{-k} \ra=\la y,[y,x],\ldots, [y,{}_k x]\ra$ for all
	$y\in \absF{}$ and $k\in \dbN$. The latter can be proved by routine induction
	using the observation that $[y,{}_k x]=[y,{}_{k-1} x](x[y,{}_{k-1} x]x^{-1})^{-1}$.
\end{proof}
\skv
\paragraph{\bf Primitive elements and free factors.} An element $g\in F$ is called primitive (for $F$) if $g$ belongs to some basis of $F$.
Every non-trivial element of $F$ is primitive for some finite index subgroup. This is a special case of the famous free factor theorem of M. Hall (see \cite[Theorem~5.1]{Ha}~and~\cite[Theorem~1]{Bu}\footnote{This result does not follow from the statement of Theorem~5.1 in 
\cite{Ha}, but it follows from its proof as explained in \cite{Bu}.}):

\begin{Theorem}\label{thm:Hallfreefactor}
Let $H$ a finitely generated subgroup of $F$. Then $H$ is a free factor for some finite index subgroup of $F$. 
\end{Theorem}

We will revisit Theorem~\ref{thm:Hallfreefactor} and discuss further generalizations in Section~\ref{sec-conjugacy}.

\subsection{Generation in profinite groups and the Frattini subgroup} \label{subsec-prelim-pro-p}

Before turning to free pro-$p$ groups, we make a short digression and discuss some basic facts about generation in arbitrary profinite groups, with emphasis on pro-$p$ groups. Following a standard convention, by a generating set of a profinite group we will mean a topological
generating set unless explicitly mentioned otherwise. In particular, we will say that a profinite group is {\it finitely generated} if it has a finite topological generating set. However, we will not follow this convention whenever discussing generation of abstract and profinite groups in the same setting, most notably in Section~\ref{sec-conjugacy}.

\begin{Definition}
	Let $G$ be a profinite group. The {\it Frattini subgroup} $\Phi (G)$ is the intersection of all maximal proper open subgroups of $G$. 
\end{Definition}

The  Frattini subgroup $\Phi(G)$ has the following key properties (see \cite[2.8.5]{RZ-book}):

\begin{Lemma}
	\label{lem:Frattinibasic}
	Let $G$ be a profinite group.  The following hold:
	\begin{itemize}
		\item[(i)] A subset $S$ of $G$ is a generating set if and of only if
		$G/\Phi(G)$ is generated by the image of $S$.
		\item[(ii)] An element $g\in G$ lies in $\Phi(G)$ if and only if $g$ is a non-generator,
		that is, for any generating set $S$ of $G$, the set $S\setminus\{g\}$ also generates $G$.
	\end{itemize}
\end{Lemma}

If $G$ is pro-$p$, there is a simple explicit formula for the Frattini subgroup:
$\Phi(G)=\overline{[G,G]G^p}$, where $G^p = \la x^p : x \in G\ra $ is the subgroup generated by the $p^{\rm th}$ powers
(see \cite[2.8.7]{RZ-book}). Thus the Frattini quotient $G/\Phi(G)$ is an elementary abelian $p$-group. Things become even nicer if in addition $G$ is finitely
generated \cite[2.8.10-2.8.13]{RZ-book}:

\begin{Proposition}
	\label{prop:Frattini}
	Let $G$ be a finitely generated pro-$p$ group. The following hold:
	\begin{enumerate}[resume]
		\item[(a)] $\Phi(G)=[G,G]G^p$ and $\Phi(G)$ is open in $G$.
		\item[(b)] Define the Frattini series $(\Phi^{\ell} (G))_{\ell\geq 0}$
inductively by $\Phi^0(G) =G$ and $\Phi^{\ell + 1} (G) = \Phi (\Phi^{\ell}(G))$ for every $\ell \geq 0$. Then 
$(\Phi^{\ell} (G))_{\ell\geq 0}$ is a base of neighborhoods of $1$ in $G$.
	\end{enumerate}
\end{Proposition}

\subsection{Free pro-$p$ groups}\label{subsec-prelim-freeprop}

Let $X$ be a finite set and $F=F(X)$, the free group on $X$.
The pro-$p$ completion $\propF{p}=\widehat{F}_p$ is called a free pro-$p$ group on $X$, and $X$ will be referred to as a basis of 
$\propF{p}$. We will also say that $\propF{p}$ is free pro-$p$ of finite rank, and the number $d=|X|$ (which is determined
by the isomorphism class of $\propF{p}$) will be called the rank of $\propF{p}$.

More generally, a subset $Y$ of $\propF{p}$ will be called a {\it basis} if the inclusion map $Y\to\propF{p}$ induces an isomorphism
from $\widehat{F(Y)}_p$ to $\propF{p}$. In complete analogy with free groups, bases of $\propF{p}$ are precisely the (topological) generating sets of the smallest cardinality~\cite[Lemma 3.3.5.b]{RZ-book}.

The Schreier index formula remains valid in pro-$p$ groups. Indeed, if $F$ is a free group of finite rank $d$
and $\propF{p} = \widehat{F}_p$ is its pro-$p$ completion, then the map $\propU{p} \to \propU{p}\cap F$ establishes an index-preserving
bijection between open subgroups of $\propF{p}$ and $p$-open subgroups of $F$, and any such $\propU{p}$ is isomorphic to the pro-$p$ completion
of $\propU{p}\cap F$.

The category of pro-$p$ groups admits free products; see \cite[\S~9.1]{RZ-book} for basic features of free pro-$\mathcal{C}$ products of finitely many groups, which is the only case we will need.  When discussing pro-$p$ groups, $A * B$ will denote the free pro-$p$ product, rather than the abstract free product.  Free pro-$p$ factors of a pro-$p$ group are always closed. We will need the following pro-$p$ analogue of Hall's free factor theorem:

\begin{Theorem}[See {\cite[Theorem~3.2]{Lubotzky}} or {\cite[Theorem~9.1.19]{RZ-book}}]\label{thm-free-factor}
Let $\propF{p}$ be a free pro-$p$ group of finite rank and let $H$ be a closed subgroup of $\propF{p}$ that is topologically finitely generated.  Then there is an open subgroup $K$ of $\propF{p}$ such that $H$ is a free pro-$p$ factor of $K$.
\end{Theorem}

\subsection{The $A$-topology on the automorphism group of a profinite group}\label{subsec-prelim-auto}
\label{subsec-prelim-auto-general}
A detailed discussion of the automorphism groups of free pro-$p$ groups will be deferred till subsection~\ref{subsec-generation-commfreeprop}. We finish this section by introducing a natural topology on $\Aut(G)$ for a profinite group $G$, which we will call the $A$-topology. We warn the reader that the canonical
map $\Aut(G)\to\Comm(G)$ is not continuous with respect to the $A$-topology unless $\Aut(G)$ is a finite extension of $\Inn(G)$,
the subgroup of inner automorphisms.

Let $G$ be a profinite group.  For every open \alb{normal} subgroup $U$ of $G$ define $\Aut(G;U)\subseteq \Aut(G)$ by $$\Aut(G;U)=\{\phi\in\Aut(G): \phi(g)\equiv g\mod U\mbox{ for all }g\in G\}.$$ 
Each $\Aut(G;U)$ is a subgroup of $\Aut(G)$, and the family of all subgroups $\Aut(G;U)$
 satisfies the hypotheses of Proposition~\ref{prop:topbase} and has trivial intersection, so
$\Aut(G)$ has the structure of a Hausdorff topological group where the subgroups
$\Aut(G;U)$ form a base of neighborhoods of the identity.  We will call this the {\it $A$-topology} on $\Aut(G)$.

When $G$ is a finitely generated profinite group, $\Aut(G)$ with the $A$-topology is also compact and hence profinite \cite[Corollary 4.4.4]{RZ-book}. More is true in the case {where} $G$ is a finitely generated pro-$p$ group \cite[Lemma 4.5.5]{RZ-book}:

\begin{Proposition}
	\label{Aut:virtprop}
	Let $G$ be a finitely generated pro-$p$ group. Then the subgroup
	$$\Aut(G;\Phi(G))=\Ker(\Aut(G)\to\Aut(G/\Phi(G)))$$ 
	(which is open by Proposition~\ref{prop:Frattini}(a)) is pro-$p$.  Thus $\Aut(G)$ is virtually pro-$p$.
\end{Proposition}

Later in the paper, we will apply Proposition~\ref{Aut:virtprop} in two separate proofs, both times in the case where $G$ is free pro-$p$.
In one of the instances the specific topology on $\Aut(G)$ will not play a role; we will just need to know that
$\Aut(G;\Phi(G))$ is pro-$p$ with respect to some topology.

\section{Some precursors to the simplicity theorems} \label{sec-prelim-simpl}

Recall that throughout the paper $F$ denotes a {non-abelian} free group {of finite rank}.

\subsection{Stability and instability of $\det$ under passing to invariant subgroups}

In this subsection we  establish  basic results relating the groups $\SAut(F)$ and $\SAut(H)$
where $H$ is a suitable finite index subgroup of $F$. The first result of this kind follows immediately 
from Lemma~\ref{lem:absperfect}(b):

\begin{Corollary}
\label{cor:SAut_char}
Suppose that $\rk(\absF{})\geq 3$. Let $H$ be a finite index characteristic subgroup of $\absF{}$, so that 
$\Aut(\absF{}) \subseteq \Aut(H)$ if we view both $\Aut(F)$ and $\Aut(H)$ as subgroups of $\Comm(F)$. Then $\SAut(\absF{}) \subseteq \SAut(H)$.
\end{Corollary}

\begin{proof} Note that $\rk(H)\geq \rk(\absF{})\geq 3$.
Hence by Lemma~\ref{lem:absperfect}(b) we have 
$$\SAut(\absF{})=[\Aut(\absF{}),\Aut(\absF{})] \subseteq [\Aut(H),\Aut(H)]=\SAut(H).\qedhere$$
\end{proof}

\begin{Notation}
If $H$ is a finite index subgroup of $F$, we denote by  $\Aut(F)_H$  the stabilizer of $H$ in $\Aut(F)$. Hence inside the ambient group $\Comm(\absF{})$ we have $\Aut(F)_H \subseteq \Aut(H)$. 
\end{Notation}

Given $\alpha\in \Aut(F)$, we set $\det(\alpha)=\det(\pi(\alpha))$ where $\pi:\Aut(\absF{})\to \GL_d(\dbZ)$ is the 
\me{natural} map.
Thus, $\alpha\in \SAut(F)$ if and only if $\det(\alpha)=1$. If $H$ is a finite index subgroup of $F$ such that $\alpha \in \Aut(F)_H$, we set $\det_H(\alpha)=\det(\alpha_{|H})$ where $\alpha_{|H}$ is the automorphism of $H$ induced by $\alpha$. 
Corollary~\ref{cor:SAut_char} can be rephrased by saying that if $\rk(F)\geq 3$, $\det_F(\alpha)=1$ and $H$ is characteristic in $F$,
then $\det_H(\alpha)=1$ as well. As we will see shortly, such implication does not hold if we only assume that $H$ is $\alpha$-invariant.
\skv

\begin{Lemma} \label{lem-stab-not-contained-special-aut}
Let $m \geq 2$ and $H \in SCQ(F,m)$. The following hold:
\begin{itemize}
\item[(a)] There exists $\alpha\in\Aut(F)_H$ such that  $\alpha_{|H} \not\in \SAut(H)$. Hence if we view $\Aut(F)$ and $\Aut(H)$
as subgroups of $\Comm(F)$, then $\Aut(H)\subseteq \la\SAut(H),\Aut(F)\ra$.
\item[(b)] If $m=2$, or if $m \equiv 3 \mod 4$ and $\rk(F)$ is even, then there exists $\alpha\in\Aut(F)_H$ such that
$\alpha\not\in \SAut(F)$, but $\alpha_{|H}\in \SAut(H)$. Hence inside $\Comm(F)$ we have $\Aut(F)\subseteq \la\SAut(F),\SAut(H)\ra$.
\end{itemize}
\end{Lemma}

\begin{proof} Let $d=\rk(F)$. By Lemma \ref{lem-description-SCQ} we have $H=F(X,x_1,m)$ for some basis $X = \left\lbrace x_1,\ldots,x_d \right\rbrace $ of $\absF{}$, 
and then by Lemma~\ref{lem:Schreier}(a) $H$ has a basis $$Y=\left\{x_1^m,   x_1^k x_i x_1^{-k}: 2\leq i\leq d, 0\leq k\leq m-1\right\}.$$
\skv

(a) Let $\alpha$ be the automorphism of $\absF{}$ given by $\alpha(x_2)=x_2^{-1}$ and $\alpha(x_i)=x_i$ for $i\neq 2$.
Then $\alpha$ preserves $H$; in fact, it inverts $m$ generators from $Y$, namely the elements $x_1^k x_2 x_1^{-k}$,
and fixes the remaining generators. Hence the matrix of the induced automorphism of $H/[H,H]$ is diagonal, with
all diagonal entries equal to $\pm 1$ and exactly $m$ entries equal to $-1$. Thus $\det_H(\alpha)=(-1)^m$, so $\alpha_{|H} \not\in \SAut(H)$ if $m$ is odd.

For $m$ even, we consider the automorphism $\gamma$ of $\absF{}$ that maps $x_2$ to $x_1  x_2 x_1^{-1}$ and fixes all other elements of $X$. 
In this case the matrix of the induced automorphism of $H/[H,H]$ is a permutation matrix associated to a cycle of length $m$,
so $\det_H(\gamma)=-1$. This finishes the proof of the first assertion of (a).

The second assertion of (a) follows from
the first one and the fact that $\SAut(H)$ has index $2$ in $\Aut(H)$, so there are no subgroups strictly between
$\Aut(H)$ and $\SAut(H)$.
\skv

(b) As in (a), it suffices to prove the first assertion.
First suppose that $m=2$ and consider the automorphism $\alpha$ used in the proof of (a).
Clearly $\det_{F}(\alpha)= -1$, and $\det_{H}(\alpha)=(-1)^2=1$ by the computation in (a), so $\alpha\not\in \SAut(F)$
while $\alpha_{|H} \in \SAut(H)$, as desired.
\skv
	
Now suppose that $m \geq 3$ and $d=\rk(F)$ is even. Define $\beta\in \Aut(F)$ by $\beta(x_1)=x_1^{-1}$ and $\beta(x_i)=x_i$ for $i\neq 1$,
so that $\det_{F}(\beta) = -1$. To show that $\beta$ preserves $H$ and compute $\det_{H}(\beta)$, it is more convenient
to use a slightly different basis $Z$ for $H$: $$Z=\left\{x_1^m, x_1^k x_i x_1^{-k}: 2\leq i\leq d, |k|\leq \frac{m-1}{2}\right\}.$$
Note that $\beta$ inverts one element of $Z$, namely $x_1^m$, fixes $d-1$ elements, namely $x_i$ for $2\leq i\leq d$, and
performs $(d-1)(m-1)/2$ transpositions on the remaining generators
(it swaps $x_1^k x_i x_1^{-k}$ with $x_1^{-k} x_i x_1^{k}$ whenever $k\neq 0$). Hence $\det_{H}(\alpha) = (-1)^{\frac{(d-1)(m-1)}{2}+1}=1$
since $m\equiv 3 \mod 4$ and $d$ is even. 
\end{proof}

\subsection{Preliminary results on the subgroups of $\Comm(F)$}

In this subsection we will establish two technical results on the subgroup structure of $\Comm(F)$. In both cases our goal is to show
that the subgroup of $\Comm(F)$ generated by a certain set contains $\SAut(H)$ where $H=F$ or $H\in SCQ(F,m)$.

\skv

We will need the following lemma, which can be extracted from the proof of \cite[Proposition~1]{Bri-Vog-hom}.  The proposition is originally stated for normal subgroups of $\Aut(F)$, but the argument works just as well for subgroups normalized by $\SAut(F)$.

\begin{Lemma} 
	\label{lem-alt-normally-generates}
	Suppose that $\rk(F)\geq 3$, and let $N$ be a subgroup of $\Aut(F)$ normalized by $\SAut(F)$.  Suppose that some element of $N$ acts as a non-trivial permutation of a basis for $F$.  Then $N$ contains $\SAut(F)$. 
\end{Lemma}

As in subsection~\ref{subsec-free-groups-gen-aut}, for any basis $X$ of $F$, we will view the symmetric group $\Sym(X)$
and the alternating group $\Alt(X)$ as subgroups of $\Aut(F)$. 

\begin{Proposition} 
	\label{prop-int-implies-saut}
Let $d=\rk(F)$ and  $X = \left\lbrace x_1,\ldots,x_d \right\rbrace $ a basis of $F$. Let $m \geq 2$, let  $H = F(X,x_1,m)$, and let $s \geq 1$ be such that $m$ does not divide $s$. Assume that $(d,m) \neq (2,2)$. Then the group generated by the $\SAut(H)$-conjugates of $x_1^s$ contains $\SAut(H)$. In particular, the  group generated by the $\SAut(H)$-conjugates of $F$ contains $\SAut(H)$.
\end{Proposition}

\begin{proof}
We write $y = x_1^m$ and $z_{k,i} = x_1^i x_k x_1^{-i}$. Recall from Lemma~\ref{lem:Schreier}(a)  that
$$Z=\{y, z_{k,i}: 2\leq k\leq d, 0\leq i\leq m-1\}$$ is a basis of $H$. Note that $|Z| = 1 + m(d-1) \ge 4$. Let $\alpha$ denote the automorphism of $H$ induced by the conjugation by $x_1^s$, and let $N$ denote the subgroup of $\Aut(H)$ generated by the $\SAut(H)$-conjugates of $x_1^s$.
By Lemma \ref{lem-alt-normally-generates}, it is enough to show that $N \cap \Alt(Z)$ is non-trivial. The remainder of the proof consists of constructing a non-trivial element in $N \cap \Alt(Z)$.

Since $m\nmid s$, one can write $s=mq+r$ with $ 1 \leq r \leq m-1$, so that $x_1^s = y^q x_1^r$. We have $\alpha(y)=y$, $\alpha(z_{k,i} ) = y^q \, z_{k,i+r} \, y^{-q}$ for $0\leq i \leq m-1-r$ and 
$\alpha(z_{k,i} ) =y^{q+1} \, z_{k,i+r-m} \, y^{-q-1}$ for $m-r\leq i \leq m-1$ . 
\skv

{\it Case 1:  $d \ge 3$}.  Define $\tau\in \SAut(H)$ by $\tau(z_{2,0})=z_{3,0}^{-1}$, $\tau(z_{3,0})= z_{2,0}$, and $\tau(z)=z$
for all $z\in Z\setminus\{z_{2,0},z_{3,0}\}$. {A direct computation shows that $\alpha \tau \alpha^{-1}$ maps 
{$z_{2,r}$ to $z_{3,r}^{-1}$, maps $z_{3,r}$ to $z_{2,r}$} and fixes the remaining elements of $Z$. 
Consider the commutator 
$\beta= \alpha \tau \alpha^{-1} \tau^{-1}$.} 

{For any $z\in Z$ denote by $\eps_z$ the automorphism of $H$ which  sends $z$ to $z^{-1}$ and fixes other elements of $Z$.}
Then we can write $\beta=\varepsilon \sigma$ where $\varepsilon = \varepsilon_{z_{2,0}} \varepsilon_{z_{3,r}}$ and $\sigma$ is the product of two transpositions $(z_{2,0},z_{3,0}) (z_{2,r},z_{3,r}) \in \Alt(Z)$. 

Note that $\beta \in N$ since $\alpha \in N$, $\tau \in \SAut(H)$ and $N$ is normalized by $\SAut(H)$. Let $\gamma$ be the $3$-cycle $(y, z_{3,0}, z_{2,r}) \in \Alt(Z)\subset \SAut(H)$.  Then  $\beta^{-1}\gamma \beta \gamma^{-1}\in N$ as well; on the other hand,
$\gamma$ commutes with $\varepsilon$, whence $\beta^{-1}\gamma \beta \gamma^{-1} = \sigma^{-1}\gamma \sigma \gamma^{-1}\in \Alt(Z)$. It follows that  $\sigma^{-1}\gamma \sigma \gamma^{-1}  \in N \cap \Alt(Z)$, and this element is non-trivial. Hence the proof is complete in this case.
\skv

{\it Case 2: $d=2$ and $m = 3$}. Then $r=1$ or $r=2$, and replacing $x_1^s$ by $x_1^{2s}$ if needed, we can assume that $r=1$. 
Define $\tau\in \SAut(H)$ by $\tau(z_{2,0})= z_{2,1}^{-1}$,  $\tau(z_{2,1})= z_{2,0}$ and $\tau(z)=z$ for $z\in Z\setminus\{z_{2,0},z_{2,1}\}$,
and as in Case~1 define $\beta=\alpha \tau \alpha^{-1} \tau^{-1}$. Then $\beta=\varepsilon \sigma$ where $\varepsilon=\eps_{z_{2,0}}\eps_{z_{2,2}}$
and $\sigma=(z_{2,0},z_{2,2},z_{2,1})$.  The equality $\sigma = \beta^2 \sigma \beta^{-1} \sigma^{-1}$ shows $\sigma \in N$.

\skv

{\it Case 3: $d=2$ and $m \geq 4$}. Again replacing $x_1^s$ by $x_1^{2s}$ if needed, we can assume $ 2 \leq r \leq m-2$. This ensures that $\left\lbrace z_{2,r}, z_{2,r+1}  \right\rbrace $ is disjoint from $\left\lbrace z_{2,0}, z_{2,1}  \right\rbrace $. Define 
$\tau\in \SAut(H)$ by $\tau(z_{2,0})= z_{2,1}^{-1}$,  $\tau(z_{2,1})=z_{2,0}$ and $\tau(z)=z$
for all $z\in Z\setminus\{z_{2,0},z_{2,1}\}$, and define $\beta$ as in previous cases. Then $\beta = \varepsilon \sigma$ where $\varepsilon = \varepsilon_{z_{2,0}} \varepsilon_{ z_{2,r+1}}$ and $\sigma=(z_{2,0},z_{2,1}) (z_{2,r},z_{2,r+1}) \in \Alt(Z)$. As in Case~1, 
if $\gamma$ is the $3$-cycle $(y, z_{2,1}, z_{2,r})$, then $\beta^{-1}\gamma \beta \gamma^{-1}$ is a non-trivial element in $N \cap \Alt(Z)$. This concludes the proof.
\end{proof}

\begin{Lemma} 
\label{lemma-generation-indexp}
Let $m \geq 2$. Then the subgroup of $\Comm(\absF{})$ generated by $\Aut(H)$ (resp. $\SAut(H)$), where $H$ ranges over $SCQ(F,m)$, contains $ \Aut(F)$  (resp $\SAut(\absF{})$).
\end{Lemma}

\begin{proof}
	Let $d=\rk(F)$, and fix a basis $X=\{x_1,\ldots, x_d\}$ of $F$. We first prove the assertion about $\SAut$. Recall that $\SAut(\absF{})$ is generated by the Nielsen maps $R_{ij}$ and $L_{ij}$. By symmetry, it suffices to show that {$(R_{12})_{|H}$} lies in $\SAut(H)$ for some $H \in SCQ(F,m)$. 

\skv
{\it Case 1: $d\geq 3$}. Consider the subgroup $H=F(X,x_d,m)$; recall that it has a basis
$Y=\{x_d^m, x_d^{j}x_ix_d^{-j}: 1\leq i\leq d-1, 0\leq j\leq m-1\}$. Since $d\geq 3$, $H$ is invariant under $R_{12}$,
and the induced automorphism $(R_{12})_{| H}$ sends $x_d^j x_1 x_d^{-j}$ to $(x_d^j x_1 x_d^{-j})(x_d^j x_2 x_d^{-j})$ for all $0\leq j\leq m-1$ 
and fixes other elements of $Y$.
In particular, $(R_{12})_{| H}$ is a product of $m$ Nielsen maps (relative to $Y$) and hence lies in $\SAut(H)$, as desired. 
\skv

{\it Case 2: $d=2$}. Now let $H=F(X,x_1,m)$; similarly to Case~1, it has a basis $Y=\{x_1^m, z_{i}: 0\leq i\leq m-1\}$
where $z_i = x_1^i x_2 x_1^{-i}$. 
We will show that $H$ is $R_{12}$-invariant and $(R_{12})_{| H} \in \SAut(H)$. We claim that
\begin{itemize}
\item[(i)] $R_{12}(x_1^m) = z_{1} \cdots z_{m-1} x_1^m z_{0}$;
\item[(ii)] $R_{12}(z_{i}) \in H$ and $R_{12}(z_{i}) = z_{i} \mod [H,H]$ for all $i$.
\end{itemize}
Conditions (i) and (ii) would imply that $R_{12}$ preserves $H$ and the matrix of the induced action of $R_{12}$ on the abelianization of $H$ is unipotent, so $(R_{12})_{| H} \in \SAut(H)$ as desired.

Condition (i) can be checked by direct computation. Let us now prove (ii) by induction on $i$. The base case $i=0$ is clear since $z_{0} = x_2$ and $R_{12}$ fixes $x_2$. For the induction step, we have 
$R_{12}(z_{i+1}) = R_{12}( x_1 z_{i} x_1^{-1}) = x_1 x_2 R_{12}(z_{i}) x_2^{-1} x_1^{-1}$. By the induction hypothesis, 
$R_{12}(z_{i}) = z_{i} \mod [H,H]$. Since $x_2 \in H$ we deduce that $R_{12}(z_{i+1}) = x_1 z_{i} x_1^{-1}  \mod [H,H] = z_{i+1} \mod [H,H]$. This completes the induction step. Combined with {(i)}, we deduce that $R_{12}$ preserves $H$ and the matrix of $R_{12}$ at the level of the abelianization of $H$ is a unipotent matrix. Hence $(R_{12})_{| H} \in \SAut(H)$. 
\skv
	
	Thus we proved the part of Lemma~\ref{lemma-generation-indexp} dealing with $\SAut(F)$. The assertion about $\Aut(F)$ follows from the one about $\SAut(F)$ combined with the fact that the automorphism $\alpha\in\Aut(\absF{})$ given by $\alpha(x_1)=x_1^{-1}$
	and $\alpha(x_k)=x_k$ for $k>1$ does not lie $\SAut(\absF{})$ and stabilizes $H$.
\end{proof}

\section{Normal subgroups of $\Comm(F)$} \label{sec-comm-abstract-free}

\subsection{Simplicity of $\AComm(F)$}

Let $F$ be a non-abelian free group of finite rank.
In this subsection we will prove the first part of Theorem~\ref{thmA} which asserts that $\AComm(F)$ is simple and is equal to $\mon(\Comm(F))$, the monolith of $\Comm(F)$. 

First we will show that all non-trivial elements of $F$
lie in a single orbit under the action of $\AComm(F)$ (see Proposition~\ref{prop-one-conj-class} below). We will establish this result as a straightforward consequence of M. Hall's free factor theorem (Theorem~\ref{thm:Hallfreefactor}).

\begin{Lemma}
\label{lem-Hall-improved}
Let $g \in F$ be a non-trivial element. Then there exists a finite index subgroup $H$ of $F$ and an element $x\in H$ such that $g \in H$, $g$ is primitive in $H$ and $x$ is primitive for both $H$ and $F$.
\end{Lemma}
\begin{proof} Let $x$ be any primitive element of $F$ which does not commute with $g$ and let $C=\la x,g\ra$. Then $C$ is free of rank $2$
(since it is $2$-generated and non-abelian) and hence $\{x,g\}$ is a basis for $C$. By Theorem~\ref{thm:Hallfreefactor}, there exists a finite index subgroup $H$ of $F$ which contains $C$
as a free factor. By construction $x$ and $g$ are both primitive for $C$ and hence also for $H$.
\end{proof}

\begin{Proposition}
\label{prop-one-conj-class}
Let $y_1,y_2$ be non-trivial elements of $F$. Then there is $\psi \in \AComm(F)$ such that $\psi y_1 \psi^{-1} = y_2$.
Moreover, there exist finite index subgroups $H_1,H_2$ of $F$ and $\phi_i \in \Aut(H_i)$ and 
$\phi \in \Aut(F)$ such that $\psi = \phi_2^{-1}\, \phi\, \phi_1$.
\end{Proposition}
\begin{proof} By Lemma~\ref{lem-Hall-improved}, there exist finite index subgroups $H_1$ and $H_2$ of $F$ and elements $x_i\in H_i$ 
such that $H_i$ contains both $y_i$ and $x_i$, $y_i$ is primitive for $H_i$ and $x_i$ is primitive for both $H_i$ and $F$ for $i=1,2$.
Since primitive elements in a finitely generated free group form a single orbit under the action of the automorphism group, 
there exist $\phi_i \in \Aut(H_i)$, $i=1,2$, and $\phi\in \Aut(F)$ such that $\phi_i$ maps $y_i$ to $x_i$ and $\phi$ maps
$x_1$ to $x_2$.
Thus, $\phi_i\, y_i \,\phi_i^{-1}=x_i$ and $\phi\, x_1\,\phi^{-1}=x_2$ (now viewing $x_i$ and $y_i$ as elements of $\Comm(F)$) 
whence $\psi = \phi_2^{-1}\, \phi\, \phi_1$ satisfies the required equality.
\end{proof}

The following result is an immediate consequence of Lemma~\ref{lem-stab-not-contained-special-aut}(b) applied with $m=2$:

\begin{Proposition}
	\label{AComm=SComm}
	Denote by $\SComm(F)$ the subgroup of $\Comm(F)$ generated by $\SAut(H)$ 
where $H$ ranges over all finite index subgroups of $F$. 
 Then $\SComm(\abs{F})=\AComm(\abs{F}).$
\end{Proposition}

We are now ready to prove the first part of Theorem~\ref{thmA}. In fact, we will prove a slightly stronger statement:

\begin{Theorem}
\label{thmAextended}
Let $N$ be a non-trivial subgroup of $\Comm(F)$ normalized by $\AComm(F)$. Then $N$ contains $\AComm(F)$. Therefore $\AComm(F)$ is simple and equals $\mon(\Comm(F))$, the monolith of $\Comm(F)$. 
\end{Theorem}

\begin{proof}
Recall that $\Comm(F)$ does not depend on $\rk(F)$. By Proposition~\ref{AComm=SComm} $\AComm(F)$ is generated by the subgroups $\SAut(H)$ where $H$ ranges over finite index subgroups of $F$. Below we  prove that $N$ contains $\SAut(F)$. The same argument will show that $N$ contains $\SAut(H)$ for any finite index subgroup $H$ of $F$, and hence contains $\AComm(F)$.

By Lemma \ref{lem-Comm-trivialQZ} (applied to $F$ equipped with the profinite topology), the subgroup $N\cap F$ is non-trivial. By Proposition~\ref{prop-one-conj-class}, all non-trivial elements of $F$ lie in a single conjugacy class in $\AComm(F)$. Since $N$ is normalized by 
$\AComm(F)$, we deduce that $N$ contains 
$F$. Applying Proposition \ref{prop-int-implies-saut} (for instance with $m=3$), we deduce that $N$ contains $\SAut(H)$ for every subgroup $H$ of index $3$ in $F$. Hence $N$ contains $\SAut(F)$ by Lemma~\ref{lemma-generation-indexp}. 

We have shown in particular that $\AComm(F)$ is simple. The equality $\AComm(F) = \mon(\Comm(F))$ also follows since $\AComm(F)$ is normal in $\Comm(F)$ by Lemma~\ref{lemma:Aut_1}. 
\end{proof}

\subsection{Non-simplicity of $\Comm(\absF{})$}

In this subsection we will establish the second part of Theorem~\ref{thmA}. It will be derived as a consequence of the following general theorem which 
establishes a sufficient condition for $\AComm(\Gama)$ to be a proper subgroup of $\Comm(\Gama)$: 

\begin{Theorem}
\label{thm:nonsimple}
Let $\Gama$ be a finitely generated abstract group. Assume that
\begin{itemize}
\item[(a)] $\Gama$ has either trivial virtual center or the unique root property;
\item[(b)]  $\Gama$ has two isomorphic normal finite index subgroups $U$ and $V$
such that the quotients $\Gama/U$ and $\Gama/V$ have distinct (multisets of) composition factors. 
\end{itemize}
Then $\AComm(\Gama)$ is a proper subgroup of $\Comm(\Gama)$
and therefore $\Comm(\Gama)$ is not simple.
\end{Theorem}

\begin{proof}
	Non-simplicity of $\Comm(\Gama)$ follows from the first assertion of the theorem and Lemma~\ref{lemma:Aut_1}(b), so we just need to show that $\AComm(\Gama)\neq \Comm(\Gama)$.

We start by introducing some terminology. Given a group $\Lambda$ and a finite index subnormal
subgroup $H$ of $\Lambda$, denote by $CF(\Lambda/H)$ the multiset of (isomorphism classes of) the composition factors
of any composition series which starts with $\Lambda$ and ends with $H$. This multiset is independent of the choice of such a series, which can be shown by applying the Jordan-H\"older theorem to the group $\Lambda/Core_{\Lambda}(H)$
where $Core_{\Lambda}(H)=\bigcap_{g\in \Lambda}H^g$ is the largest normal subgroup of $\Lambda$ contained in $H$.

Define $\Comm_{SN}(\Gama)$ to be the set of elements of $\Comm(\Gama)$ which can be represented by a virtual
isomorphism $\phi:A\to B$ where $A$ and $B$ are both subnormal in $\Gama$ and $CF(\Gama/A)=CF(\Gama/B)$. We will prove that
\begin{itemize}
\item[(i)] $\Comm_{SN}(\Gama)$ is a subgroup;
\item[(ii)] $\Comm_{SN}(\Gama)$ contains $\AComm(\Gama)$;
\item[(iii)] $\Comm_{SN}(\Gama)\neq \Comm(\Gama)$.
\end{itemize}
Theorem~\ref{thm:nonsimple} follows directly from (ii) and (iii).
\skv
We start with the proof of (i). Clearly, we only need to show that $\Comm_{SN}(\Gama)$ is closed under the group
operation in $\Comm(\Gama)$. So take any two virtual isomorphisms $\phi:A\to B$ and $\psi:C\to D$
where $A,B,C$ and $D$ are subnormal in $\Gama$, $CF(\Gama/A)=CF(\Gama/B)$ and $CF(\Gama/C)=CF(\Gama/D)$, and let 
$[\phi]$ and $[\psi]$ be the corresponding elements of $\Comm(\Gama)$. By definition of the group operation,
$[\psi][\phi]=[\theta]$ where $\theta:\phi^{-1} (B\cap C)\to \psi(B\cap C)$ is the composition of
$\psi$ and $\phi$ restricted to $\phi^{-1} (B\cap C)$.

Since $C$ is subnormal in $\Gama$, $B\cap C$ is subnormal in $B$ (and hence in $\Gama$). Moreover, since $\phi^{-1}$ (resp. $\psi$) is an isomorphism from $B$ to $A$ (resp. from $C$ to $D$) sending $B\cap C$ to $\phi^{-1} (B\cap C)$
(resp. $\psi(B\cap C)$), it follows that $\phi^{-1} (B\cap C)$ is subnormal in $A$, $\psi(B\cap C)$
is subnormal in $D$ (so both are subnormal in $\Gama$), $CF(A/(\phi^{-1} (B\cap C)))=CF(B/B\cap C)$ and $CF(C/B\cap C)=CF(D/\psi(B\cap C))$.
Therefore,
\begin{multline*}
CF(\Gama/(\phi^{-1} (B\cap C)))=CF(\Gama/A)\sqcup CF(A/(\phi^{-1} (B\cap C)))\\
=CF(\Gama/B)\sqcup CF(B/B\cap C)=CF(\Gama/B\cap C)
=CF(\Gama/C)\sqcup CF(C/B\cap C)\\
=
CF(\Gama/D)\sqcup CF(D/\psi(B\cap C))=CF(\Gama/\psi(B\cap C)).
\end{multline*}
Thus, $\theta:\phi^{-1} (B\cap C)\to \psi(B\cap C)$ represents an element of $\Comm_{SN}(\Gama)$, and we proved (i).
\skv
(ii) Since $\Comm_{SN}(\Gama)$ is a subgroup, we just need to check that it contains $\Aut(A)$ for 
every finite index subgroup $A$ of $\Gama$. This is automatic if $A$ is subnormal. In the general case,
choose a finite index subgroup $B$ contained in $A$ which is normal in $\Gama$ and let $C$ be a finite index subgroup of $B$ which is characteristic in $A$ -- such $C$ exists since $\Gama$ (and hence $A$) is finitely generated. Then $C$ is subnormal in $\Gama$ and $\Aut(A)\subseteq \Aut(C)\subseteq \Comm_{SN}(\Gama)$, as desired.
\skv
(iii) Let $U$ and $V$ be isomorphic finite index normal subgroups of $\Gama$ such that
$CF(\Gama/U)\neq CF(\Gama/V)$, and let $\phi:U\to V$ be any isomorphism. We claim that its class $[\phi]$ does not lie in $\Comm_{SN}(\Gama)$. 

Indeed, suppose that there exist subnormal finite index subgroups $C$ and $D$ of $\Gama$ with $CF(\Gama/C)=CF(\Gama/D)$
and an isomorphism $\psi:C\to D$ such that $[\phi]=[\psi]$. Let $C_1$ be a normal subgroup of $C$ contained in 
$U \cap \psi^{-1}(D\cap V)$ and $D_1 = \psi(C_1)$. Then $D_1\subseteq V$, $D_1$ is normal in $D$ and $C/C_1\cong D/D_1$, so
$CF(C/C_1)=CF(D/D_1)$ and therefore $CF(\Gama/C_1)=CF(\Gama/D_1)$.

Note that $\phi$ and $\psi$ are equivalent virtual isomorphisms of $\Gama$ which are both defined on $C_1$. By Lemma \ref{lem-BEW} or Lemma \ref{lem-unique-root},
hypothesis (a) in Theorem~\ref{thm:nonsimple} ensures that $\phi_{|C_1}=\psi_{|C_1}$. In particular,
$\phi(C_1)=\psi(C_1)=D_1$ whence $CF(U/C_1)=CF(\phi(U)/\phi(C_1))=CF(V/D_1)$.
On the other hand, $CF(\Gama/U)\neq CF(\Gama/V)$ (by the choice of $U$ and $V$), so
$CF(\Gama/C_1)=CF(U/C_1)\sqcup CF(\Gama/U)\neq CF(V/D_1)\sqcup CF(\Gama/V)=CF(\Gama/D_1)$, contrary to our earlier conclusion.
\end{proof} 

Given a group $\Gama$ satisfying the hypotheses of Theorem~\ref{thm:nonsimple}, a natural problem is to understand the quotient $\Comm(\Gama)/\AComm(\Gama)$.
Theorem~\ref{thm:CommACommquotient} below gives a sufficient condition for this quotient to be infinite.

\begin{Theorem}
\label{thm:CommACommquotient}Let $\Gama$ be a finitely generated abstract group. Assume that
\begin{itemize}
\item[(a)] $\Gama$ has either trivial virtual center or the unique root property;
\item[(b)] for every $n\in\dbN$ there exist pairwise isomorphic finite index normal subgroups $U_1,\ldots, U_n$ of $\Gama$ such that
the quotients $\Gama/U_i$ and $\Gama/U_j$ have distinct composition factors for all $i\neq j$. 
\end{itemize}
Then the quotient
$\Comm(\Gama)/\AComm(\Gama)$ is infinite.
\end{Theorem}
\begin{proof}
For $1\leq i\leq n$ choose an isomorphism $\phi_i:U_1\to U_i$. Then for any $i\neq j$ the map
$\phi_j\phi_i^{-1}$ is an isomorphism from $U_i$ to $U_j$, and hence by the proof of Theorem~\ref{thm:nonsimple}
$[\phi_j\phi_i^{-1}]\not\in \AComm(\Gama)$. Thus, the commensurations $[\phi_1],\ldots, [\phi_n]$
represent distinct elements of $\Comm(\Gama)/\AComm(\Gama)$, so $|\Comm(\Gama)/\AComm(\Gama)|\geq n$, and since
$n$ is arbitrary, $\Comm(\Gama)/\AComm(\Gama)$ is infinite.
\end{proof}

The hypotheses of Theorem~\ref{thm:CommACommquotient} are easily seen to be satisfied if $\Gama$ is a non-abelian free group or 
a non-abelian (orientable) surface group. Hypothesis (a) is clear.
In both cases subgroups of the same finite index are isomorphic, so to check (b) we just need to construct
normal subgroups $U_1,\ldots, U_n$ of $\Gama$ of the same finite index such that
the quotients $\Gama/U_i$ and $\Gama/U_j$ have distinct composition factors for all $i\neq j$. 

Suppose first that $\Gama$ is free and non-abelian. Choose $n$ pairwise non-isomorphic
finite simple groups $S_1,\ldots, S_n$ and then choose cyclic groups $C_1,\ldots, C_n$ such that $|S_i|\cdot |C_i|$ is the same for all $i$.
Each $S_i$ is $2$-generated. A standard argument shows that each $P_i=S_i\times C_i$ is also $2$-generated, so we can find
epimorphisms $\phi_i: \Gama\to P_i$, and then the subgroups $U_i=\Ker(\phi_i)$ have the desired properties.

Now let $\Gama=S_g$, the orientable surface group of genus $g>1$, and choose any epimorphism from $S_g$ to $F_g$, a free group of rank $g$. 
We just proved that there exist subgroups $V_1,\ldots, V_n$ of $F_g$ of the same finite index, say, $c_n$, and such that
$F_g/V_i$ and $F_g/V_j$ have distinct composition factors for all $i\neq j$. Now if $U_i$ is the preimage of $V_i$ in $S_g$,
then each $U_i$ has index $c_n$ in $S_g$ and $S_g/U_i\cong F_g/V_i$, so $U_1,\ldots, U_n$ satisfy (b).

\begin{Corollary}
	\label{commF:notsimple}
	Let $\Gama$ be non-abelian free group of finite rank or a non-abelian orientable surface group. Then $\Comm(\Gama)$ is not virtually simple.
\end{Corollary}

\section{A  family of finitely generated simple groups} \label{sec-finite-generated}
		
		Let $F$ be a non-abelian free group of finite rank. Recall that for $m \geq 2$ we denote by $SCQ(F,m)$ the set of all normal subgroups $H$ of $F$
	such that $F/H$ is cyclic of order $m$.

	\begin{Definition} \label{defi-An-Sn}
For $ m\geq 2$, let $A_{m}(F)$ be the subgroup of $\mathrm{Comm}(F)$ generated by the subgroups $\Aut(H)$ where $H$ ranges over $SCQ(F,m)$, and let $S_{m}(F)$ be the subgroup generated by the subgroups $\SAut(H)$ where $H$ ranges over $SCQ(F,m)$. 
	\end{Definition}

Clearly $S_{m}(F) \subseteq A_{m}(F)$. Note that the groups $A_m(F)$ and $S_m(F)$ are finitely generated since $SCQ(F,m)$ is finite and $\Aut(H)$ and $\SAut(H)$ are finitely generated for
each $H$.

\begin{Lemma} \label{lem-properties-Am(F)}
Let $m \geq 2$. The following hold: \begin{enumerate}
	\item \label{item-SAutF-in-Sm} $\SAut(F) \subseteq S_{m}(F)$  and $\Aut(F) \subseteq A_{m}(F)$. 
	\item \label{item--Sm-gen} For every $H \in SCQ(F,m)$, we have $S_{m}(F) =  \left\langle \SAut(F), \SAut(H) \right\rangle $ and  $A_m(F) = \left\langle \Aut(F), \Aut(H) \right\rangle $. 
	\item \label{item--Sm-index-2}  $S_{m}(F)$ has index at most $2$ in $A_m(F)$.
\end{enumerate} 
\end{Lemma}

\begin{proof}
(\ref{item-SAutF-in-Sm}) is Lemma~\ref{lemma-generation-indexp}. (\ref{item--Sm-gen}) follows from (\ref{item-SAutF-in-Sm}) and the fact that $\SAut(F)$ acts transitively on $SCQ(F,m)$ (by Lemma~\ref{lem-description-SCQ}). Let us prove (\ref{item--Sm-index-2}). Since $\Aut(F)$ permutes the subgroups $\SAut(H)$ where $H$ ranges over $SCQ(F,m)$, $\Aut(F)$ normalizes $S_m(F)$. 
Lemma~\ref{lem-stab-not-contained-special-aut}(a) implies that $A_m(F) = \left\langle \Aut(F), S_m(F) \right\rangle $.
Hence $S_m(F)$ is normal in $A_m(F)$ and we can write $A_m(F) = \Aut(F) S_m(F)$. Since $S_m(F)$ already contains $\SAut(F)$,
it follows that $[A_m(F):S_m(F)]\leq [\Aut(F):\SAut(F)]=2$.
\end{proof}

\begin{Remark}
	 Lemma~\ref{lem-stab-not-contained-special-aut}(b) implies that	$S_{m}(F) = A_{m}(F)$ if either $m=2$, or $m \equiv 3 \mod 4$ and 
$d=\rk(F)$ is even. We do not know whether there exist values of $(d,m)$ such that $S_{m}(F)$ is a subgroup of $A_{m}(F)$ of index exactly $2$, i.e.\ such that $S_{m}(F)$ is a proper subgroup of $A_{m}(F)$.
\end{Remark}

\begin{Lemma} \label{lem-conjugate-power-SmF}
Let $m \geq 2$ and $x$ a primitive element of $F$. Then $x$ is conjugate to $x^m$ in $S_{m}(F)$. 
\end{Lemma}

\begin{proof}
	Let	$X = \left\lbrace x_1,\ldots,x_d \right\rbrace $ be a basis of $F$ with $x=x_1$. The elements $x_1$ and $x_2$ are conjugate in $\SAut(F)$, and hence in $S_{m}(F)$ since $\SAut(F) \subseteq S_{m}(F)$ (Lemma \ref{lem-properties-Am(F)}). On the other hand $ \left\lbrace x_2, x_1^m \right\rbrace$ is a subset of a basis of $H = F(X,x_1,m)$, so one can conjugate $x_2$ to $ x_1^m$ in $\SAut(H)$.
\end{proof}

\begin{Proposition} \label{prop-SmF-fi-F}
Let $m \geq 2$, and let $d=\rk(F)$. Suppose $(d,m) \neq (2,2)$. Then the group $S_{m}(F)$ has no proper finite index subgroup. More generally, if $N$ is a normal subgroup of $S_{m}(F)$ such that $N$ contains a finite index subgroup of $F$, then $N = S_{m}(F)$. 
\end{Proposition}

\begin{proof}
Since every finite index subgroup contains a normal finite index subgroup, the first claim follows from the second one. So let $N$ be as in the statement. Let $H \in SCQ(F,m)$, and choose a basis $X$ of $F$ and $x \in X$ such that $H = F(X,x,m)$. Since $N$ contains a finite index subgroup of $F$, there is $n \geq 1$ such that $x^n \in N$. Write $n=m^k s$ where $k\geq 0$ and $m$ does not divide $s$. By Lemma~\ref{lem-conjugate-power-SmF} there is $c \in S_m(F)$ such that $cxc^{-1} = x^m$. Then $c^k x^s c^{-k} = x^n$. Since $N$ is normal in $S_m(F)$, 
we get $x^s=c^{-k}x^n c^k\in N$. Since  $m$ does not divide $s$ and $(d,m) \neq (2,2)$, we can apply Proposition~\ref{prop-int-implies-saut} and deduce that $\SAut(H) \subseteq N$. Since $H$ was arbitrary in $SCQ(F,m)$, this proves that $N = S_{m}(F)$. 
\end{proof}

In the sequel we write $\Gamma=\mathrm{PSL}_2(\Z)$. Recall that the principal congruence subgroup $\Gamma(\ell)$ of level $\ell \geq 2$ of $\Gamma$ is the kernel of the homomorphism $\Gamma \to \mathrm{PSL}_2(\Z/\ell\Z)$ defined by the reduction modulo $\ell$. 
It is well known that $\Gamma(\ell)$ is a free group.

\begin{Proposition} \label{prop-arithmetic}
Let $\Gamma = \mathrm{PSL}_2(\Z)$ and $G = \mathrm{PSL}_2(\R)$. Let $m \geq 2$, let $\ell \geq 2$ be a multiple of $m$, and let $F = \Gamma(\ell)$. Then the image of 
$\mathrm{PSL}_2(\Z[1/m])$ under the  map $\psi: \mathrm{Comm}_G(F) \to \mathrm{Comm}(F)$ lies in $A_{m}(F)$.
\end{Proposition}

Here $\mathrm{Comm}_G(F)$ stands for the commensurator of $F$ in $G$ (equivalently, the commensurator of $\Gamma$ in $G$ since $F$ and $\Gamma$ are commensurable). Recall that no non-trivial element of $G$ centralizes a finite index subgroup of $\Gamma$, so $\psi: \mathrm{Comm}_G(F) \to \mathrm{Comm}(F)$ is injective.

\begin{proof}
Throughout the proof we will use the following standard abuse of notation: when defining a subset of $\Gamma=\PSL_2(\dbZ)$ by certain conditions
on the matrix entries, we will mean the image of the corresponding subset of $\SL_2(\dbZ)$ in $\Gamma$.

Let $H$ be the set of matrices  $\begin{bmatrix}
 a & b \\ c & d
\end{bmatrix} \in \Gamma$ such that $a,d \equiv 1 \mod \ell$, $ \ell \mid b$ and $m \ell \mid c$. 
By definition, $H$ is contained in $F$. \alb{Further, if we consider the restriction to $F$ of the reduction modulo $m\ell$ homomorphism $F \to \mathrm{PSL}_2(\Z/m\ell\,\Z)$, then $H$ is precisely the preimage in $F$ of the upper-triangular subgroup of $\mathrm{PSL}_2(\Z/m\ell\,\Z)$. {\me {In particular,}} $H$ is a subgroup of $F$.} Moreover, since  $m$ divides $\ell$, a direct computation shows that $H$ is normal in $F$ and that $F/H$ is cyclic of order $m$. 

 Let $\delta_m  = \begin{bmatrix}
	0 & - m^{-1/2}  \\ m^{1/2} & 0
	\end{bmatrix} \in G$.  Then $\delta_m \begin{bmatrix}
 a & b \\ c & d
\end{bmatrix}\delta_m^{-1} =\begin{bmatrix}
 d & - c/m  \\ - mb  & a
\end{bmatrix}.$ In particular, $\delta_m$ normalizes $H$. It follows that if we let $\Delta =\la \Gamma,\delta_m\ra$ be the subgroup of $G$ generated  by $\Gamma$ and $\delta_m$,
then $\Delta \subseteq \mathrm{Comm}_G(F)$. Moreover, the image of $\Gamma$ under $\psi: \mathrm{Comm}_G(F) \to \mathrm{Comm}(F)$ lies in $\Aut(F)$, and the image of $\delta_m$ under $\psi$ lies in $\Aut(H)$. We have $\Aut(F) \subseteq A_{m}(F)$ by Lemma~\ref{lemma-generation-indexp}, and since $F/H$ is cyclic of order $m$, we also have $\Aut(H) \subseteq A_{m}(F)$. Hence $\psi(\Delta)\subseteq A_{m}(F)$. So to finish the proof it suffices to show that $\Delta$ contains $\mathrm{PSL}_2(\Z[1/m])$.

Let  $p$ be a prime divisor of $m$. We check that $\Delta$ contains $\mathrm{PSL}_2(\Z[1/p])$. Since the subgroups $\mathrm{PSL}_2(\Z[1/p])$ generate $ \mathrm{PSL}_2(\Z[1/m])$ where $p$ ranges over the prime divisors of $m$, this would imply that $\mathrm{PSL}_2(\Z[1/m]) \subseteq \Delta$.  Let $u=E_{21}(m/p)=\begin{bmatrix}
 1 & 0 \\ m/p & 1
\end{bmatrix}$ and $v=E_{12}(1/p)=\begin{bmatrix}
 1 & 1/p \\ 0 & 1
\end{bmatrix}$. We have $u\in \Gamma$ and $\delta_m u \delta_m^{-1} = v^{-1}$, so $v\in \Delta$. Hence 
$\left\langle \Gamma, v \right\rangle \subseteq \Delta$, and to complete the proof it suffices to check that 
$$\left\langle \Gamma, v \right\rangle  = \mathrm{PSL}_2(\Z[1/p]).$$ This equality is a consequence of the amalgamated free product decomposition of $ \mathrm{PSL}_2(\Z[1/p])$ (see, e.g., \cite[II.1.4 Cor. 2]{Serre}) but can also be verified directly as follows.

Since $\Z[1/p]$ is a PID, the group $\mathrm{PSL}_2(\Z[1/p])$ is generated by elementary matrices, so we just need to check
that the matrices $E_{12}(p^k)$ and $E_{21}(p^k)$ lie in $\la \Gamma, v \ra$ for arbitrarily small $k\in\dbZ$.
Since $v=E_{12}(1/p)$, the group $\la \Gamma, v \ra$ contains
$D=E_{12}(1/p)E_{21}(-p)E_{12}(1/p){\small\begin{bmatrix}
 0 & -1 \\  1 & 0
\end{bmatrix}}$ which is equal to  ${\small\begin{bmatrix}
 1/p   & 0 \\  0 & p
\end{bmatrix}}$. Hence $$E_{12}(p^{2k})=D^{-k} E_{12}(1)D^{k}\quad\mbox{ and }\quad E_{21}(p^{2k})= D^{k} E_{21}(1)D^{-k}$$ 
also lie in $\la \Gamma, v \ra$. 
\end{proof}

The following is the main simplicity result of this section. The proof shares the following feature with the strategy from \cite{Mar80}, \cite{BM00b} and \cite{Ca} which consists of the following two independent contributions: we first prove that a non-trivial normal subgroup must be relatively large (which here means that it contains a finite index subgroup of $F$), and then show that any normal subgroup with this property must be the entire group.

\begin{Theorem} \label{thm-fg-simple-general}
Let $m \geq 2$, let $\ell$ be a multiple of $m$, and let $d_{\ell}$ be the rank of the principal congruence subgroup $\Gamma(\ell)$.
If $F$ is a free group of rank $d_{\ell}$, then the group $S_{m}(F)$ is simple. 
\end{Theorem}

\begin{proof}
Let $N$ be a non-trivial normal subgroup of $S_{m}(F)$. Since $F$ has trivial virtual center, we can apply 
Lemma~\ref{lem-Comm-trivialQZ} to $G=F$ endowed with the profinite topology. Lemma~\ref{lem-Comm-trivialQZ}(\ref{item-Comm-local-normal})
implies that the intersection $N \cap F$ is non-trivial. The subgroup $N$ has at most two $A_m(F)$-conjugates, which intersect $F$ 
along non-trivial normal subgroups of $F$, say $K_1$ and $K_2$, and $K_1\cap K_2$ is non-trivial (since it contains $[K_1,K_2]$). Hence upon 
replacing $N$ by the intersection of its $A_m(F)$-conjugates, we can assume that $N$ is actually normal in $A_m(F)$.

Let $\Lambda$ be the image of $\mathrm{PSL}(2,\Z[1/m])$ under the map $\psi$ from Proposition \ref{prop-arithmetic}.
Then $\Lambda$ is isomorphic to $\mathrm{PSL}(2,\Z[1/m])$, contains $F$, and by Proposition \ref{prop-arithmetic},
$\Lambda$ is contained in  $A_{m}(F)$. The subgroup $N \cap \Lambda$ is non-trivial because $N \cap F$ is non-trivial, and $N \cap \Lambda$ is normal in $\Lambda$ because $N$ is normal in $A_{m}(F)$. Now the only non-trivial normal subgroups of $\Lambda$ are the ones of finite index -- see Mennicke \cite{Mennicke-SL2} for the case where $m$ is prime  and Serre \cite{Serre-congr-SL2} for the general case (this result is also a consequence of Margulis' normal subgroup theorem \cite{Margulis79,Margulis-book}). So $N \cap \Lambda$ has finite index in $\Lambda$, and we derive in particular that $N$ contains a finite index subgroup of $F$. Proposition \ref{prop-SmF-fi-F} now implies that $N=S_m(F)$, as desired.
\end{proof}

\begin{Remark} \label{remark-not-CAT(0)}
As mentioned in the introduction, the previously known examples of finitely generated infinite simple subgroups of $\Comm(F)$ were groups acting properly and cocompactly on a product of trees. The groups $S_{m}(F)$ seem to be of very different nature. Since $S_{m}(F)$ contains a subgroup isomorphic to $\SAut(H)$ for some free group $H$ of rank $\geq 3$, by \cite{Ger94} the group $S_{m}(F)$ cannot act properly and cocompactly by isometries on a complete CAT(0) space. More generally, for the same reason $S_{m}(F)$ cannot act properly by semi-simple isometries on a complete CAT(0) space \cite[7.18(2)]{Bri-Hae}. 
\end{Remark}

\begin{Remark} \label{remark-all-ranks}
	We believe that Theorem \ref{thm-fg-simple-general} remains true for every $m \geq 2$ and non-abelian free groups of every rank. However establishing an analogue of Proposition \ref{prop-arithmetic} (i.e.\ proving the existence of a just-infinite group in $S_m(F)$ which contains $F$) for arbitrary $d$ could be technically difficult.
\end{Remark}

We finish this section with an observation relating Proposition \ref{prop-arithmetic} to a peculiar property of characteristic subgroups
in free groups and some of their finite index subgroups.
\begin{Observation}\label{obs:abstractcharacteristic}
For each $m\geq 2$ there exist infinitely many $d \geq 2$ with the following property. If $F$ is a free group of rank $d$ and 
$H\in SCQ(F,m)$, then no non-trivial subgroup of $H$ can be characteristic in both $H$ and $F$.
\end{Observation}
\begin{proof}
Let $\ell$ and $d_\ell$ be as in Theorem \ref{thm-fg-simple-general} and set $d=d_{\ell}$. Recall that $\Gamma(\ell)$ 
denotes the principal congruence subgroup of level $\ell$ in $\mathrm{PSL}_2(\Z)$.
By Proposition~\ref{prop-arithmetic} there exists an embedding $\psi:\mathrm{PSL}_2(\Z[1/m])\to A_m(F)$ which maps $\Gamma(\ell)$ onto $F$.

Suppose now that $N\subseteq H$ is characteristic in both $H$ and $F$. Since $A_m(F)$ is generated by $\Aut(F)$ and $\Aut(H)$ by
Lemma~\ref{lem-properties-Am(F)}(\ref{item--Sm-gen}), $N$ is normal in $A_m(F)$ and hence $\psi^{-1}(N)$ is normal in $\mathrm{PSL}_2(\Z[1/m])$.
Since $\mathrm{PSL}_2(\Z[1/m])$ is just-infinite and $\psi^{-1}(N)\subseteq \Gamma(\ell)$ has infinite index in $\mathrm{PSL}_2(\Z[1/m])$, we deduce that $\psi^{-1}(N)$ is trivial, as desired. 
\end{proof}

We do not know if the analogue of Observation~\ref{obs:abstractcharacteristic} holds for free pro-$p$ groups with $m=p$ (even for a single value of $d$) -- see Question~\ref{q:superinvariant} in Section~\ref{sec:questions}. But what would be really interesting in view of our results in Section~\ref{sec-compact-generated} is a positive answer to the following question: 

\textit
{Let $F$ be a non-abelian free group of finite rank, $U$ a normal subgroup of index $p$ in $F$, let $\propF{p}$ be the pro-$p$ completion of $F$ and $\propU{p}$ the closure of $U$ in $\propF{p}$. Is it true that no non-trivial subgroup of $\propU{p}$ is invariant under both $\Aut(U)$and $\Aut(F)$?}

As we will see in Section~\ref{sec-compact-generated}, if true, this would solve in the affirmative Question~\ref{qCM} discussed in subsection~\ref{sec:tdlcintro} -- see 
Corollaries~\ref{Nn-superinvariant}~and~\ref{cor-conditional-Ln-simple}.

\section{Generation of commensurators by automorphisms of subgroups} \label{sec-gen-by-aut}

Following our general convention, in this section $F$ will denote a non-abelian free group of finite rank and $\propF{p}$ a non-abelian free pro-$p$ group of finite rank. We have shown earlier that $\SComm(F)=\AComm(F)$ and that $\AComm(F)$ is a proper subgroup of $\Comm(F)$. In this section we consider the corresponding questions for $\Comm_p(F)$  and $\Comm(\propF{p})$. The answers in these two cases will be similar to each other, and somewhat different from those for the corresponding subgroups of $\Comm(F)$.   

Our first result shows that every element of $\Comm_p(F)$ (resp. $\Comm(\propF{p})$) can be written as a product of
automorphisms of $p$-open subgroups of $F$ (resp. open subgroups of $\propF{p}$) and in fact provides an algorithm for finding such factorization.

\begin{Proposition}\label{prop-decomposing-prod-aut}
Let $\calF$ be either a non-abelian free group $\absF{}$ equipped with the pro-$p$ topology or a non-abelian free pro-$p$ group $\propF{p}$.  Let $U$ and $V$ be open subgroups of $\calF$ of the same index $p^n$, let $f:U\to V$ be an isomorphism and $[f]$ the corresponding element
of $\Comm_p(F)$ or $\Comm(\propF{p})$, respectively. Then there exist open subgroups $V = K_n \subseteq \ldots \subseteq K_0 = \calF$ with 
$[K_i: K_{i+1}] = p$ and 
automorphisms $f_0,f_1,\ldots,f_n$ such that 
\begin{enumerate}
	\item $f_0 \in \Aut(\absF{})$;
	\item $f_i \in \Aut(K_i)$ for every $1 \leq i \leq n$;
	\item $[f] = [f_n] \cdots [f_1] [f_0]$.
\end{enumerate} 
Moreover, in the case $\calF = \absF{}$ we can ensure that $f_i \in \SAut(K_i)$ for every $1 \leq i \leq n$. 
\end{Proposition}
\begin{proof}
We will argue by induction on $n$. If $n=0$, then $f\in \Aut(\calF)$, and there is nothing to prove. Next consider the case $n=1$ (in principle, we could go straight to the induction step, but we present the case $n=1$ separately for clarity). In this case $U$ and $V$ are both normal subgroups
of index $p$. Since $\Aut(\calF)$ acts transitively on normal subgroups
of index $p$, there exists $f_0\in \Aut(\calF)$ such that $f_0(U)=V$. But then $f_1=f(f_0^{-1})_{|V}\in \Aut(V)$ 
and $[f]=[f_1][f_0]$. In the case $\calF = \absF{}$ we shall check the last assertion. If $f_1\in \SAut(V)$, we are done. If $f_1\not\in \SAut(V)$, we choose $\alpha\in \Aut(\absF{})_V$ (the stabilizer of $V$ in $\Aut(\absF{})$)
whose restriction $\alpha_{|V}$ to $V$ does not lie in $\SAut(V)$ (this is possible by Lemma~\ref{lem-stab-not-contained-special-aut}(a)).
Then $f_1\alpha_{|V}\in \SAut(V)$, and replacing $f_0$ by $\alpha^{-1}f_0$ and $f_1$ by $f_1\alpha_{|V}$, we obtain a desired factorization.

Finally, we treat the general case. Since $U$ and $V$ are open, we can choose normal subgroups of index $p$, call them $M$
and $N$, such that $U\subseteq M$ and $V\subseteq N$. As in the case $n=1$, there exists $f_0\in \Aut(\calF)$ such that $f_0(M)=N$. Then
$f_0(U)$ and $V$ are both open subgroups of $N$ of index $p^{n-1}$. Since $h=f({f_0}^{-1})_{|N}$ sends $f_0(U)$ to $V$, by the induction
hypothesis (applied to $\Comm(N)$), there exist $p$-open subgroups $U = K_n \subseteq\ldots  \subseteq  K_1 = N$ with $[K_i: K_{i+1}] = p$, $f_1\in \Aut(N)$ and automorphisms
 $f_i\in \Aut(K_i)$ for $2\leq i\leq n$ such that $[h]=[f_n] \cdots [f_1]$ and hence $$[f]=[h][f_0]=[f_n] \cdots [f_1][f_0].$$
In  the case $\calF = \absF{}$ we can modify $f_1$ and $f_0$ by a suitable $\alpha\in \Aut(F)_N$ as in the case $n=1$ to ensure $f_1\in \SAut(N)$. 
\end{proof}

\begin{Remark}
Apart from the last assertion about $\SAut$, the proof of Proposition~\ref{prop-decomposing-prod-aut} is not very specific to free groups. The important property used in the proof is that for every open subgroup $U$ the automorphism group of $U$ acts transitively on the set of normal subgroups of index $p$ in $U$. 
\end{Remark}

For the rest of the section we will consider $\Comm_p(\absF{})$ and $\Comm(\propF{p})$ separately. 
We start with the technically easier case of  $\Comm_p(F)$.  

\subsection{Generation of $\Comm_p(\absF{})$}\label{sec:generating}
Recall from subsection \ref{subsec-def-AComm} that {we defined $\AComm_p(F)$ as} the subgroup of $\Comm_p(F)$ generated by the subgroups $\Aut(H)$ where $H$ ranges over all $p$-open subgroups of $F$. The following result is an immediate consequence of Proposition~\ref{prop-decomposing-prod-aut}:

\begin{Corollary} \label{prop-Cp-generation}
$\Comm_p(\absF{}) = \AComm_p(\absF{})$.
 \end{Corollary}

\begin{Definition} Let us now define $\SComm_p(F)$ to be the subgroup of $\Comm_p(F)$ generated by $\SAut(H)$ where $H$ ranges over all $p$-open subgroups of $F$. 
\end{Definition}
Since $\SAut(F)$ has index 2 in $\Aut(F)$, the last assertion Proposition \ref{prop-decomposing-prod-aut} implies the following:

\begin{Corollary} \label{cor-C+-index-small}
The group $\SComm_p(\absF{})$ has index at most $2$ in $\Comm_p(\absF{})$.
\end{Corollary}

We do not know if $\SComm_p(\absF{})$ can actually be a proper subgroup of $\Comm_p(\absF{})$.
However, we can prove that $\SComm_p(\absF{}) = \Comm_p(\absF{})$ for some values of $p$ and $\rk(F)$:

\begin{Corollary} \label{prop-C+equalsC}
Suppose that either $p=2$, or that $p = 3 \mod 4$ and $d$ is even. Then $\SComm_p(\absF{d}) = \Comm_p(\absF{d})$.
\end{Corollary}
\begin{proof}
Choose any index $p$ normal subgroup $H$ of $F$.
By Lemma~\ref{lem-stab-not-contained-special-aut}(b), there exists $\alpha\in \Aut(F)\setminus\SAut(F)$
which stabilizes $H$ and such that $\alpha_{|H}\in\SAut(H)$. Thus the subgroup of $\Comm(F)$ generated by
$\SAut(F)$ and $\SAut(H)$ contains $\Aut(F)$. The equality  $\SComm_p(\absF{}) = \Comm_p(\absF{})$ now follows from 
Proposition~\ref{prop-decomposing-prod-aut}.
\end{proof}

\subsection{Generation of $\Comm(\propF{p})$}\label{subsec-generation-commfreeprop}

For $\Comm(\propF{p})$ the analogue of Corollary~\ref{prop-Cp-generation} again follows directly from Proposition~\ref{prop-decomposing-prod-aut}:
\skv

\begin{Proposition} 
\label{Comm=AComm}
$\Comm(\propF{p})=\AComm(\propF{p})$, that is,
$\Comm(\propF{p})$ is generated by subgroups of the form $\Aut(\propU{p})$ where $\propU{p}$ is open in $\propF{p}$.
\end{Proposition}

\begin{Notation}
Similarly to the case of free (abstract) groups, if $d=\rk(\propF{p})$, we define $\SAut(\propF{p})$ as the preimage of $\SL_d(\dbZ_p)$ under the \me{natural} epimorphism $\pi:\Aut(\propF{p})\to \Aut(\propF{p}/[\propF{p},\propF{p}]))\cong \GL_d(\dbZ_p)$.
\end{Notation}

Note that unlike the case of free groups, $\SAut(\propF{p})$ has infinite index in $\Aut(\propF{p})$.

\begin{Definition}
 We define $\SComm(\propF{p})$ to be the subgroup of $\Comm(\propF{p})$ generated by $\SAut(\propU{p})$ 
where $\propU{p}$ ranges over the open subgroups of $\propF{p}$.
\end{Definition}

We just proved that $[\Comm_p(F):\SComm_p(F)]\leq 2$ for a non-abelian free group $F$. In the pro-$p$ case we will establish a similar, albeit slightly weaker result:

\begin{Proposition}
\label{AComm=SComm2}
The subgroup $\SComm(\propF{p})$ is normal in $\Comm(\propF{p})$, and the group $\Comm(\propF{p})/\SComm(\propF{p})$ is a quotient of
$\dbZ/(p-1)\dbZ$.
\end{Proposition}

In order to prove Proposition~\ref{AComm=SComm2}, we will need to obtain a slightly different description of $\SComm(\propF{p})$ (Proposition \ref{prop:SCommalt}). To this end, we introduce some additional notations:

\begin{itemize}
	\item[(1)] Let $\IA(\propF{p})$ denote the kernel of $\pi:\Aut(\propF{p})\to \Aut(\propF{p}/[\propF{p},\propF{p}]))\cong \GL_d(\dbZ_p)$. It is called the Torelli subgroup. Note that $\IA(\propF{p})=\Aut(\propF{p};[\propF{p},\propF{p}])$ using our notation from subsection~\ref{subsec-prelim-auto}.
	\item[(2)] Let $\IA(\propF{p},p)=\Aut(\propF{p};\Phi(\propF{p}))$ be the kernel of
	$\pi_p:\Aut(\propF{p})\to \Aut(\propF{p}/\Phi(\propF{p})))\cong \GL_d(\dbF_p)$. It is often called the mod $p$ Torelli subgroup. 
	\item[(3)] Let $\SFAut(\propF{p})$ denote the preimage of $\SL_{d}(\dbF_p)$ under the map $\pi_p:\Aut(\propF{p})\to  \GL_d(\dbF_p)$.
\end{itemize}

Thus, we have the inclusions $\IA(\propF{p})\subset \IA(F,p)\subset \SFAut(\propF{p})$
and $\IA(\propF{p})\subset \SAut(\propF{p})\subset \SFAut(\propF{p})$. Note that the subgroups $\IA(\propF{p})$ and $\IA(\propF{p},p)$ are pro-$p$. This is true for $\IA(\propF{p},p)$ by Proposition~\ref{Aut:virtprop} and hence also for $\IA(\propF{p})$
as it is a closed subgroup of $\IA(\propF{p},p)$.
The group $\SFAut(\propF{p})$ is not pro-$p$, but is still generated by pro-$p$ elements
(see Lemma~\ref{lem:propelements} below). We  say that an element $g$ of a profinite group $G$ is a {\it pro-$p$ element} if $\overline{\la g\ra}$, the procyclic subgroup generated by $g$, is a pro-$p$ group.

\begin{Lemma} 
	\label{lem:propelements}
	$\SFAut(\propF{p})$ is precisely the subgroup of $\Aut(\propF{p})$ generated by the pro-$p$ elements.
\end{Lemma}
\begin{proof} Denote the subgroup of $\Aut(\propF{p})$ generated by the pro-$p$ elements by $\Aut^+(\propF{p})$. 
	Recall that $\IA(\propF{p},p)$ denotes the kernel of the map $\pi_p:\Aut(\propF{p})\to \GL_d(\dbF_p)$. Since $\IA(\propF{p},p)$ is a pro-$p$ group, $\Aut^+(\propF{p})$ contains $\IA(\propF{p},p)$.
	
	By definition of $\SFAut(\propF{p})$ we have $\SFAut(\propF{p})/\IA(\propF{p},p)\cong \SL_d(\dbF_p)$. 
	The group $\SL_d(\dbF_p)$ is generated by (non-trivial) elementary matrices, each of which
	has order $p$. Since $\IA(\propF{p},p)$ is pro-$p$, any lift of any of these elementary matrices
	to $\SFAut(\propF{p})$ is a pro-$p$ element. It follows that $\Aut^+(\propF{p})$ contains $\SFAut(\propF{p})$.
	
	On the other hand, $\SFAut(\propF{p})$ is the kernel of a homomorphism from $\Aut(\propF{p})$
	to $\dbZ/(p-1)\dbZ$. Since $\dbZ/(p-1)\dbZ$ has no non-trivial pro-$p$ elements, 
	any pro-$p$ element of $\Aut(\propF{p})$ must lie in this kernel, so $\Aut^+(\propF{p})\subseteq \SFAut(\propF{p})$. 
\end{proof}

As an immediate consequence of Lemma~\ref{lem:propelements}, we obtain the following result, which can be seen as a counterpart of Corollary~\ref{cor:SAut_char}:

\begin{Proposition}
	\label{prop:SFpinvariant} Let $\propU{p}$ be an open characteristic subgroup of $\propF{p}$,
	so that $\Aut(\propF{p})\subseteq \Aut(\propU{p})$ if we view $\Aut(\propF{p})$ and $\Aut(\propU{p})$ as subgroups of $\Comm(\propF{p})$. 
Then $\SFAut(\propF{p})\subseteq \SFAut(\propU{p})$.
\end{Proposition}

\begin{Remark}\rm 
	We do not know if Proposition~\ref{prop:SFpinvariant} remains true if we replace $\SFAut$ by $\SAut$.
\end{Remark}

\begin{Proposition}
\label{prop:SCommalt}
The subgroup $\SComm(\propF{p})$ contains $\SFAut(\propU{p})$
for every open subgroup $\propU{p}$ of $\propF{p}$. Hence $\SComm(\propF{p})$ is equal to the subgroup
generated by all $\SFAut(\propU{p})$.
\end{Proposition}

\begin{proof} 
By Proposition~\ref{prop:SFpinvariant}, it suffices to prove Proposition~\ref{prop:SCommalt}
assuming that $\propU{p}$ is a proper open subgroup.

Take any such $\propU{p}$ and choose any subgroup $\propV{p}$ containing $\propU{p}$ with $[\propV{p}:\propU{p}]=p$. 
By the pro-$p$ analogue of Lemma~\ref{lem-description-SCQ} (whose proof is identical to the abstract version),
there exists a basis $X=\{x_1,\ldots, x_m\}$ of $\propV{p}$ such that $\propU{p}$ has a basis 
$$Y=\{x_1^p\}\cup\{x_1^j x_i x_1^{-j}: 2\leq i\leq m, 0\leq j\leq p-1\}.$$
\skv
First we will assume that $p>2$. The case $p=2$ will require minor modifications and
will be handled at the end. The multiplicative group $\dbZ_p^{\times}$ is a direct product of  $1+p\dbZ_p$ and the subgroup of roots of unity of order coprime to $p$, which maps isomorphically onto $\dbF_p^{\times}$. Since $p>2$, we have $1+p\dbZ_p\cong\dbZ_p$, so $\dbZ_p^{\times}\cong \dbZ_p\times \dbZ/(p-1)\dbZ$ is procyclic. For the rest of the proof we fix a generator $\omega$ for $\dbZ_p^{\times}$.

Note that $\omega^{p-1}$ is a generator for the kernel of the projection $\dbZ_p^{\times}\to
\dbF_p^{\times}$, so $\SFAut(\propU{p})$ is precisely the set of $\phi\in \Aut(\propU{p})$ with 
$\det_{\propU{p}}(\phi)\in \overline{\la \omega^{p-1}\ra}$. As in the case of free groups, by $\det_{\propU{p}}$ of an element of $\Aut(\propU{p})$
we mean the determinant of its image in $\Aut(\propU{p}/[\propU{p},\propU{p}])$.

\skv
Let us now consider the automorphisms $\alpha$ and $\beta$ of $\propV{p}$ defined by 
$\alpha(x_1)=x_1^{\omega}$, $\alpha(x_i)=x_i$ for $i\neq 1$ and
$\beta(x_2)=x_2^{\omega}$, $\beta(x_i)=x_i$ for $i\neq 2$. Then $\det_{\propV{p}}(\alpha)=\det_{\propV{p}}(\beta)=\omega$,
so $\alpha^{-1}\beta\in \SAut(\propV{p})$. 
\skv
We claim that $\propU{p}$ is invariant under both $\alpha$ and $\beta$. Indeed, $\propU{p}$
is normally generated in $\propV{p}$ by the set $T=\{x_1^p,x_2,\ldots, x_m\}$. Clearly,
both $\alpha$ and $\beta$ preserve the procyclic subgroup $\overline{\la g \ra}$
for each $g\in T$ and hence preserve $\propU{p}$ as well. Thus we can consider
both $\alpha$ and $\beta$ as automorphisms of $\propU{p}$. Let us now compute $\det_{\propU{p}}$ for these automorphisms.
\skv
Note that $\beta$ raises $p$ elements of $Y$ (namely, the conjugates of $x_2$)
to the power $\omega$ and fixes the remaining elements of $Y$, so $\det_{\propU{p}}(\beta)=\omega^p$.

To compute $\det_{{\propU{p}}}(\alpha)$ we only need to know the elements of $\alpha(Y)$
modulo {$[\propU{p},\propU{p}]$. Let $r$ denote the projection of $\omega$ to $\dbF_p$. Then
$x_1^{\omega j}=x_1^{p\lam_j} x_1^{(rj\,\mathrm{mod}\, p)}$ for some $\lam_j\in\dbZ_p$.
Since $x_1^p\in Y$, the element $x_1^{p\lam_j}$ lies in $\propU{p}$.} Hence
for any $i\neq 1$
we have 
\begin{equation}
\label{eq:conjugation}
\alpha(x_1^j x_i x_1^{-j})=x_1^{\omega j}x_i\, x_1^{-\omega j}\equiv  x_1^{(rj\,\mathrm{mod}\, p)}x_i\, x_1^{-(rj\,\mathrm{mod}\, p)} \mod 
{[\propU{p},\propU{p}]}.
\end{equation}
Since $\omega$ is a generator
of $\dbZ_p^{\times}$, $r$ is a primitive root of unity
mod $p$, so for a fixed $i\neq 1$, the set $\{x_1^j x_i x_1^{-j}: 0\leq j\leq p-1\}$ coincides with
$\{x_i, x_1^{r}x_i x_1^{-r}, x_1^{r^2} x_i\, x_1^{-r^2},\ldots, x_1^{r^{p-1}}x_i\, x_1^{-r^{p-1}}\}$ mod {$[\propU{p},\propU{p}]$}. 
By \eqref{eq:conjugation}, $\alpha$ acts on the latter sequence, modulo {$[\propU{p},\propU{p}]$}, as a cyclic shift, and this shift is an even permutation since $p$ is odd. Since $\alpha(x_1^p)=(x_1^p)^{\omega}$, using Lemma~\ref{lem:absperfect}(a) we get $\det_{\propU{p}}(\alpha)=\omega$. 

Hence $\det_{{\propU{p}}}(\alpha^{-1}\beta)=\omega^{p-1}$, so the set 
$\overline{\la\alpha^{-1}\beta\ra}\SAut(\propU{p})$ consists of all $\phi\in\Aut(\propU{p})$ with 
$\det_{\propU{p}}(\phi)\in\overline{\la\omega^{p-1}\ra}$, and by an earlier remark this set is exactly
$\SFAut(\propU{p})$. Since $\alpha^{-1}\beta\in \SAut(\propV{p})$, we proved that $\SFAut(\propU{p})\subseteq \SComm({\propF{p}})$, as desired.
\skv
Let us now consider the case $p=2$. In this case $\dbZ_2^{\times}=\la -1 \ra\times C$ where $C\cong \dbZ_2$.
Choose a generator $\omega$ for $C$ and define $\alpha$ and $\beta$ as in the case $p>2$. The same computation
yields $\det_{{\propU{p}}}(\alpha^{-1}\beta)=\pm \omega$. The sign depends on the parity of $\rk(\propV{p})$, but we can replace
$\omega$ by $-\omega$ since $\la -1 \ra\times \overline{\la \omega\ra}=\la -1 \ra \times \overline{\la -\omega\ra}$,
so we can assume that $\det_{\propU{p}}(\alpha^{-1}\beta)=\omega$.
\skv
As in the case $p>2$, we deduce that $\SComm(\propF{p})$ contains all $\phi\in \Aut(\propU{p})$ with 
$\det_{\propU{p}}(\phi)\in \overline{\la \omega\ra}$. On the other hand, the computation from the proof
of Lemma~\ref{lem-stab-not-contained-special-aut}(b) for $p=2$ shows that $\SComm(\propF{p})$ contains some $\gamma\in \Aut(\propU{p})$ with $\det_{\propU{p}}(\gamma)=-1$.
Since $\la -1 \ra\times \overline{\la \omega\ra}=\dbZ_2^{\times}$, it follows that $\SComm(\propF{p})$ contains
the entire $\SFAut(\propU{p})$ (which in the case $p=2$ equals $\Aut(\propU{p})$).
\end{proof}

We are finally ready to prove Proposition~\ref{AComm=SComm2}.

\begin{proof}[Proof of Proposition~\ref{AComm=SComm2}]
First we prove that $\SComm(\propF{p})$ is normal in $\Comm(\propF{p})$.
By definition of $\SComm(\propF{p})$,
we just need to show
that $[\phi]\,\SAut(\propU{p})\,[\phi]^{-1} \subseteq \SComm(\propF{p})$ for any open subgroup $\propU{p}$ of $\propF{p}$ and $[\phi]\in \Comm(\propF{p})$. Let us now fix such $\propU{p}$ and $\phi$.

As in the proof of Lemma~\ref{lemma:Aut_1}(b), there exists an open subgroup
$\propW{p}$ which is characteristic in $\propU{p}$ such that $\phi$ is defined on $\propW{p}$,
and let $\propZ{p}=\phi(\propW{p})$. Then $\iota: \psi\mapsto \phi\psi\phi^{-1}$
is a continuous map from $\Aut(\propW{p})$ to $\Aut(\propZ{p})$ and therefore
sends pro-$p$ elements to pro-$p$ elements. Hence
by Lemma~\ref{lem:propelements} we have $\iota(\SFAut(\propW{p}))\subseteq\SFAut(\propZ{p})$.

Since $\SAut(\propU{p})\subseteq \SAut(\propW{p})$ by Proposition~\ref{prop:SFpinvariant},
 $\SAut(\propW{p})\subseteq \SFAut(\propW{p})$ (by definition) and 
$\SFAut(\propZ{p})\subseteq \SComm(\propF{p})$ by Proposition~\ref{prop:SCommalt}, it follows that
$$\iota(\SAut(\propU{p}))\subseteq \iota(\SFAut(\propW{p}))\subseteq\SFAut(\propZ{p})\subseteq\SComm(\propF{p}).$$ Thus, 
$[\phi]\,\SAut(\propU{p})\,[\phi]^{-1} \subseteq \SComm(\propF{p})$, as desired.
\skv
If $p=2$, we have $\SFAut(\propU{p})=\Aut(\propU{p})$ for any open subgroup $\propU{p}$,
so Proposition~\ref{prop:SCommalt} already implies that $\SComm(\propF{p})=\Comm(\propF{p})$.
Thus from now on we will assume that $p>2$.
\skv

Now let $Q=\Comm(\propF{p})/\SComm(\propF{p})$. As in the proof of Proposition~\ref{prop:SCommalt},
fix a generator $\omega$ of $\dbZ_p^{\times}$.
For each open subgroup $\propU{p}$ of $\propF{p}$ choose $\beta_{\propU{p}}\in \Aut({\propU{p}})$ with 
$\det_{\propU{p}}(\beta_{\propU{p}})=\omega$,
and let $b_{\propU{p}}$ be the image of $\beta_{\propU{p}}$ in $Q$. Note that $b_{\propU{p}}$ is independent of the choice of
$\beta_{\propU{p}}$ -- indeed, if $\beta'_{\propU{p}}$ is another element with $\det_{\propU{p}}(\beta'_{\propU{p}})=\omega$,
then $\beta'_{{\propU{p}}}\beta_{\propU{p}}^{-1}\in \SAut({\propU{p}})\subset \SComm(\propF{p})$.

Since $\beta_{\propU{p}}^{p-1}\in \SFAut(\propU{p})\subset \SComm(\propF{p})$
by Proposition~\ref{prop:SCommalt}, $b_{\propU{p}}$ has finite order (dividing $p-1$), so the image of the procyclic
subgroup $\overline{\la \beta_{\propU{p}} \ra}$ in $Q$ is the abstract group $\la b_{\propU{p}}\ra$ and
hence $Q$ is abstractly generated by all $b_U$.

Suppose now that $\propU{p}$ and $\propV{p}$ are open subgroups of $\propF{p}$ with $\propU{p}\subset \propV{p}$ and $[\propV{p}:\propU{p}]=p$.
Then the automorphism $\beta$ from the proof of Proposition~\ref{prop:SCommalt}
can be used as $\beta_{\propV{p}}$. By computation from that proof we have $\det_{\propU{p}}(\beta)=\omega^p$.
Hence $\beta=\beta_{\propV{p}}$ is congruent to $\beta_{\propU{p}}^p$ modulo $\SComm(\propF{p})$, so $b_{\propV{p}}=b_{\propU{p}}^p$. But
$b_{\propU{p}}^{p-1}=1$ as observed before, so $b_{\propU{p}}=b_{\propV{p}}$.

It follows that $b_{\propU{p}}=b_{\propV{p}}$ for any open subgroups ${\propU{p}}$ and ${\propV{p}}$ of $\propF{p}$. Thus $Q$
is generated by a single element of order dividing $p-1$, which finishes the proof.
\end{proof}

\section{Conjugacy of elements and subgroups} \label{sec-conjugacy}

As before, $\absF{}$ will denote a non-abelian free group of finite rank and $\propF{p}$ a non-abelian free pro-$p$ group of finite rank.  
Our main goal in this section is to classify the conjugacy classes of elements and (topologically) finitely generated closed subgroups of $\Comm(\absF{}), \Comm_p(\absF{})$ and $\Comm(\propF{p})$ that have representatives in $\absF{}$, respectively in $\propF{p}$.
Since in this section we will be discussing generation of abstract and profinite groups in the same setting, we will not follow our standard convention that generation means topological generation for profinite groups, to avoid confusion.

\subsection{Conjugacy of elements in $\Comm(F)$ and $\Comm(\propF{p})$}

Recall that we already proved in Section~\ref{sec-comm-abstract-free} that any two non-trivial elements of $F$ lie in the same conjugacy class
of $\Comm(F)$. In this subsection we will show that the same is true for the elements of $\propF{p}$ inside $\Comm(\propF{p})$; we will also
find a simple necessary and sufficient condition for two elements of $F$ to be conjugate in $\Comm_p(F)$. In order to apply these results later in the paper, it will be more convenient to deal with the conjugacy classes of a slightly larger collection of elements which we call {\it bounded}.

\begin{Definition} 
\label{def:bounded}
Let $G$ be one of the groups $\Comm(F)$, $\Comm_p(F)$ or $\Comm(\propF{p})$. An element $g\in G$ will be called {\it bounded}
if it belongs to some subgroup $H$ of $G$ which is open and
\begin{itemize}
\item free and commensurable with $F$ if $G=\Comm(F)$;
\item free pro-$p$ and commensurable with $\propF{p}$ if $G=\Comm(\propF{p})$;
\item free and $p$-commensurable with $F$ if $G=\Comm_p(F)$.
\end{itemize}
\end{Definition}
It is routine to check that in each case the set of bounded elements is conjugation-invariant.

\begin{Proposition}\label{prop-conjugacy-Comm}
Let $\calF=F$ or $\propF{p}$. Then any two non-trivial bounded elements of $\Comm(\calF)$
are conjugate in $\Comm(\calF)$.
\end{Proposition}

\begin{proof}

Let $g_1,g_2\in \Comm(\calF)$ be non-trivial bounded elements.

{\it Case 1: $g_1,g_2\in \calF$}. In this case the result has already been established in Proposition~\ref{prop-one-conj-class}. 
Note that while we only stated Proposition~\ref{prop-one-conj-class} for free groups, the result remains
true in the pro-$p$ case with the same proof thanks to Theorem~\ref{thm-free-factor}.
\skv

{\it Case 2: $g_1,g_2$ belong to the same free subgroup $H$ which is open and commensurable with $\calF$}. This case reduces to Case~1 because 
{$H$ and $\calF$ are virtually isomorphic and hence} we can identify $\Comm(H)$ with $\Comm(\calF)$ (as explained in Section~\ref{sec-prelim-general}).
\skv
{\it General case}. By assumption there exist free subgroups $H_1$ and $H_2$ of $\Comm(\calF)$ such that $g_i\in H_i$ for $i=1,2$
and $H_1$ and $H_2$ are both open and commensurable with $\calF$ and hence commensurable with each other, so in particular, $H_1\cap H_2$
is non-trivial. If $g$ is any non-trivial element of $H_1\cap H_2$, then applying Case~2 with $H=H_i$, we deduce
that $g$ and $g_i$ are conjugate in $\Comm(\calF)$  for $i=1,2$ and hence $g_1$ and $g_2$ are conjugate in $\Comm(\calF)$.
\end{proof}

\subsection{Conjugacy of elements in $\Comm_p(F)$ and applications.}

The main goal of this subsection is to establish when two bounded elements of $\Comm_p(F)$ are conjugate 
(see Proposition~\ref{prop-conjugacy-Comm-p}). In order to prove it we need a suitable modification of M.~Hall's free factor theorem dealing with pro-$p$ topology.

Let $K$ be a finitely generated subgroup of $F$. Recall that by Theorem~\ref{thm:Hallfreefactor}, $K$ is a free factor of some finite index subgroup $H$ of $F$. Moreover, \cite[Theorem~5.1]{Ha} implies that $K$ is an intersection of finite index subgroups of $F$ or, equivalently, an intersection of subgroups open in the profinite topology, and thus $K$ itself is closed in the profinite topology.

It is natural to ask how the picture changes if the profinite topology is replaced by the pro-$p$ topology. Recall that subgroups which are open (resp. closed) in the pro-$p$ topology are called $p$-open (resp. $p$-closed). Of course, it is not true that any finitely generated subgroup $K$ is a free factor of a $p$-open subgroup or that such $K$ is always $p$-closed; however, the last two conditions on $K$ turn
out to be equivalent. More generally, we have the following result of Ribes and Zalesski:

\begin{Proposition}[{See \cite[Corollary 3.3]{RZ}}]\label{prop-RZ-closed-iff-freefact}
Let $K$ be a finitely generated subgroup of $\absF{}$. The following hold:
\begin{enumerate}
	\item \label{RZ-closed-1} If $K$ is a free factor of a $p$-closed subgroup of $\absF{}$, then $K$ is $p$-closed.
	\item \label{RZ-closed-2} If $K$ is $p$-closed, then there is a $p$-open subgroup $H$ of $\absF{}$ such that $K$ is a free factor of $H$. 
\end{enumerate}
\end{Proposition}

We will now establish a sufficient condition for two elements of $F$ to be conjugate in $\Comm_p(F)$.
This special case will be a key step in the proof of the general conjugacy criterion (Proposition~\ref{prop-conjugacy-Comm-p}), but it will also be sufficient for the proof of virtual simplicity of $\Comm_p(F)$ 
in the next section (see Theorem~\ref{thm-simplicity-Commp}).

\begin{Lemma} 
\label{lem-easy-conjugacy-Commp}
Any two elements of $F$ which are not proper powers (in $F$) are conjugate in $\Comm_p(F)$.
\end{Lemma}

\begin{proof} Fix $g\in F$ which is not a proper power. Since any two primitive elements of $F$ lie in the same conjugacy class in
$\Aut(F)$ and hence in $\Comm_p(F)$, it suffices to prove that $g$ is $\Comm_p(F)$-conjugate to some primitive element $y\in F$.

Let $\propF{p}$ be the pro-$p$ completion of $F$ and $C$ the closure of $\la g\ra$ in $\propF{p}$. Then $C$
is abelian, so $C\cap F$ is also abelian and hence cyclic. Since $g$ is not a proper power, it follows
that $C\cap F=\la g\ra$, so $\la g\ra$ is $p$-closed in $F$. Hence by 
Proposition~\ref{prop-RZ-closed-iff-freefact}, there exists a $p$-open subgroup $U$ of $F$ 
such that $\la g\ra$ is a free factor of $U$ or, equivalently, $g$ is primitive for $U$.

\skv

Suppose that $[F:U]=p^n$. Choose any basis $X$ of $F$ and any $x\in X$, and let $V=F(X,x,p^n)$. Then 
$V$ is $p$-open, isomorphic to $U$, and any element $y\in X\setminus\{x\}$ lies in $V$ and is primitive for both $V$ and $F$.
Since $g$ is primitive for $U$, we can choose an isomorphism $f:U\to V$ such that $f(g)=y$, and the corresponding commensuration $[f]\in \Comm_p(F)$ conjugates $g$ to $y$, as desired.
\end{proof}

We now turn to the general case of the conjugacy problem for bounded elements in $\Comm_p(F)$. Given a bounded element $w\in\Comm_p(\absF{})\setminus\{1\}$, define 
\[
d_p(w):=[\overline{\langle w \rangle}: \langle w \rangle],
\]
where $\overline{\langle w \rangle}$ is the closure of $\langle w \rangle$ in $\Comm_p(F)$.
We will prove that $d_p(w)$ is a complete invariant for conjugacy of bounded elements (see Proposition~\ref{prop-conjugacy-Comm-p}). But first we will obtain a more explicit description for $d_p(w)$.

\begin{Lemma}
\label{lem:dpexplicit}
Let $w\in \Comm_p(\absF{})\setminus\{1\}$ be a bounded element, and let $H$ be a free subgroup of $\Comm_p(\absF{})$ which contains $w$
and is $p$-commensurable with $F$ (such $H$ exists by the definition of a bounded element). Write $w=v^{m}$ where $m\in\dbN$ and
$v\in H$ is not representable as a proper power (in $H$). Then $d_p(w)$ is the largest divisor of $m$ which is coprime to $p$. 
\end{Lemma}
\begin{proof} First note that being $p$-commensurable with $F$, the subgroup
$H$ is open (and hence closed) in $\Comm_p(\absF{})$. Hence $\overline{\langle w \rangle}$ 
is contained in $H$ and is equal to the closure of $\langle w \rangle$ in $H$. Since $\overline{\langle w \rangle}$ is abelian, it is contained in
$C_H(w)$, the centralizer of $w$ in $H$. Since $H$ is free, $C_H(w)$ is cyclic and thus is equal to $\la v\ra$ for $v$ in the statement
of Lemma~\ref{lem:dpexplicit}. On the other hand, since $H$ is Hausdorff, $C_H(w)$ is closed in $H$ (and hence in $\Comm_p(\absF{})$),
so  $\overline{\langle w \rangle}$ is the closure of $\langle w \rangle=\la v^m\ra$ in $C_H(w)=\la v\ra$. Since the induced topology
on $\la v\ra$ is the pro-$p$ topology, this closure is equal to $\la v^e\ra$ where $e$ is the largest power of $p$ dividing $m$,
so $d_p(w)=[\la v^e\ra: \la v^m\ra]=\frac{m}{e}$, which is the largest divisor of $m$ coprime to $p$, as desired.
\end{proof}

We are now ready to prove that the number $d_p(w)$ is a complete invariant for conjugacy of bounded elements.

\begin{Proposition}\label{prop-conjugacy-Comm-p}
	Let $w_1,w_2 \in \Comm_p(\absF{})$ be non-trivial bounded elements.  Then $w_1$ is conjugate to $w_2$ in $\Comm_p(\absF{})$ if and only if $d_p(w_1) = d_p(w_2)$.
\end{Proposition}

\begin{proof}
	The topology on $\Comm_p(\absF{})$ is conjugation-invariant, so if $w_1$ is conjugate to $w_2$ then certainly $d_p(w_1) = d_p(w_2)$.  

Let us now prove the backwards direction. Fix bounded elements $w_1,w_2 \in \Comm_p(\absF{})$ with $d_p(w_1) = d_p(w_2)$
and choose free subgroups $H_i$, $i=1,2$ such that $H_i$ is $p$-commensurable with $F$ and $w_i\in H_i$. We need to show that
$w_1$ and $w_2$ are conjugate in $\Comm_p(\absF{})$.
\skv

{\it Special case: $H_1=H_2$.} Since $H_1=H_2$ is $p$-commensurable with $F$, the pro-$p$ topologies of $H_1$ and $\absF{}$ are compatible and we can identify $\Comm_p(H_1)$ with $\Comm_p(\absF{})$. Thus, without loss of generality we can assume that $H_1=H_2=F$, so
$w_1,w_2\in \absF{}$.

If $d=d_p(w_1) = d_p(w_2)$, then by Lemma~\ref{lem:dpexplicit} for $i=1,2$ we can write
$w_i=v_i^{p^{e_i}d}$ where $v_1,v_2\in \absF{}$ are not proper powers.  Choose any primitive element $x\in F$.
By Lemma~\ref{lem-easy-conjugacy-Commp}, $v_1$ and $v_2$ are both conjugate to $x$ in $\Comm_p(F)$. On the other hand,
by Lemma~\ref{lem-conjugate-power-SmF}, $x$ is conjugate in $\Comm_p(F)$ to $x^p$ and hence also to $x^{p^e}$ for all $e\in\dbN$.
It follows that $w_1$ and $w_2$ are both conjugate to $x^{d}$.

\skv
{\it General case:}	The subgroups $H_1$ and $H_2$ are both $p$-commensurable with $F$ and hence $p$-commensurable with each other. Hence
there exists $n\in\dbN$ such that $w_i^{p^n}\in H_1\cap H_2$ for $i=1,2$. From Lemma~\ref{lem:dpexplicit} it is clear that $d_p(w^{p^e})=d_p(w)$
for any bounded element $w$ and $e\in\dbN$. Thus, $d_p(w_1^{p^n})=d_p(w_1)=d_p(w_2)=d_p(w_2^{p^n})$. On the other hand, each pair
$(w_1,w_1^{p^n})$, $(w_1^{p^n},w_2^{p^n})$, $(w_2,w_2^{p^n})$ lies in some free subgroup $p$-commensurable with $\absF{}$, so by the special case
the elements in each pair are conjugate in $\Comm_p(F)$, and hence $w_1$ and $w_2$ are conjugate in $\Comm_p(F)$.
\end{proof}

We point out an important special case of Proposition \ref{prop-conjugacy-Comm-p} which describes when two powers of a bounded element
are conjugate in $\Comm_p(F)$:

\begin{Corollary}\label{cor:p-power-conjugation}
Let  $w \neq 1$ be a bounded element of $\Comm_p(\absF{})$ and $m,n \in \Z$ non-zero integers.  Then $w^m$ is conjugate to $w^n$ in $\Comm_p(\absF{})$ if and only if $|m/n| = p^r$ for some $r \in \Z$.
\end{Corollary}

\begin{proof}
	If we write $n = p^\alpha n'$ and $m = p^\beta m'$ with $n',m'$ coprime to $p$, then $d_p(w^n) = |n' |d_p(w)$ and $d_p(w^m) = |m' |d_p(w)$. Hence the assertion follows from Proposition \ref{prop-conjugacy-Comm-p}.
\end{proof}

Examples of pairs $(w,c)$ with $w \in F$ and $c\in \Comm_p(F)$ such that $cw^pc^{-1} = w^{p^2}$ explicitly appear in \cite{BouRabee-Young}. 

\begin{Remark}\label{rem:BS}
Specializing the result of Corollary~\ref{cor:p-power-conjugation} further, we deduce that for any non-trivial bounded element
$w$ of $\Comm_p(F)$ there exists $c\in \Comm_p(F)$ such that $cwc^{-1} = w^p$, so the subgroup $\Gamma=\la c,w\ra$ of $\Comm_p(F)$
is a quotient of the Baumslag--Solitar group $\mathrm{BS}(1,p)=\la a,b \mid bab^{-1}=a^p\ra$. In fact, the subgroup $\Gamma$
is isomorphic to $\mathrm{BS}(1,p)$ since $w$ has infinite order in $\Gamma$ (being an element of a free subgroup), while 
the image of $a$ in every proper quotient of $\mathrm{BS}(1,p)$ is finite -- the latter follows from the fact that
$BS(1,p)\cong \Z[1/p] \rtimes \Z$ where $a$ maps to $1\in \Z[1/p]$. 
\end{Remark}

We finish this subsection with a simple applications of Proposition~\ref{prop-conjugacy-Comm-p}.
The following result is well known:

\begin{Lemma}[\cite{Schutzenberger}] 
\label{lem-existence-nonpower}
Let $w\in \absF{}$ be a non-trivial commutator. Then $w$ is not a proper power in $ \absF{}$. In particular,
any non-abelian subgroup of $\absF{}$ contains an element that is not a power.
\end{Lemma}

\begin{Corollary} \label{cor-non-ab-intersects-conj-class}
Let $\Lambda$ be a non-abelian subgroup of $\absF{}$. Then $\Lambda$ intersects the conjugacy class of every bounded element.
\end{Corollary}

\begin{proof}
Lemma \ref{lem-existence-nonpower} provides $w \in \Lambda$ such that $d_p(w) = 1$. Let $v$ be any bounded element of $\Comm_p(\absF{})$
and $n=d_p(v)$. Then $n$ is coprime to $p$ by Lemma~\ref{lem:dpexplicit}, so $d_p(w^n)=n$ and hence $v$ is conjugate to $w^n\in\Lambda$ by Proposition \ref{prop-conjugacy-Comm-p}. 
\end{proof}

\subsection{Conjugacy of finitely generated subgroups}

As before, let $F$ be a free group of finite rank and $\propF{p}$ a free pro-$p$ group of finite rank.
In this subsection we will characterize when two finitely generated subgroups of $F$ are conjugate in $\Comm(F)$,
when two finitely generated $p$-closed subgroups of $F$ are conjugate in $\Comm_p(F)$ and when two finitely generated
closed subgroups of $\propF{p}$ are conjugate in $\Comm(\propF{p})$.

Similarly to the corresponding results for the conjugacy of elements, in each case we will apply a suitable form
of the free factor theorem. We will also need two additional results on the structure of subgroups
of free products containing one of the factors.

\begin{Lemma}\label{lem:intermediate-free-factor}
Let $A$ and $B$ be groups, let $G = A * B$, and let $H$ be a subgroup of $G$ containing $A$. Then $H = A * K$ for some group $K$.
\end{Lemma}

\begin{proof} By the Kurosh subgroup theorem, as formulated, e.g., in \cite[Ch.2, Theorem~19.1]{Bogo-book}, any subgroup $H$ of $G$ decomposes
as a free product of a free group and subgroups of the form $H\cap xAx^{-1}$ and $H\cap yBy^{-1}$ where $x$ (resp. $y$)
ranges over a set of representatives of the double cosets $H\backslash G/A$ (resp. $H\backslash G/B$). Since we can always
let $x=1$ to be one of the representatives, we deduce that $H\cap A$ is a free factor of $H$,
which yields the assertion of the lemma.
\end{proof}

In the pro-$p$ case, we will use a more specialized result, dealing only with subgroups of free pro-$p$ groups:

\begin{Lemma}\label{lem:intermediate-free-factor:pro-p}
Let $H$ be a free pro-$p$ factor of $\propF{p}$.  Then $H$ is also a free pro-$p$ factor of every open subgroup of $\propF{p}$ that contains $H$.\end{Lemma}

Lemma~\ref{lem:intermediate-free-factor:pro-p} is an immediate consequence of the main theorem of \cite{BNW}.
\skv

We are now ready to state and prove our criterion for the conjugacy of subgroups.

\begin{Proposition}\label{prop-bounded-subgroup-conjugacy}
Let	$\calF=F$ or $\propF{p}$. Let $C$ be either $\Comm(\calF)$ or $\Comm_p(\calF)$
if $\calF=F$ and $C=\Comm(\calF)$ if $\calF=\propF{p}$.  Let $H_1$ and $H_2$ be closed subgroups of $C$ which are contained in $\calF$ and finitely generated (topologically finitely generated if $\calF=\propF{p}$). Then $H_1$ and $H_2$
are conjugate in $C$ if and only if
	\begin{enumerate}
		\item $\rk(H_1) = \rk(H_2)$ and
		\item if one of $H_1$ and $H_2$ is open in $C$, then so is the other. 
	\end{enumerate}
\end{Proposition}

\begin{proof}
The `only if' direction is clear, so we only need to prove the converse. Thus, assume that $H_1$ and $H_2$ satisfy all the hypotheses
in Proposition~\ref{prop-bounded-subgroup-conjugacy} including (1) and (2). Condition (1) ensures that there exists an isomorphism 
of topological groups $\phi: H_1 \rightarrow H_2$.  If $H_1$ and $H_2$ are open then $\phi$ induces a virtual isomorphism of $\absF{}$ in the appropriate topology, which corresponds to an element $[\phi]$ of $C$, and then we have $[\phi]H_1[\phi^{-1}] = H_2$.

	From now on we assume that neither $H_1$ nor $H_2$ is open.  Then there are open subgroups $K_1,K_2$ of $C$ contained in $\calF$ such that $K_i$ can be written as a free product $K_i = H_i * M_i$ (in the suitable category) for $i=1,2$.  This holds by
Theorem~\ref{thm:Hallfreefactor} if $C=\Comm(F)$, by Theorem~\ref{thm-free-factor} if $C=\Comm(\propF{p})$ and
by Proposition~\ref{prop-RZ-closed-iff-freefact}(2) if $C=\Comm_p(F)$ (note that in the latter case our requirement is that
$K_i$ are $p$-open in $F$).
	
Note that the indices $m_i=[K_i:K_1 \cap K_2]$ are finite.  We  have natural projections $\pi_i: K_i \rightarrow M_i$ where the kernel is the normal closure of $H_i$.  Since $H_i$ is closed but not open, the groups $M_i$ are themselves free 
(abstract or pro-$p$) of some nonzero finite rank.  We can then find a subgroup $L_i$ of $K_i$ which is the preimage under $\pi_i$ of a normal subgroup of $M_i$ of index $m_i$: in the cases $C = \Comm_p(\absF{})$ and $C = \Comm(\propF{p})$ note that $m_i$ is a power of $p$, so we can choose $L_i$ to be $p$-open in $K_i$.
	
	  It follows that $L_i$ and $K_1 \cap K_2$ have the same index in $K_i$, and hence $\rk(L_1) = \rk(L_2)$.  From the open case of the proposition, we know that there exists $g \in C$ such that $gL_1g^{-1} = L_2$, so we may assume $L_1 = L_2$.  Since $H_i$ is a free factor of $K_i$ for $i=1,2$, it is also a free factor of $L_1$, by Lemma~\ref{lem:intermediate-free-factor} if $\calF=F$ and by 
Lemma~\ref{lem:intermediate-free-factor:pro-p} if $\calF=\propF{p}$.  In turn, all free factors of $L_1$ of a given rank belong to the same $\Aut(L_1)$-orbit, so there is $h \in \Aut(L_1)$ which sends $H_1$ to $H_2$. In particular, $H_1$ and $H_2$ are conjugate in $C$.
\end{proof}

\section{Simple subgroups of the commensurator of a free pro-$p$ group}  \label{sec-comm-p-simple}

\subsection{Simplicity of $\SComm_p(\absF{})$}

Let $F$ be a non-abelian free group of finite rank. Recall that $\SComm_p(F)$ is the subgroup of 
$\Comm_p(F)$ generated by all subgroups $\SAut(H)$ where $H$ ranges over $p$-open subgroups of $F$.
In this subsection we will prove that $\SComm_p(\absF{})$ is simple; in fact, we will establish 
a slightly stronger statement  -- see Theorem~\ref{thm-simplicity-Commp} below.
Since $\SComm_p(F)$ has index at most $2$ in $\Comm_p(F)$ by Corollary~\ref{cor-C+-index-small}, this will imply Theorem~\ref{ThmC}.

\begin{Theorem} \label{thm-simplicity-Commp}
Every non-trivial subgroup of  $\Comm_p(\absF{})$ normalized by $\SComm_p(\absF{})$ contains  $\SComm_p(\absF{})$.
\end{Theorem}

The proof of Theorem~\ref{thm-simplicity-Commp} will follow the same general outline as that of Theorem~\ref{thmAextended}, except this time we use Lemma~\ref{lem-easy-conjugacy-Commp} instead of Proposition~\ref{prop-one-conj-class}.

\begin{proof}[Proof of Theorem~\ref{thm-simplicity-Commp}]

First, since $[\Comm_p(\absF{}):\SComm_p(\absF{})]\leq 2$ by Corollary~\ref{cor-C+-index-small}, arguing exactly as in the proof of Theorem~\ref{thm-fg-simple-general}, we reduce Theorem~\ref{thm-simplicity-Commp} to the case of normal subgroups, so let $N$ be a non-trivial normal subgroup of $\Comm_p(F)$.

By Lemma \ref{lem-Comm-trivialQZ} (applied to $F$ equipped with the pro-$p$ topology), the subgroup $M = N \cap F$ is non-trivial.	
Since $\Comm_p(F)$ contains $\Aut(F)$, the group $M$ is also characteristic in $F$, so in particular, it contains an element that is not a proper power (in $F$) {by Lemma~\ref{lem-existence-nonpower}}. Since $N$ is normal in $\Comm_p(\absF{})$, by Lemma~\ref{lem-easy-conjugacy-Commp} $M$ contains a primitive element of $\absF{}$. Since $M$ is characteristic in $F$, we must have $M = \absF{}$, so $N$ contains $\absF{}$. 
	
	Let now $H$ be a proper $p$-open subgroup of $F$. Since $\rk(H)\geq 3$ and $H \subseteq N$, Proposition~\ref{prop-int-implies-saut} is applicable and implies that $N$ contains $\SAut(K)$ for every index $p$ normal subgroup $K$ of $H$. Applying Lemma~\ref{lemma-generation-indexp} to $H$, we deduce that $N$ contains $\SAut(H)$. Applying Lemma~\ref{lemma-generation-indexp} again, this time to $F$, we see that $N$ also contains $\SAut(F)$. This shows {that} $N$ contains $\SComm_p(\absF{})$, as desired. 	
\end{proof}

\subsection{A simple locally free pro-$p$ group associated to $\Comm_p(F)$} \label{sec-tdlc}

Let $\absF{}$ be a non-abelian free group of finite rank and $\propF{p}$ a free pro-$p$ group
of the same rank. As before, we will identify $\propF{p}$ with the pro-$p$ completion of $F$.
Recall that the associated homomorphism $ \Comm_p(\absF{}) \to \Comm(\propF{p})$ defined in Lemma \ref{lem-localiso-extends-completion} is injective, so we will view $\Comm_p(\absF{})$ as a subgroup of $\Comm(\propF{p})$. Likewise we will view $\Aut(F)$ as a subgroup
of $\Aut(\propF{p})$.
\skv

Thus, we have already found one natural (abstractly) simple subgroup of $\Comm(\propF{p})$, namely $\SComm_p(\absF{})$.  In this subsection we establish that the closure of $\SComm_p(\absF{})$ in $\Comm(\propF{p})$ is also abstractly simple.

\begin{Definition} \label{defi-closure-Cp(Fk)}
	We denote by $\CpC_p(\absF{})$ (resp.\ $\SCpC_p(\absF{})$) the closure of $ \Comm_p(\absF{})$ (resp.\ $\SComm_p(\absF{})$)  in $\Comm(\propF{p})$. 
\end{Definition}

Clearly $\SCpC_p(\absF{}) \subseteq  \CpC_p(\absF{})$, and by Corollary \ref{cor-C+-index-small}, the subgroup $\SCpC_p(\absF{})$ has index at most $2$ in $\CpC_p(\absF{})$. Since $\absF{}$ is dense in $\propF{p}$ and $\SComm_p(\absF{})$ contains $\absF{}$,  the subgroup $\SCpC_p(\absF{})$ contains $\propF{p}$. In particular, $\SCpC_p(\absF{})$ and $ \CpC_p(\absF{})$ are open subgroups of $\Comm(\propF{p})$,
and we have the following alternative descriptions of those groups:

\begin{Observation}\label{lemma-Scp-explicit}
$\CpC_p(\absF{})$ (resp. $\SCpC_p(\absF{})$) is the subgroup of $\Comm(\propF{p})$ generated by $\propF{p}$ and the groups 
$\Aut(U)$ (resp. $\SAut(U)$) where $U$ ranges over $p$-open subgroups of $F$.
\end{Observation}

Note also that since $\Comm_p(\absF{})$ is countable, the groups $\SCpC_p(\absF{})$ and $ \CpC_p(\absF{})$ are $\sigma$-compact. 

\skv
Before proving abstract simplicity of $\SCpC_p(\absF{})$ we need some preparations. First, we will use two well-known results dealing with abstract normal subgroups of finitely generated pro-$p$ groups.
The first result is derived in the course of the proof of \cite[Proposition~1.19]{DDMS}:

\begin{Lemma}
	\label{DDMS1}
	Let $G$ be a pro-$p$ group generated by a finite set $X_0$. Then the commutator subgroup $[G,G]$ is abstractly generated by 
	$\{[x,g]: x\in X_0, g\in G\}$.
\end{Lemma}

We do not know a reference in the literature for the second result, so we will include a proof.

\begin{Lemma} [folklore]
	\label{abstractperfectquotient}
	Let $G$ be a finitely generated pro-$p$ group and $N$ an abstract normal subgroup of $G$.
	If $G/N$ is perfect, then $N=G$.   
\end{Lemma}

\begin{proof}
	Since $G/N$ is perfect, we have $[G,G]N=G$. In particular, $N$ surjects onto $G/\Phi(G)$ and hence contains some generating set $X$ for $G$ by Lemma \ref{lem:Frattinibasic}. By Lemma~\ref{DDMS1}, the commutator subgroup $[G,G]$ is abstractly generated by $\{[x,g]: x\in X, g\in G\}$, so $N$ contains $[G,G]$. Hence $G/N$ is both perfect and abelian and thus trivial.
\end{proof}

The next two lemmas deal with subgroups of $\Comm(\propF{p})$ which have a sufficiently large normalizer.
The first one establishes a dichotomy for subgroups normalized by $\SAut(F)$.

\begin{Lemma}
\label{Frattini argument}
Let $N$ be a subgroup of $\Comm(\propF{p})$ normalized by $\SAut(F)$ and let $K=N\cap \propF{p}$.
Then either $K\subseteq \Phi(\propF{p})$ or $K\Phi(\propF{p})=\propF{p}$.
\end{Lemma}

\begin{proof}
The group $K=N\cap \propF{p}$ is $\SAut(F)$-invariant
and hence so is the quotient $K/(K\cap \Phi(\propF{p}))\cong K\Phi(\propF{p})/\Phi(\propF{p})$.
On the other hand, $\SAut(F)$ acts transitively on nonzero elements of $\propF{p}/\Phi(\propF{p})$ since this action factors through
the standard action of $\SL_d(\mathbb F_p)$ on $\mathbb F_p^d\setminus\{0\}$ where $d=\rk(F)$. It follows that either
$K\Phi(\propF{p})=\propF{p}$ or $K\Phi(\propF{p})=\Phi(\propF{p})$, and in the latter case $K\subseteq \Phi(\propF{p})$.
\end{proof}

The next lemma yields a stronger conclusion for subgroups normalized by $\la \propF{p},S_p(\absF{}) \ra$.

\begin{Lemma} \label{prop-dense-normal}
Let $N$ be a subgroup of $\Comm(\propF{p})$ normalized by $\la \propF{p},S_p(\absF{}) \ra$,  and assume that
$N\cap \,\propF{p}\not \subseteq \Phi(\propF{p})$. Then $N$ contains $\propF{p}$. 
\end{Lemma}

\begin{proof}
As in Lemma~\ref{Frattini argument}, we set $K=N\cap \propF{p}$, so that $K\not\subseteq \Phi(\propF{p})$
by our hypotheses. Hence $K\Phi(\propF{p})=\propF{p}$ by Lemma~\ref{Frattini argument}. 
Let $X$ be a basis for $\absF{}$. Then $K$ has non-trivial intersection with the coset $\Phi(\propF{p})x$ for each $x \in X$, so $K$ contains a generating set $X_0$ for $\propF{p}$ by Lemma~\ref{lem:Frattinibasic}(i). Since $K$ is normal in $\propF{p}$, it must contain $[\propF{p},\propF{p}]$ by Lemma~\ref{DDMS1}. Since $X$ generates $\propF{p}$,
we have \[\propF{p}=[\propF{p},\propF{p}]\cdot \prod\limits_{x\in X}\overline{\la x\ra}, \] and thus it remains to show that $K$ contains $\overline{\la x\ra}$ for each $x\in X$.
\skv

Fix $x\in X$, and let $H$ be any normal subgroup of index $p$ in $\absF{}$ containing $x$. By Lemma~\ref{lem:Schreier}(b), $H$
has a free generating set $Y$ containing $x$ and some element $y\in [H,H]$.
Since $\rk(H)\geq 3$, there exists $\sigma\in \SAut(H)$ which permutes the elements of $Y$ and {sends $x$ to $y$}. Since $N$ is normalized by $S_p(F)$, we have
$$\overline{\la x\ra}=\sigma^{-1}\overline{\la y\ra}\sigma \subseteq \sigma^{-1}[\propF{p},\propF{p}]\sigma \subseteq \sigma^{-1}N\sigma=N.$$
Hence $\propF{p} \subseteq N$.
\end{proof}

We are now ready to prove that $\SCpC_p(\absF{})$ is abstractly simple. 
As with the analogous result for $\AComm(F)$, we will establish a stronger statement:

\begin{Theorem} \label{thm-completion-abs-simple}
Let $\absF{}$ be a non-abelian free group of finite rank and $\propF{p}$ its pro-$p$ completion. Then every non-trivial subgroup of $\Comm(\propF{p})$ normalized by $\SCpC_p(\absF{})$ contains  $\SCpC_p(\absF{})$. 
In particular, $\SCpC_p(\absF{})$ is abstractly simple.
\end{Theorem}

\begin{proof}
Let $N$ be a non-trivial subgroup of $\Comm(\propF{p})$ normalized by $\SCpC_p(\absF{})$. We want to show that $\SCpC_p(\absF{}) \subseteq N$. Since $\SCpC_p(\absF{})$ is open in $\Comm(\propF{p})$ and normalizes $N$, the subgroup $N \cap \propF{p}$  is non-trivial by 
Lemma \ref{lem-Comm-trivialQZ}(\ref{item-Comm-local-normal}).
	
By Proposition~\ref{prop:Frattini}(b), the Frattini series $(\Phi^{n}(\propF{p}))_{n=0}^{\infty}$ has trivial intersection,
so $N \cap \propF{p}$  cannot be contained in all its terms. Let $\propU{p}= \Phi^{\ell}(\propF{p})$ be the last term that contains $N \cap \propF{p}$.

As usual, we can identify $\Comm(\propU{p})$ with $\Comm(\propF{p})$. 
Also, if we let $U=\propU{p}\cap F$, then $\propU{p}$ can be viewed as the pro-$p$ completion of $U$ and we have
$\la \propU{p}, S_p(U)\ra\subseteq \SCpC_p(\absF{})$, so $N$ is normalized by $\la \propU{p}, S_p(U)\ra$. 
Moreover, by the choice of $\ell$ we have $N\cap \propU{p}=N\cap \propF{p}\not\subseteq \Phi(\propU{p})$,
so we can apply Lemma~\ref{prop-dense-normal} to $N$ with $U$ and $\propU{p}$ playing the role of $F$ and $\propF{p}$, respectively.
Thus we deduce that $N$  contains
$\propU{p}$; in particular $N\cap\, \SComm_p(\absF{})$ is non-trivial.
Since $\SComm_p(\absF{})$ is simple and normalizes $N$, it follows that $N$ contains $\SComm_p(\absF{})$ and hence
contains $\la \propU{p}, \SComm_p(\absF{}) \ra=\SCpC_p(\absF{})$, as desired (the last equality holds since
$\la \propU{p}, \SComm_p(\absF{}) \ra$ is both open and dense in $\SCpC_p(\absF{})$).
\end{proof}

\subsection{On the monolith of $\Comm(\propF{p})$}
As in previous subsection, let $\absF{}$ be a non-abelian free group of finite rank and  $\propF{p}$ its  pro-$p$ completion. 
Theorem~\ref{Theorem-intro-Comm-pro-p} restated as Corollary~\ref{Comm-prop-almost-simple} below is an immediate consequence of Theorem~\ref{thm-completion-abs-simple}.

\begin{Corollary}\label{Comm-prop-almost-simple}
	The group $\Comm(\propF{p})$ is monolithic and its monolith is simple.  Moreover, $\mon(\Comm(\propF{p}))$ is equal to the (abstract) normal closure of $\SCpC_p(\absF{})$ in $\Comm(\propF{p})$. 
\end{Corollary}

\begin{proof}
Let $N$ denote the normal closure of $\SCpC_p(\absF{})$ in $\Comm(\propF{p})$.
Theorem~\ref{thm-completion-abs-simple} implies that $\Comm(\propF{p})$ is monolithic and its monolith $M = \mon(\Comm(\propF{p}))$  
contains $\SCpC_p(\absF{})$ and hence contains $N$. On the other hand, since $N$ is normal and non-trivial, it must contain $M$, so $N=M$.

Since $M$ contains $\SCpC_p(\absF{})$, we can apply Theorem~\ref{thm-completion-abs-simple} again, now to $M$, and deduce that $M$ is also monolithic. The monolith of $M$ is characteristic in $M$ and hence normal in $\Comm(\propF{p})$. Hence $\mon(M)=M$, so $M$ is simple.
\end{proof}

Corollary~\ref{Comm-prop-almost-simple} does not give us much information about the size of the monolith of $\Comm(\propF{p})$.
In particular, we do not know whether $\mon(\Comm(\propF{p}))$ is equal to $\Comm(\propF{p})$ or at least has finite index in 
$\Comm(\propF{p})$. Nevertheless, we will prove that $\mon(\Comm(\propF{p}))$ is substantially larger than
$\SCpC_p(\absF{})$ and, in particular, the index $[\mon(\Comm(\propF{p})):\propF{p}]$ is uncountable 
(see Proposition~\ref{Comm-prop-contains-A-closure} below).

In the remainder of this section we consider the group $\Aut(\propF{p})$ with the $A$-topology defined in subsection~\ref{subsec-prelim-auto-general}, which makes $\Aut(\propF{p})$ a profinite group. Note that the $A$-topology is \textit{not} the restriction to $\Aut(\propF{p})$ of the topology on $\Comm(\propF{p})$: while $\propF{p}$ is an open subgroup of $\Comm(\propF{p})$,  
it is not open in the $A$-topology on $\Aut(\propF{p})$. However $\propF{p}$ is closed in $\Aut(\propF{p})$ with respect to the $A$-topology. Recall also that we identify $\Aut(\absF{})$ with its image in $\Aut(\propF{p})$. 
\skv

\begin{Notation}
Denote by $SA(F)$ the closure of $\SAut(\absF{})$ in $\Aut(\propF{p})$ with respect to the $A$-topology. 
\end{Notation}
The following key lemma provides a substantial restriction on the structure of proper abstract normal subgroups of $SA(F)$. 

\begin{Lemma}  \label{lem:abstractnormalclosure}
Assume that $\rk(F)\geq 3$, and let $K$ be an abstract normal subgroup of $SA(F)$ whose image in $\Aut(\propF{p}/\Phi(\propF{p}))\cong \GL_d(\dbF_p)$ contains a non-scalar matrix. Then $K=SA(F)$.
\end{Lemma}

\begin{proof}
Let $d=\rk(F)$ and $\pi:\Aut(\propF{p}) \to \GL_d(\dbF_p)$ the projection  homomorphism. Also let $G=SA(F)$.
It is clear that $\pi$ is continuous with respect to the $A$-topology on $\Aut(\propF{p})$ and discrete topology on $\GL_d(\dbF_p)$.
Since $\pi(\SAut(\absF{})) = \SL_d(\dbF_p)$, we have $\pi(G) = \SL_d(\dbF_p)$ as well. 
Hence $\pi(K)$ is a normal subgroup of $\SL_d(\dbF_p)$, which is not contained in the center by assumption,
so $\pi(K) = \SL_d(\dbF_p)=\pi(G)$. Thus, if we set $P=\Ker\pi\cap G$, then $G = PK$. 
Also note that $P$ is a  pro-$p$ group by Proposition \ref{Aut:virtprop}. 
\skv

Since $\SAut(\absF{})$ is a finitely generated group which is perfect by Lemma~\ref{lem:absperfect}(c), 	$G$ is a finitely generated profinite group which is topologically perfect. By a fundamental theorem of Nikolov and Segal, the commutator subgroup of a finitely generated profinite group is closed (this is a special case of \cite[Theorem~1.4]{NS1}), so $G$ is perfect (as an abstract group), and hence so is $G/K$. 

The equality $G = PK$ implies that $G/K$ is a quotient of $P$. Since $G$ is finitely generated and $P$ is open in $G$ (as $\Ker\pi$ is open
in $\Aut(\propF{p})$), $P$ is also finitely generated. But a finitely generated
	pro-$p$ group cannot have a non-trivial perfect abstract quotient by Lemma~\ref{abstractperfectquotient}. Therefore $G/K$ is trivial.
\end{proof}

Using Lemma~\ref{lem:abstractnormalclosure}, we can now construct another simple subgroup of $\Comm(F)$
strictly containing $\SCpC_p(\absF{})$.
For each $p$-open subgroup $U$ of $F$ let $\propU{p}$ be its closure in $\propF{p}$, and  
let $SA(U)$ denote the closure of $\SAut(U)$ in $\Aut(\propU{p})$ (with respect to the $A$-topology).

\begin{Proposition} 
\label{Comm-prop-contains-A-closure}
In the above notations assume that  $\rk(F)\geq 3$, and
let $G$ be the subgroup of $\Comm(\propF{p})$ generated by the groups $SA(U)$ where
$U$ ranges over all $p$-open subgroups of $F$. Then $G$ is simple and is contained in $\mon(\Comm(\propF{p}))$.
\end{Proposition}

\begin{proof} The second assertion follows from the first one and the fact that the intersection $G\,\cap\, \mon(\Comm(\propF{p}))$
is non-trivial (for instance, since it contains $F$).

Let us now prove that $G$ is simple. Let $N$ be a non-trivial normal subgroup of $G$.  Since $G$
contains $\SCpC_p(\absF{})$, it follows from Theorem~\ref{thm-completion-abs-simple} that $N$ contains $\SCpC_p(\absF{})$,
so in particular $N$ contains $\SAut(U)$ for every $p$-open subgroup $U$ of $\absF{}$.
\skv

Now fix such a subgroup $U$ and let $\propU{p}$ be its closure in $\propF{p}$. The composite map 
$\SAut(U)\to \SAut(\propU{p})\to\SAut(\propU{p}/[\propU{p},\propU{p}]{\propU{p}}^p)$ is surjective. 
Since $N\cap SA(U)$ is a normal subgroup of $SA(U)$ which contains $\SAut(U)$, we can apply 
Lemma~\ref{lem:abstractnormalclosure} to $K=N\cap SA(U)$ (with $U$ playing the role of $F$) to conclude that
$N\cap SA(U)=SA(U)$. Thus, $N$ contains $SA(U)$. Since this is true for any $p$-open subgroup $U$ of $F$,
we deduce that $N=G$, as desired.
\end{proof}

\section{Dependency on the rank and automorphisms of the $p$-commensurator} \label{sec-out-comm-p}

In this section we will prove Theorems~\ref{thm-intro-depend-parameters}~and~\ref{thm-intro-trivial_outer} restated below as
Theorems~\ref{thm-depend-parameters}~and~\ref{thm:trivial_outer}, respectively.
As before, throughout the section $\absF{}$ will denote a non-abelian free group of finite rank.

\begin{Theorem} \label{thm-depend-parameters}
	Let $p,q$ be prime numbers, let $k,\ell \geq 2$, and let $\absF{k}$ and $\absF{\ell}$ be free groups of rank $k$ and $\ell$, respectively. 
	Then the groups $\Comm_p(\absF{k})$ and $\Comm_q(\absF{\ell})$ are isomorphic if and only if $p=q$ and there exists $s \in \Z$ such that $(k-1) / (\ell -1) = p^s$. 
\end{Theorem}

\begin{Theorem}\label{thm:trivial_outer}
Every automorphism of $\Comm_p(\absF{})$ is inner.
\end{Theorem}

The majority of work in this section will be devoted to proving Proposition~\ref{prop-key-iso-auto} below.
Theorems~\ref{thm-depend-parameters}~and~\ref{thm:trivial_outer} will follow quite easily from Proposition~\ref{prop-key-iso-auto}
and results of Section~\ref{sec-conjugacy}.

\begin{Proposition}\label{prop-key-iso-auto}
 Let $\absF{}$ and $\absF{}'$ be free groups of finite rank, and suppose that $\psi: \Comm_p(\absF{}') \to \Comm_p(\absF{})$ is an isomorphism. Then $\psi(\absF{}')$ is $p$-commensurable with $\absF{}$. 
\end{Proposition}

It is not hard to show that Proposition~\ref{prop-key-iso-auto} is equivalent to the statement that any isomorphism $\psi: \Comm_p(\absF{}') \to \Comm_p(\absF{})$ is a homeomorphism.

\subsection{Commensurated subgroups} \label{subsec-commens-subgroups}

We first consider finitely generated commensurated subgroups of $\Comm_p(\absF{})$.  The following proposition is well known, but we could not locate a proof in the literature. 

\begin{Proposition} \label{prop-commens-fg-freegroup}
Let $\Lambda$ be an infinite finitely generated commensurated subgroup of $\absF{}$. Then $\Lambda$ has finite index in $\absF{}$.
\end{Proposition}

\begin{proof}
By Theorem~\ref{thm:Hallfreefactor}, the subgroup $\Lambda$ is a free factor of a finite index subgroup $H$ of $\absF{}$. In particular, $\Lambda$ is malnormal in $H$. Since $H$ is commensurated, this implies that either $\Lambda$ is finite or $\Lambda = H$. 
\end{proof}

We will also use the following, which can easily be derived from \cite[Proposition 4]{Wil94}. 

\begin{Proposition} \label{prop-Sha-Wil}
Let $\Gamma$ be a group with a commensurated subgroup $\Lambda$. Suppose that $x \in \Gamma$ is conjugate to $x^n$ for some $n \ge 2$. Then $x$ normalizes a subgroup $\Lambda'$ of $\Gamma$ such that $\Lambda$ and $\Lambda'$ are commensurable. 
\end{Proposition}

\begin{proof}
Let $G$ be the Schlichting completion of $\Gamma$ with respect to $\Lambda$ (\cite{Schlichting}), that is, if we write $\pi: \Gamma \rightarrow \mathrm{Sym}(\Gamma/\Lambda)$ for the left translation action of $\Gamma$ on $\Gamma/\Lambda$, then $G$ is the closure of $\pi(\Gamma)$ in $\mathrm{Sym}(\Gamma/\Lambda)$ with respect to the topology of pointwise convergence.  Since $\Lambda$ is commensurated by $\Gamma$, 
the closure of $\Lambda$ in $G$ is a profinite open subgroup, so
$G$ is a totally disconnected locally compact group (see \cite[\S~3]{Shalom-Willis} or \cite[\S~5]{Elder-Willis}).  

If $t \in \Gamma$ is such that $txt^{-1} = x^n$, then 
by induction $(t^{-k} x t^k)^{n^k} =x$ for every $k\in\mathbb N$. Hence the element $x$ is infinitely divisible, and by \cite[Proposition 4]{Wil94}, it follows that $\pi(x)$ normalizes a compact open subgroup $U$ of $G$. The preimage of $U$ in $\Gamma$ satisfies the conclusion.
\end{proof}

\begin{Lemma} \label{lem-commens-intersects}
	Let $\Lambda$ be an infinite finitely generated commensurated subgroup of $\Comm_p(\absF{})$. Then $\Lambda \cap \absF{} $ is non-trivial.
\end{Lemma}

\begin{proof}
Let $N$ denote the set of elements of  $\Comm_p(\absF{})$ that centralize a finite index subgroup of $\Lambda$. Since $\Lambda$ is commensurated, $N$ is a normal subgroup of $\Comm_p(\absF{})$. By Theorem \ref{thm-simplicity-Commp}, $N$ is either trivial or contains $\SComm_p(\absF{})$. In particular, if $N$ is non-trivial, then $N$ contains $\absF{}$. Since $\absF{}$ is finitely generated, this implies that $\absF{}$ centralizes a finite index subgroup of $\Lambda$. But $\absF{}$ has trivial centralizer in $\Comm_p(\absF{})$, so 
the trivial subgroup has finite index in $\Lambda$ and thus $\Lambda$ is finite, contrary to the hypotheses. Hence $N$ is trivial.
	
	Fix $x \in \absF{}\setminus\{1\}$. By Proposition~\ref{prop-conjugacy-Comm-p}, $x$ is  conjugate to $x^p$ in $\Comm_p(\absF{})$, and hence by Proposition~\ref{prop-Sha-Wil}, $x$ normalizes a subgroup commensurable with $\Lambda$. Without loss of generality we can assume that $x$ normalizes $\Lambda$. Now since $\absF{}$ is commensurated, for every $h \in \Lambda$ the index $[F:F\cap h F h^{-1}]$ is finite, and hence
there exists $n(h) \in\mathbb N$ such that 
$h x^{n(h)} h^{-1} \in \absF{}$. It follows that $[h,x^{n(h)}] \in \Lambda\cap \absF{}$, and it remains to show that
$[h,x^{n(h)}]$ is non-trivial for some $h$.
Suppose, on the contrary, that $[h,x^{n(h)}]$ is trivial for every $h\in\Lambda$. Since $\Lambda$ is finitely generated, there is a non-trivial power of $x$ that centralizes $\Lambda$. This is a contradiction with the first paragraph.
\end{proof}

We say a subgroup $H \subseteq G$ is {\it virtually normal} if there is a normal subgroup $N$ of $G$ such that $N \subseteq H$ and $|H:N|$ is finite; equivalently, the intersection of all $G$-conjugates of $H$ has finite index in $H$. We appeal to the main theorem of \cite{CKRW}, which will help to obtain a restriction on commensurated subgroups of $\Comm_p(\absF{})$ contained in $\absF{}$.

\begin{Theorem} [{\cite[Main Theorem]{CKRW}}]  \label{thm-virtually-normal} 
Let $\Gamma$ be a group, and let $\Lambda$ be a commensurated subgroup of $\Gamma$ such that $\Gamma$ is generated by finitely many cosets of $\Lambda$ (of course, the latter is automatic if $\Gamma$ is finitely generated).  Then the intersection of all virtually normal subgroups containing $\Lambda$ is itself virtually normal in $\Gamma$.
\end{Theorem}

\begin{Proposition} \label{prop-commens-insideFk}
Let $\Lambda$ be an infinite subgroup of $\absF{}$ such that $\Lambda$ is commensurated in $\Comm_p(\absF{})$. Then the closure of $\Lambda$ is a $p$-open subgroup of $\absF{}$.
\end{Proposition}

\begin{proof}
Since the closure of a commensurated subgroup remains commensurated (\cite[Lemma 2.7]{LB-W-commens-aaut}), it is enough to prove the statement assuming that $\Lambda$ is $p$-closed. It then suffices to show that $\Lambda$ has finite index in $\absF{}$. If $\Lambda$ is finitely generated, it has finite index in 
$\absF{}$ by Proposition~\ref{prop-commens-fg-freegroup}, so from now on we will assume that $\Lambda$ is not finitely generated.

The assumption that $\Lambda$ is $p$-closed implies that $\Lambda$ is an intersection of finite index (in particular, virtually normal) subgroups of $\absF{}$. Thus, $\Lambda$ is the intersection of all virtually normal subgroups of $\absF{}$ containing $\Lambda$. Since $F$ is finitely generated, by Theorem~\ref{thm-virtually-normal} $\Lambda$ is virtually normal in $\absF{}$, that is, the intersection $N$ of all $\absF{}$-conjugates of $\Lambda$ has finite index in $\Lambda$. 

Since $N$ is non-abelian (as it is non-trivial and normal in $F$), by Lemma~\ref{lem-existence-nonpower} there exists $x\in N$ which is not a proper power in $F$. By the proof of Lemma~\ref{lem-easy-conjugacy-Commp}, $x$ is a primitive element of some $p$-open subgroup $H$ of $F$.
In particular, there exists a surjective homomorphism $\pi:H \to \Z$ such that $\pi(x)$ generates $\Z$.
  
Let $K=\Ker(\pi)$. Then $H\subseteq NK$; in addition, $K$ is not $p$-open in $F$ (as it has infinite index), but $K$ is $p$-closed in $F$.
The latter holds since $K$ is the intersection of the subgroups $\pi^{-1}(p^i\dbZ)$, $i\in\dbN$, and each of those subgroups is $p$-open in $H$
(and hence in $F$ by Lemma~\ref{lem-localiso-extends-completion}(i)). Moreover, since $\absF{}$ is finitely generated, so is $H$; thus we can take a finitely generated subgroup $K_0$ of $K$ such that $H \subseteq NK_0$.  Let $K_1$ be the closure of $K_0$ in $\absF{}$ (with respect to the pro-$p$ topology). Since $K_0$ is finitely generated,
so is $K_1$ by \cite[Proposition 3.4]{RZ}. Further, we have $K_1 \subseteq K$, so $K_1$ is not $p$-open in $\absF{}$.

Let $r$ be the rank of $K_1$. Since $\Lambda$ is not finitely generated, it admits a free factor $\Lambda_1$ of rank $r$. Being a free factor of a $p$-closed subgroup of $\absF{}$, the subgroup $\Lambda_1$ is $p$-closed in $\absF{}$ by 
Proposition \ref{prop-RZ-closed-iff-freefact}(\ref{RZ-closed-1}). Since neither $\Lambda_1$ nor $K_1$ is $p$-open, we are therefore in position to apply Proposition \ref{prop-bounded-subgroup-conjugacy}, which asserts that there is $c \in \Comm_p(\absF{})$ such that 
$c\Lambda_1c^{-1} = K_1$. In particular, $K_1 \subseteq c \Lambda c^{-1}$. Since $\Lambda \cap c \Lambda c^{-1}$ has finite index in $c \Lambda c^{-1}$, it follows that $\Lambda \cap K_1$ has finite index in $K_1$.  Since $H \subseteq NK_1$, we deduce that $\Lambda$ contains a finite index subgroup of $\absF{}$, {namely $N(\Lambda\cap K_1)$}.
\end{proof}

Combining Lemma~\ref{lem-commens-intersects} with Proposition~\ref{prop-commens-insideFk} yields the following.

\begin{Corollary}\label{cor-commens-fg}
Let $\Lambda$ be an infinite finitely generated commensurated subgroup of $\Comm_p(\absF{})$.  Then the closure of $\Lambda \cap \absF{}$ is $p$-open in $\absF{}$.
\end{Corollary}

\begin{proof}
By Lemma~\ref{lem-commens-intersects}, the group $\Lambda' = \Lambda \cap \absF{}$ is non-trivial, hence infinite as $\absF{}$ is torsion-free.  The conclusion then follows from Proposition~\ref{prop-commens-insideFk}.
\end{proof}

A tdlc group $G$ is {\it residually discrete} if for every non-trivial $g \in G$ there exists a discrete group $Q$ and a continuous surjective homomorphism $\varphi: G \to Q$ such that $\varphi(g) \neq 1$. Equivalently, the intersection of all open normal subgroups of $G$ is the trivial subgroup. The following is  \cite[Corollary 4.1]{CaMo-decompo}.

\begin{Theorem} \label{thm-CaMo-resdisc}
Let $G$ be a compactly generated residually discrete tdlc group. Then for every compact open subgroup $U$ of $G$ there exists a finite index open subgroup $V$ of $U$ such that $V$ is normal in $G$. 
\end{Theorem}

\begin{Proposition} \label{prop-commens-SIN}
Let $H$ be a commensurated subgroup of a group $C$. Let $G$ be a tdlc group, and suppose that $\varphi : C \to G$ is an injective homomorphism with dense image, and that the only closed normal subgroup of $G$ contained in {$\overline{\varphi(H)}$} is the trivial subgroup. Let $L$ be a finitely generated subgroup of $H$. Then for every  compact open subgroup $U$ of $G$, $L$ normalizes a finite index subgroup of $\varphi^{-1}(U \cap \overline{\varphi(L)})$. 
\end{Proposition}

\begin{proof}
Let $J =  \overline{\varphi(H)}$. Since $\varphi(H)$ is commensurated by $\varphi(C)$, one easily verifies that $J$ is also commensurated by $\varphi(C)$. Set \[K =  \bigcap_{{c\in C}} \varphi(c) J \varphi(c)^{-1}. \] The subgroup $K$ is closed and normalized by the dense subgroup $\varphi(C)$, so it follows that $K$ is normal in $G$. Moreover, $K$ is contained in $J$, so by the assumption $K$ must be trivial. Since all the conjugates $\varphi(c) J \varphi(c)^{-1}$ are commensurable with each other, it follows that $J$ is residually finite, {so in particular residually discrete}.
Now if  $L$ is a finitely generated subgroup of $H$, the subgroup $J' =  \overline{\varphi(L)}$ is  a residually discrete tdlc group
which is also compactly generated (as it is locally compact and admits a finitely generated dense subgroup). 
Therefore Theorem~\ref{thm-CaMo-resdisc}  applies to $J'$. If $U$ is a compact open subgroup of $G$, then $U \cap J'$ is a compact open subgroup of $J'$ and hence contains a finite index open subgroup that is normal in $J'$. Taking the preimage in $C$ provides the conclusion.
\end{proof}

\subsection{Isomorphisms between $p$-commensurators of free groups}

In this subsection we will prove Proposition~\ref{prop-key-iso-auto} restated below. We will then use it to prove Theorems~\ref{thm-depend-parameters} and~\ref{thm:trivial_outer}.

Before proving Proposition~\ref{prop-key-iso-auto}, we record a general lemma which must be well known, but we are not aware of a reference where it is stated as below.

\begin{Lemma}\label{lem:no-invariant-compact-open}
Let $G$ be a tdlc group whose compact open subgroups are topologically finitely generated. Assume that $g,x\in G\setminus\{1\}$ are such that
the sequence $(g^n x g^{-n})_{n=1}^{\infty}$ converges to $1$. Then $g$ cannot normalize any compact open subgroup of $G$.  
\end{Lemma}
\begin{proof} Suppose, on the contrary, that $g$ normalizes some compact open subgroup $U$ of $G$. Since $U$ is topologically finitely generated,
it is hereditarily characteristically based by Lemma~\ref{lem:hcb}. In particular, there exists an open characteristic subgroup $V$ of $U$ which does not contain $x$. Since $g$ normalizes $U$, it must also normalize $V$. Thus, $g^n x g^{-n}\not\in V$ for all $n\in\dbN$, which contradicts our hypothesis
since $V$ is open.
\end{proof}

\begin{Proposition10.3}
 Let $\absF{}$ and $\absF{}'$ be free groups of finite rank, and suppose that $\psi: \Comm_p(\absF{}') \to \Comm_p(\absF{})$ is an isomorphism. Then $\psi(\absF{}')$ is $p$-commensurable with $\absF{}$. 
\end{Proposition10.3}

\begin{proof}
For simplicity we write $C = \Comm_p(\absF{})$ and $G = \CpC_p(\absF{})$. Recall that we view $C$ as a subgroup of $G$; the closure of $\absF{}$ in $G$ is the free pro-$p$ group $\propF{}$. We also set $H = \psi(\absF{}')$ and let $J = \overline{H}$ be the closure of $H$ in $G$. 
{We divide the proof into several steps}.
\vskip .1cm

\emph{Step 1: $J$ is open in $G$}. 
The subgroup $H$ is finitely generated and is commensurated in $C$. 
Consider the subgroup $\Lambda= \absF{} \cap H$. 
By Lemma~\ref{lem-commens-intersects} $\Lambda$ is non-trivial. Being the intersection of two commensurated subgroups, $\Lambda$ is commensurated in $C$.  Moreover, by Corollary~\ref{cor-commens-fg}, the closure of $\Lambda$ in $C$ is $p$-open in $\absF{}$. Thus $\Lambda$ is dense in an open subgroup of $\propF{}$, and hence $J$ contains an open subgroup of $\propF{}$, showing that $J$ is open in $G$.
\vskip .1cm

\emph{Step 2: Every element of $H$ is conjugate in $C$ to an element of $\absF{}$}. Applying Corollary~\ref{cor-non-ab-intersects-conj-class} inside $\Comm_p(\absF{}')$ to the subgroup $\psi^{-1}(\Lambda)$, we infer that every element of $\absF{}'$ is conjugate in $\Comm_p(\absF{}')$ to an element of $\psi^{-1}(\Lambda)$. Applying $\psi$ yields that every element of $H$ is conjugate in $C$ to an element of $\Lambda$, which in particular implies the assertion of Step~2.

\vskip .1cm
\emph{Step 3: $J$ has infinite index in $G$.} Step~2 implies that every element of $H$ normalizes some conjugate of $\propF{}$ in $G$.  In any tdlc group, the union of normalizers of compact open subgroups is closed.  (This is an immediate consequence of the continuity of Willis' scale function -- see \cite[Corollary 4]{Wil94}.)  Thus each element of $J$ normalizes some compact open subgroup of $G$.  

On the other hand, by Proposition~\ref{prop-conjugacy-Comm-p}, there exist $g \in C$ and $x \in \absF{} \setminus \{1\}$ such that $gxg^{-1} = x^{p}$ and hence $g^nxg^{-n}=x^{p^n}$ for all $n\in\dbN$. By Lemma~\ref{lem:no-invariant-compact-open}, $g$ cannot normalize any compact open subgroups of $G$, and moreover, the same is true for $g^k$ for all $k\in\dbN$. Thus, by the previous paragraph, $g^k\not\in J$ for all $k\in \dbN$ and hence $[G:J]$ is infinite.

\vskip .1cm
\emph{Step 4: $H$ lies inside $\Aut(K)$ for some finite index subgroup $K$ of $F$.} By Theorem~\ref{thm-completion-abs-simple},
any non-trivial normal subgroup of $G=\CpC_p(\absF{})$ contains $\SCpC_p(\absF{})$ and thus has index at most $2$ in $G$.
Hence by Step~3, $J$ does not contain any non-trivial normal subgroup of $G$, and we can apply Proposition \ref{prop-commens-SIN} (with $L=H$). Together with the fact that $J$ is open in $G$, the proposition implies that $H$ normalizes a finite index subgroup $K$ of $\absF{}$. 
This means that viewed inside $C$, the subgroup $H$ lies inside $\Aut(K)$, as desired. 

\vskip .1cm
\emph{Step 5: $H$ and $F$ are commensurable.}
The group $\mathrm{Out}(K)$ is  virtually torsion-free by \cite{Baumslag-Taylor}. Let $A$ be the preimage in $\Aut(K)$ of a finite index torsion-free subgroup of $\mathrm{Out}(K)$.  Let $h \in H \cap A $. There is $c \in C$ such that $h\in c\absF{}c^{-1} $. Since $c\absF{}c^{-1} $ is commensurable with $\absF{}$, there is $n \geq 1$ such that $h^n \in \absF{}$, and upon enlarging $n$ we have $h^n \in K$. Since $A$ has torsion-free image in $\mathrm{Out}(K)$, we deduce that $h \in K$ and hence $H \cap A \subseteq K$. Since $A$ has finite index in $\Aut(K)$, it follows that $H$ is virtually contained in $K$. Thus $H \cap K$ is a commensurated subgroup of $K$ that is commensurable with $H$ and hence finitely generated. This implies that $H \cap K$ has finite index in $K$ by Proposition \ref{prop-commens-fg-freegroup}. Therefore, $H$ and $K$, and hence $H$ and $\absF{}$, are commensurable. 

\vskip .1cm
\emph{Step 6: $H$ and $F$ are $p$-commensurable.} By definition we need to show that $H \cap \absF{}$ is $p$-open in both $\absF{}$ and $H$.
Both assertions can be proved similarly, so we will only do the first one. Recall that a finite index subgroup of a group $\Gamma$
is open in the pro-$p$ topology on $\Gamma$ if and only if every element has a $p$-power that belongs to $\Gamma$. We already know by Step~5 that $H\cap F$ has finite index in $F$,
so it remains to show that every $x \in \absF{}$ has a $p$-power in $H \cap \absF{}$ (equivalently, a $p$-power in $H$).

 By Corollary \ref{cor-non-ab-intersects-conj-class}, there is $c \in C$ such that $x \in c\Lambda c^{-1}$. Now since $H$ is $p$-commensurated (Lemma \ref{lem-im-p-commens}), $H \cap cHc^{-1}$ is $p$-open in $cHc^{-1}$, so $H \cap c\Lambda c^{-1}$ is $p$-open in $c\Lambda c^{-1}$. Hence $x$ has a $p$-power in $H \cap c\Lambda c^{-1}\subseteq H$, as desired.
\end{proof}

We are now ready to prove Theorems~\ref{thm-depend-parameters}~and~\ref{thm:trivial_outer}.

\begin{proof}[Proof of Theorem~\ref{thm-depend-parameters}]
	Suppose  that $\psi: \Comm_q(\absF{\ell}) \to \Comm_p(\absF{k})$ is an isomorphism. We first show that $p = q$.  Arguing as in the beginning of the proof of Proposition \ref{prop-key-iso-auto}, we see that $\absF{k} \cap \psi(\absF{\ell})$ is not trivial. Let $w \in \absF{k} \cap \psi(\absF{\ell})$ be a non-trivial element. By Corollary \ref{cor:p-power-conjugation}, $w$ and $w^p$ are conjugate in $\Comm_p(\absF{k})$. Hence the same holds for $\psi^{-1}(w)$ and $\psi^{-1}(w^p)$, and applying Corollary \ref{cor:p-power-conjugation} in the group $\Comm_q(\absF{\ell})$ yields that the primes $p$ and $q$ must be equal. 
	
	For an integer $i\geq 2$ we let \[ r_p(i) = 1 + \frac{i-1}{p^\alpha}\] where $p^\alpha$ is the largest power of $p$ dividing $i-1$. Given $i,j\in\dbZ$ we have $r_p(i) = r_p(j) $ if and only if $(i-1) / (j -1) = p^s$ for some $s\in\mathbb Z$. We also note that $r_p(i) $ is the smallest rank of a free subgroup of $\Comm_p(\absF{i})$ that is $p$-commensurable with $\absF{i}$. 

Since $\psi$ maps the $p$-commensurability class of $\absF{\ell}$ onto the $p$-commensurability class of $\absF{k}$ by Proposition \ref{prop-key-iso-auto}, we conclude that $r_p(k) = r_p(\ell) $. This proves the {forward direction of Theorem~\ref{thm-depend-parameters}}. The converse implication is clear as if $r_p(k) = r_p(\ell) = r$, then  $\Comm_p(\absF{k})$ and $\Comm_p(\absF{\ell})$ are both isomorphic to $\Comm_p(\absF{r})$.
\end{proof}

\begin{proof}[Proof of Theorem~\ref{thm:trivial_outer}]
Let $\alpha \in \Aut(\Comm_p(\absF{}))$. We want to show that $\alpha$ is inner {or, equivalently, $\alpha$ can be made trivial after composition with inner automorphisms.}

The subgroup $ \alpha(\absF{})$ is free of the same rank as $F$, and by Proposition \ref{prop-key-iso-auto} $ \alpha(\absF{})$ is $p$-commensurable with $\absF{}$. By Proposition~\ref{prop-bounded-subgroup-conjugacy} this implies that  {$\alpha(\absF{})=c\absF{}c^{-1}$ for some $c\in\Comm_p(\absF{})$}. Hence upon composing 
$\alpha$ with an inner automorphism {of $\Comm_p(\abs{F})$ (namely, conjugation by $c^{-1}$)}, we can assume $\alpha(\absF{}) = \absF{}$. In that situation it follows that  $\alpha$ induces an automorphism of $\absF{}$, {call it $\phi$. Viewing $\phi$ as an inner automorphism of
$\Comm_p(F)$ and composing $\alpha$ with $\phi^{-1}$}, we can assume that $\alpha$ is the identity on $\absF{}$. Lemma \ref{lem-BEW} applied to the topological group $\Comm_p(\absF{})$ (and with $U=V=\Comm_p(\absF{})$) implies that every such automorphism is trivial, as desired. 
\end{proof}

\section{A family of compactly generated simple groups} \label{sec-compact-generated}

In this section we will prove Theorem~\ref{thm:maintdlc}. This theorem has several parts which will be established as separate statements: part (i) holds by Observation~\ref{Spn:increasing}, (ii) is Observation~\ref{obs:Nntrivial}, (iii) holds by Theorem~\ref{thm:normal-sbgps-Ln}, (iv) is Proposition~\ref{prop:F(n)-properties}(\ref{item-Fn-many-quotients})(\ref{item-Fn-iso-subgroups}), (v) is 
a combination of Corollary~\ref{cor-Ln-simple-Nntrivial} and Proposition~\ref{prop:F(n)-properties}(\ref{item-Fn-non-iso}), and
(vi) is part of Corollary~\ref{cor-conditional-Ln-simple}. 
\vskip .12cm

Let $\absF{}$ be a non-abelian free group of finite rank $d$ and $\propF{p}$ its pro-$p$ completion. As before, we identify $F$ with its image in $\propF{p}$, and we also view the group $\Comm_p(\absF{})$ as a subgroup of $\Comm(\propF{p})$. 
The majority of results in this section will require the additional assumption $(p,d)\neq (2,2)$ (which will always be stated explicitly).

Recall that in Section~\ref{sec-finite-generated} we introduced a natural family of finitely generated subgroups of $\Comm(F)$,
denoted by $S_m(F)$, and showed that these groups are simple under some conditions on $m$ and $\rk(F)$. In this section we consider an analogous problem inside the totally disconnected group $\Comm(\propF{p})$: we will define certain compactly generated open subgroups of $\Comm(\propF{p})$, and our primary goal is to determine whether such groups can be (abstractly) simple or at least close to being simple.

Also recall that the finitely generated simple subgroups $S_p(F)$ of $\Comm(F)$ 
constructed in Section~\ref{sec-finite-generated} are contained in the subgroup $\SComm_p(F)$ which itself was shown to be simple in Section~\ref{sec-comm-p-simple}. Further, the group $\SCpC_p(\absF{})$, defined as the closure of $\SComm_p(F)$ in $\Comm(\propF{p})$, is open in $\Comm(\propF{p})$ and also simple by
Theorem~\ref{thm-completion-abs-simple}. Thus, in view of the above goal, it is natural to focus our attention in this section on
compactly generated open subgroups contained in  $\SCpC_p(\absF{})$.

 The group $\SCpC_p(\absF{})$ admits a natural ascending sequence of compactly generated open subgroups $(L_n)$, each containing $\propF{p}$, whose union is equal to $\SCpC_p(\absF{})$. The main result of this section yields a normal subgroup $M_n$ of $L_n$ such that the quotient 
$Q_n = L_n / M_n$ is abstractly simple (Theorem \ref{thm:normal-sbgps-Ln}). The subgroup $M_n$ is relatively small -- it is an extension of a compact group $K_n$ contained in $\propF{p}$ by a discrete group. It is possible that $M_n=K_n$ (so in particular $M_n$ is compact). It is actually possible that $M_n$ is the trivial subgroup; we do not know whether this is the case. Denote the image of the free pro-$p$ group $\propF{p}$ in $Q_n$ by 
$F(n)$ -- it is a compact open subgroup of $Q_n$. We will see that the following three conditions are equivalent:
\begin{itemize}
\item[(i)] $M_n$ is trivial;
\item[(ii)] $K_n$ is trivial;
\item[(iii)] $F(n)$ is a free pro-$p$ group.
\end{itemize}
Further, we will show that whether these equivalent conditions hold for $n$ large enough is an intrinsic property of the group $\SCpC_p(\absF{})$ (Corollary \ref{cor-conditional-Ln-simple}). Although we did not manage to elucidate whether the latter is true or not, we establish a number of properties of the groups $F(n)$ which  suggest that $F(n)$ are rich and interesting pro-$p$ groups (Proposition~\ref{prop:F(n)-properties}). 

\subsection{Construction and abstract simplicity of the groups $Q_n$}

We start with several definitions.

\begin{Definition} 
For $n \geq 1$ let 
\begin{itemize}
\item $S_{p,n}(F)$ be the subgroup of $\mathrm{Comm}_p(F)$ generated by the subgroups $\SAut(H)$ where $H$ ranges over all $p$-open subgroups of $F$ of index $\leq p^n$;
\item $L_n = \left\langle \propF{p}, S_{p,n}(F) \right\rangle $, 
the subgroup of $\Comm(\propF{p})$ generated by $\propF{p}$ and $S_{p,n}(F)$.
\end{itemize}
\end{Definition}

Since $F \subseteq S_{p,n}(F)$ and $F$ is dense in $\propF{p}$ {while $\propF{p}$ is open in $\Comm(\propF{p})$}, the group $L_n$ is equal to the closure of $S_{p,n}(F)$ in $\Comm(\propF{p})$.  The following result is immediate from the definitions:

\begin{Observation}  \label{Spn:increasing}
The group $\SComm_p(\absF{})$ is the ascending union $\bigcup_{n \ge 1}S_{p,n}(F)$, and the group $\SCpC_p(\absF{})$ is the ascending union $\bigcup_{n \ge 1} L_n$.
\end{Observation}

Note that by Lemma~\ref{lem-properties-Am(F)}(\ref{item-SAutF-in-Sm}), we have $S_{p,1}(F) = S_{p}(F)$. In particular, $L_1$ is the group $\la \propF{p},S_p(\absF{}) \ra$ that appeared in Lemma~\ref{prop-dense-normal}.
Also observe that $S_{p,n}(F)$ is  generated by the subgroups $S_{p}(H)$ where $H$ ranges over $p$-open subgroups of $F$ of index at most $p^{n-1}$.

\begin{Definition}\label{def:Nn}
 Let $K_n$ denote the normal core of $\propF{p}$ in $L_n$.
\end{Definition}

\begin{Observation}\label{obs:Nntrivial}
The sequence $(K_n)$ is descending and has trivial intersection.
\end{Observation}
\begin{proof}
The sequence $(K_n)$ is descending since $(L_n)$ is ascending, and the intersection $\bigcap\limits_{n\geq 1}K_n $ is trivial 
since $\bigcup_{n \ge 1} L_n=\SCpC_p(\absF{})$ is simple by 
Theorem~\ref{thm-completion-abs-simple}. 
\end{proof}

Observation~\ref{obs:Nntrivial} also follows from
Corollary~\ref{Nn-superinvariant} below, which will provide an alternative characterization of the subgroups $K_n$.
\skv

Let us now consider the groups $$P_n=L_n/K_n\mbox{ and }F(n)=\propF{p}/K_n.$$
Note that $F(n)$ is a compact open subgroup of $P_n$, and therefore
we have a homomorphism $\pi_n:P_n\to \Comm(F(n))$.  Let
$Q_n$ denote the image of $\pi_n$. The kernel of $\pi_n$ is $\QZ(P_n)$, the quasi-center of $P_n$, so $Q_n\cong P_n/\QZ(P_n)$. As we will see in the following statement, $\QZ(P_n)$ is a discrete subgroup of $P_n$, and hence $Q_n$ is a tdlc group.

The following theorem is the main result of this section.

\begin{Theorem}\label{thm:normal-sbgps-Ln}
Suppose $(d,p) \neq (2,2)$, and let $n \ge 1$.  The following hold:
\begin{itemize}
\item[(a)] The group $F(n)$ has trivial quasi-center. Therefore  $\QZ(P_n)$ is a discrete subgroup of $P_n$, and $F(n)$ can be identified with a compact open subgroup of $Q_n$. 
\item[(b)] The group $Q_n$ is a non-discrete compactly generated abstractly simple tdlc group.
\end{itemize}
\end{Theorem}

Let us point out a simple but important consequence of Theorem~\ref{thm:normal-sbgps-Ln}.

\begin{Corollary} \label{cor-Ln-simple-Nntrivial}
Suppose $(d,p) \neq (2,2)$, and let $n \ge 1$.  Then $L_n$ is abstractly simple if and only if $K_n$ is trivial.
\end{Corollary}

\begin{proof}
The `only if' direction is clear. Suppose now that $K_n$ is trivial, so that $P_n=L_n$.
Note that $L_n$ has trivial quasi-center by Lemma~\ref{lem-Comm-trivialQZ}(\ref{item-Comm-trivial-QZ}).
Thus $\QZ(P_n)=\{1\}$, and hence $L_n=P_n=P_n/\QZ(P_n)$ is isomorphic to $Q_n$ which is abstractly simple
by Theorem~\ref{thm:normal-sbgps-Ln}(b).
\end{proof}
\skv

Before proving Theorem~\ref{thm:normal-sbgps-Ln} we will establish two auxiliary results.

\begin{Lemma}\label{claim:proper-open-normal}
	Suppose $(d,p) \neq (2,2)$.  If $N$ is an open normal subgroup of $L_n$, then $N = L_n$.
\end{Lemma}

\begin{proof}
Since $N$ is open in $L_n$, it contains an open subgroup of $\propF{p}$ and hence contains a $p$-open subgroup of $F$.
Now let $H$ be any $p$-open subgroup of $\absF{}$ with $[\absF{}:H] \le p^{n-1}$, {so that $S_p(H)\subseteq S_{p,n}(F)\subseteq L_n$}.
Since $N$ and $S_p(H)$ both contain a finite index subgroup of $F$ (and hence so does their intersection $N\cap S_p(H)$),
we can apply Proposition \ref{prop-SmF-fi-F} to the normal subgroup $N \cap S_p(H)$ of $S_p(H)$ and deduce that 
$S_p(H) \subseteq N$.  

Thus, we proved that $N$ contains $S_p(H)$ for every $p$-open subgroup $H$ of $\absF{}$ of index at most $p^{n-1}$.
As we observed earlier, the subgroups $S_p(H)$ for such $H$ generate $S_{p,n}(F)$, so $N$ contains $S_{p,n}(F)$.
Since $S_{p,n}(F)$ is a dense subgroup of $L_n$ and since $N$ is open and hence closed in $L_n$, it follows that $N = L_n$.
\end{proof}

\begin{Notation}
	For $n \ge 0$, we denote by $\mathcal{U}_n$ the set of open subgroups of $\propF{p}$ of index at most $p^n$.
\end{Notation}

Observe that since intersecting with $F$ defines a bijection between open subgroups of $\propF{p}$ of index $p^n$ and $p$-open subgroups of $F$ of index $p^n$, $S_{p,n}(F)$ can equivalently be described as the subgroup generated by $\SAut(\propU{p} \cap F)$ where $\propU{p}$ ranges over $\mathcal{U}_n$.

\begin{Proposition}\label{claim:proper_normal}
Suppose $(d,p) \neq (2,2)$ and let $N$ be a proper (not necessarily closed) normal subgroup of $L_n$. The following hold:

 \begin{itemize}
	\item[(a)] $N\cap \propU{p} \subseteq \Phi(\propU{p})$ for every $\propU{p} \in \mathcal{U}_{n-1}$;
	\item[(b)] $N \cap \propF{p}$ is normal in $L_n$ and $N \cap \propF{p} \subseteq K_n$.
\end{itemize}
\end{Proposition}

\begin{proof}
(a) Suppose that $N\cap \propU{p} \not\subseteq \Phi(\propU{p})$ for some $\propU{p} \in \mathcal{U}_{n-1}$.  We note that $\Comm(\propF{p}) = \Comm(\propU{p})$ and $N$ is normalized by the subgroup $\la \propU{p},S_p(U) \ra$ of $L_n$.  Applying Lemma~\ref{prop-dense-normal}, we deduce that $N$ contains $\propU{p}$, so $N$ is open, and then Lemma~\ref{claim:proper-open-normal} implies that $N=L_n$, a contradiction.

\skv	
(b) The assertion of (a) can be reformulated as follows: for every $\propU{p}\in \mathcal{U}_{n-1}$ and every index $p$ subgroup 
$\propV{p}$ of $\propU{p}$ we have $N \cap \propU{p}=N \cap \propV{p}$. Since every $\propU{p}\in \mathcal{U}_n$
can be connected to $\propF{p}$ by a chain of subgroups in which every subgroup has index $p$ in the next one, it follows that
$N \cap \propF{p}=N \cap \propU{p}$ for every $\propU{p}\in \mathcal{U}_n$. Since both $N$ and $\propU{p}$
are normalized by $\SAut(\propU{p}\cap F)$, it follows that $N \cap \propF{p}$ is normalized by
$\SAut(\propU{p}\cap F)$ for every $\propU{p}\in \mathcal{U}_n$ and hence is normal in $L_n$. Since $K_n$ is the largest normal subgroup of $L_n$ contained in $\propF{p}$, we conclude that $N \cap \propF{p} \subseteq K_n$. 
\end{proof}

As an easy consequence of Proposition~\ref{claim:proper_normal}, we obtain an alternative description of the subgroups $K_n$:

\begin{Corollary} \label{Nn-superinvariant}
The group $K_n$ is the largest subgroup of $\propF{p}$ which is 
\begin{itemize}
\item[(a)] contained in every $\propU{p}\in\calU_n$ and 
\item[(b)] invariant under $\SAut(\propU{p}\cap F)$ for every $\propU{p}\in\calU_n$.
\end{itemize}
\end{Corollary}

\begin{proof} First we explain why $K_n$ satisfies (a) and (b). Note that (a) must hold for (b) to be a meaningful
statement, but once (a) is established, (b) is automatic since by assumption $K_n$ is normal in $L_n$. 
Property (a) for $K_n$
follows from Proposition~\ref{claim:proper_normal}(a)
by induction on the index $[\propF{p}:\propU{p}]$ starting with the equality $K_n=K_n\cap \propF{p}$ (which holds by definition).
\skv
Now let $N$ be any subgroup of $\propF{p}$ satisfying (a) and (b). Then the same is true for $\overline{N}$, the closure of $N$.
By (b) the normalizer of $\overline{N}$ in $\Comm(\propF{p})$ contains $F$ (since $F\subseteq \SAut(F)$) and hence also $\propF{p}$
(since $\overline{N}$ is closed). Thus by (b), $\overline{N}$ is normal in 
$\la\propF{p}, \bigcup\limits_{\propU{p}\in\calU_n} \SAut(\propU{p}\cap F)\ra=L_n$ and hence 
$N\subseteq\overline{N}=\overline{N} \cap \propF{p}$ is contained in $K_n$ by Proposition~\ref{claim:proper_normal}(b).
\end{proof}

We are now ready to prove Theorem~\ref{thm:normal-sbgps-Ln}.

\begin{proof}[Proof of Theorem \ref{thm:normal-sbgps-Ln}]
(a) Note that $\QZ(F(n))=\QZ(P_n)\cap F(n)$ since $F(n)$ is an open subgroup of $P_n$, so let us prove that $\QZ(P_n)\cap F(n)=\{1\}$.

Suppose first that $\QZ(P_n)=P_n$. By \cite[Prop. 4.3]{CaMo-decompo} (see also \cite[Thm 4.8]{BEW}), a compactly generated tdlc group with
dense quasi-center has a base of neighborhoods of the identity consisting of compact open normal subgroups. In particular, 
this would imply that $P_n$ has a proper open normal subgroup, and hence the same would be true for $L_n$, contrary to 
Lemma~\ref{claim:proper-open-normal}.

Thus, $\QZ(P_n)\neq P_n$, and if $\rho:L_n\to P_n = L_n/K_n$ is the canonical projection, then $M_n=\rho^{-1}(\QZ(P_n))$ is a proper normal subgroup of $L_n$. Proposition~\ref{claim:proper_normal} then implies that $M_n\,\cap\, \propF{p} \subseteq K_n$. Hence
$$\rho^{-1}(\QZ(P_n)\,\cap\, F(n))=\rho^{-1}(\QZ(P_n))\,\cap\, \rho^{-1}(F(n))=M_n\,\cap\, \propF{p} \subseteq K_n$$ and therefore
$\QZ(P_n)\,\cap\, F(n)=\{1\}$, as desired.
\skv

(b) The group $F(n)$ is pro-$p$ and clearly non-trivial. Since we just showed that $F(n)$ has trivial quasi-center, it must be infinite
and also can be identified with its image in $\Comm(F(n))$. Thus, $Q_n$ contains the infinite pro-$p$ group $F(n)$ 
as an open subgroup, so in particular, $Q_n$ is tdlc and non-discrete. It remains to show that $Q_n$ is abstractly simple.

Let $N$ be a non-trivial normal subgroup of $Q_n$. Since $F(n)$ has trivial quasi-center,
by Lemma~\ref{lem-Comm-trivialQZ}(\ref{item-Comm-local-normal}) 
the intersection $N\cap F(n)$ is non-trivial.
Hence if $\pi:L_n\to Q_n$ is the canonical map, then $\pi^{-1}(N)\cap \propF{p} \not\subseteq K_n$, and Proposition~\ref{claim:proper_normal}
implies that $\pi^{-1}(N)=L_n$ whence $N=Q_n$, as desired.
\end{proof}

\subsection{Additional properties of the groups $F(n)$}\label{subsec-add-properties-Fn}

We have already proved that $F(n)$ is a non-trivial pro-$p$ group with trivial quasi-center; in particular, it is not torsion and not nilpotent.
The following lemma provides alternative proofs of these facts but has the advantage of producing explicit elements which do not lie in $K_1$
(and hence also in $K_n$ for all $n\in\dbN$).

\begin{Lemma} The following hold:
\begin{itemize}
\item[(a)] Assume that $d\geq 2$, and let $x_1$ be a primitive element of $F$. Then $x_1^{p^k}\not\in K_1$ for any $k\in\dbN$.
\item[(b)] Assume that $d\geq 3$ and $\{x_1,x_2\}$ is a subset of a basis of $F$. Then the left-normed commutator
$h_k=[x_2,x_1,x_1^p,x_1^{p^2},\ldots, x_1^{p^{k-1}}]$ does not lie in $K_1$ for any $k\in\dbN$.
\end{itemize}
\end{Lemma}
\begin{proof}
\skv
(a) First note that $K_1$ is contained in $\bigcap\limits_{\propU{p}\in\calU_1} \propU{p}=\Phi(\propF{p})$, so $x_1\not\in K_1$. On the other hand,
by Proposition~\ref{lem-conjugate-power-SmF}, $x_1$ is conjugate in $S_p(F)$ (and hence in $L_1$) to $x_1^p$
and hence to $x_1^{p^k}$ for any $k\in\dbN$. These two facts imply (a) since $K_1$ is normal in $L_1$. 

\skv 

(b) Let $X$ be any basis of $F$ containing $x_1$ and $x_2$ and choose any $x_3\in X\setminus\{x_1,x_2\}$ (this is possible since
$d\geq 3$ by assumption). For $i=1$ and $i=3$ let $U_i=F(X_0,x_i,p)$, the unique subgroup of index $p$ in $F$
which contains $X_0\setminus \{x_i\}$. Then 
$S_i=\{[x_j,\underbrace{x_i,\ldots, x_i}_{k\mbox{ times }}]: j\neq i, 0\leq k\leq p-1\}\cup\{x_i^p\}$
is a basis for $U_i$ by Lemma~\ref{lem:Schreier}(b).
Since $S_1$ contains $x_1^p$ and $[x_2,x_1]$ while $S_3$ contains $x_1$ and $x_2$,
there exists an isomorphism $\phi:U_1\to U_3$ such that $\phi(x_1^p)=x_1$ and $\phi([x_2,x_1])=x_2$.

Now for each $k\in\dbN$ consider the left-normed commutator $h_k=[x_2,x_1,x_1^p,x_1^{p^2},\ldots, x_1^{p^{k-1}}]$.
Then $\phi$ is defined on each $h_k$ and $\phi(h_k)=h_{k-1}$ for all $k\geq 2$, so $\phi^k(h_k)=x_2$ for all $k\in\dbN$.
Since $K_1$ does not contain $x_2$, as explained in (a), and $K_1$ is $\phi$-invariant, it follows that
$K_1$ does not contain $h_k$ for any $k$.
\end{proof}

Recall that the sequence of compact subgroups $(K_n)$ is decreasing and has trivial intersection. We do not know whether $K_n$ is trivial for $n$ large enough or whether $K_n$ is trivial for every $n \geq 1$. The condition that $K_n$ is trivial  is equivalent to saying that $F(n)$ is a free pro-$p$ group (see item (\ref{item-Fn-non-iso}) below). The following proposition shows that $F(n)$ has several properties which are typical for the free pro-$p$ group of rank $d$. This can either be seen as evidence that $F(n)$ is free pro-$p$ or, in case $F(n)$ is not free pro-$p$, that $F(n)$ is not isomorphic to any previously studied group.

\begin{Proposition}\label{prop:F(n)-properties}
	Suppose that $d \geq 3$, and let $n \geq 1$. The following hold:
	\begin{enumerate}
		\item \label{item-Fn-many-quotients} The minimal number of generators of $F(n)$ is $d$, and every $d$-generated group of order $p^j$ for $j \le n+1$ occurs as a quotient of $F(n)$.
		\item \label{item-Fn-iso-subgroups} For all $1 \le j \le n+1$ any two subgroups of index $p^j$ in $F(n)$ are isomorphic. 
		\item \label{item-Fn-non-iso} Let $m,n\in\dbN$. Then $F(m) \cong F(n)$ if and only if $K_m = K_n$, and $F(n)$ is a free pro-$p$ group if and only if $K_n = \{1\}$.
	\end{enumerate}
\end{Proposition}

\begin{proof}

(\ref{item-Fn-many-quotients}) Proposition~\ref{claim:proper_normal} applied with $N = K_n$ yields  $K_n \subseteq \Phi(\propU{p})$ for every 
$\propU{p} \in \mathcal{U}_n$. In particular, we have $F(n) / \Phi(F(n)) \simeq \propF{p} / \Phi(\propF{p})$, and hence the minimal number of generators of $F(n)$ is $d$. If $P$ is a $d$-generated group finite $p$-group, then by the universal property of free pro-$p$ groups, there is a continuous surjective homomorphism $\pi_P: \propF{p} \rightarrow P$.  If in addition, $|P|\leq p^{n+1}$, then $\ker\pi_P$ contains $\Phi(\propU{p})$ for some $\propU{p} \in \mathcal{U}_n$, so 
$K_n \subseteq \ker\pi_P$, and hence $\pi_P$ factors through a surjective homomorphism from $F(n)$ to $P$.  This completes the proof of (\ref{item-Fn-many-quotients}).

(\ref{item-Fn-iso-subgroups}) Given two subgroups of $F(n)$ of the same index $p^j$ with $1\leq j\leq n+1$, we can represent
them as $\propU{p}_1/K_n$ and $\propU{p}_2/K_n$ where $\propU{p}_1$ and $\propU{p}_2$ are subgroups of index $p^j$ in $\propF{p}$. Since $K_n$ is normalized by
$S_{p,n}(F)$, to prove that $\propU{p}_1/K_n\cong \propU{p}_2/K_n$ it suffices to show that $\propU{p}_1$ and $\propU{p}_2$ are $S_{p,n}(F)$-conjugate. 

Suppose  first that $j \le n$ and write $U_i = \propU{p}_i \cap \absF{}$.  Then $U_1$ and $U_2$ both have index $p^j$ in $\absF{}$, so they are free groups of the same rank; hence there exists an isomorphism $\phi$ from $U_1$ to $U_2$.  Arguing as in the proof of Proposition~\ref{prop-decomposing-prod-aut}, we deduce that $\phi$ is realized by conjugating by some $\gamma \in S_{p,n}(F)$; by continuity, it follows that $\gamma \propU{p}_1 \gamma^{-1} = \propU{p}_2$.  

{For $j = n+1$, choose a subgroup $\propV{p}_i$ of index $p^n$ containing $\propU{p}_i$ for $i=1,2$.  By the $j=n$ case there exists
$\gamma \in S_{p,n}(F)$ such that $\gamma \propV{p}_1 \gamma^{-1} = \propV{p}_2$. Then $\gamma\propU{p}_1\gamma^{-1}$ and $\propU{p}_2$
are both index $p$ subgroups of $\propV{p}_2$, so there exists $\phi\in \SAut(\propV{p}_2 \cap F)$ such that 
$\phi(\gamma\propU{p}_1\gamma^{-1})\phi^{-1}=\propU{p}_2$. Since $\SAut(\propV{p}_2 \cap F)\subseteq S_{p,n}(F)$,
the subgroups $\propU{p}_1$ and $\propU{p}_2$ are conjugate by the element $\phi\gamma\in S_{p,n}(F)$, as desired.} 
 
	(\ref{item-Fn-non-iso}) We may assume $m < n$, which means that $K_n \subseteq K_m$, and hence there is a quotient homomorphism $\theta: F(n) \rightarrow F(m)$ with kernel $K_m/K_n$.  If $K_m = K_n$, then clearly $\theta$ is an isomorphism.  On the other hand, since $F(n)$ is a finitely generated pro-$p$ group, it is Hopfian, so if $\theta$ is not injective (in other words, $K_m \neq K_n$), then $F(m)$ is not isomorphic to $F(n)$.  Similarly, since $\propF{p}$ is Hopfian, we have $F(n) \cong \propF{p}$ if and only if $K_n$ is trivial. 
Finally, by (\ref{item-Fn-many-quotients}), if $F(n)$ is free pro-$p$, it must be free of rank $d$ and thus isomorphic to $\propF{p}$.
\end{proof}

\subsection{Consequences of Theorem \ref{thm:normal-sbgps-Ln}}\label{subsec-questions-tdlc}

Let $\propF{p}$ be a non-abelian free pro-$p$ group of finite rank. Recall that Question~\ref{qCM} formulated in the introduction has a positive answer 
for $\propF{p}$ if and only if $\Comm(\propF{p})$ has
a compactly generated topologically simple open subgroup containing $\propF{p}$. Earlier in this section we considered the groups 
$L_n$ as potential candidates {for such a subgroup}. Recall that each $L_n$ is contained in $\SCpC_p(\absF{})$ and that $\SCpC_p(\absF{})=\bigcup\limits_{n \geq 1} L_n$. {We will now} show that if one is trying
to find an answer to Question~\ref{qCM} within $\SCpC_p(\absF{})$, there is no loss of generality in restricting to the groups $L_n$ (see Corollary~\ref{cor-conditional-Ln-simple} below).

We will use the following terminology. 

\begin{Definition}\label{nocompactnormalsubgroups}
We will say that a topological group $L$ has {\it(NCNS)} if $L$ has no non-trivial compact normal subgroup.
\end{Definition}

We will need a simple general observation.

\begin{Lemma} \label{lem-nocptnormal-overgroup}
Let $G$ be a profinite group with trivial quasi-center. \begin{enumerate}
	\item \label{item-lemma-NCNS-overgroup}  If $L$ is an open subgroup of $\Comm(G)$ which has (NCNS), then any subgroup $L'$  of $\Comm(G)$ containing $L$ also has (NCNS).
	\item \label{item-lemma-NCNS-all-cpct} Suppose $G$ is torsion-free. Then an open subgroup $L$ of $\Comm(\alb{G})$ has (NCNS) if and only if there exists a compact open subgroup $U$ of $L$ such that the  normal core of $U$ in $L$ is trivial.
\end{enumerate}
\end{Lemma}

\begin{proof}
(\ref{item-lemma-NCNS-overgroup}). 	Let $K$ be any compact normal subgroup of $L'$. Then $K \cap L$ is compact and normal in $L$ and hence is trivial. Since $L$ is open and $K$ is compact, this implies that $K$ is finite. Hence $K$ is a finite subgroup of $\Comm(G)$ whose normalizer is open. This implies that $K$ lies in the quasi-center of $\Comm(G)$. Since the latter is trivial by Lemma \ref{lem-BEW}, $K$ is trivial.

(\ref{item-lemma-NCNS-all-cpct}). The `only if' direction is clear. Suppose now that $U$ is a compact open subgroup of $L$ with trivial normal core in $L$. This is equivalent to saying that $L$ acts faithfully on the set $\Omega=L/U$. Let $K$ be any compact normal subgroup of $L$. Since $U$ is open and $K$ is compact, every $K$-orbit in $\Omega$ is finite. Since $K$ is normal in $L$, the $L$-action preserves the partition of $\Omega$ into $K$-orbits, and since $L$ acts transitively on $\Omega$, $L$ also acts transitively on the set $\alb{K \backslash \Omega}$ of $K$-orbits. It follows that there is a finite group $A$ such that the $K$-action induces a homomorphism $K \to \prod_{\alb{K \backslash \Omega}} A$, which is injective since the $L$-action is faithful. In particular, $K$ is torsion. Since we assume that $G$ is torsion-free, the intersection $K \cap G$ must be trivial. Hence $K$ is finite. For the same reason as in the proof of (\ref{item-lemma-NCNS-overgroup}), $K$ is trivial.
\end{proof}

\begin{Corollary} \label{cor-conditional-Ln-simple}
	The following are equivalent:
	\begin{enumerate}
		\item \label{item-G-no-cpt-norm} There exists a compactly generated open subgroup $L$ of $\SCpC_p(\absF{})$ 
which has property (NCNS).
		\item \label{item-G-top-simple} There exists a compactly generated open subgroup $L$ of $\SCpC_p(\absF{})$ which is topologically simple.
		\item \label{item-Ln-no-cpt-norm} {For all sufficiently large $n$ the subgroup $K_n$ is trivial.}
		\item \label{item-Ln-simpl}  {For all sufficiently large $n$ the group $L_n$ is abstractly simple.}
	\end{enumerate}
\end{Corollary}

\begin{proof}
	The implications (\ref{item-Ln-simpl}) $\Rightarrow$ (\ref{item-G-top-simple}) $\Rightarrow$ (\ref{item-G-no-cpt-norm}) are immediate,
	and the implication (\ref{item-Ln-no-cpt-norm}) $\Rightarrow$ (\ref{item-Ln-simpl}) holds by Corollary \ref{cor-Ln-simple-Nntrivial}.

``(\ref{item-G-no-cpt-norm})$\Rightarrow$ (\ref{item-Ln-no-cpt-norm})'' Suppose $L$ is as in (\ref{item-G-no-cpt-norm}). Since $L$ is compactly generated and $\SCpC_p(\absF{})$ is equal to the ascending union $\bigcup\limits_{n=1}^{\infty}L_n$ by Observation~\ref{Spn:increasing}, 
we have $L \subseteq L_n$ for sufficiently large $n$. By Lemma \ref{lem-nocptnormal-overgroup}(\ref{item-lemma-NCNS-overgroup}), $L_n$ has (NCNS) for any such $n$, so (\ref{item-Ln-no-cpt-norm}) holds. 
\end{proof}

{\begin{Question} Are the equivalent conditions of Corollary~\ref{cor-conditional-Ln-simple} true or false ? 
\end{Question}}

\section{Questions}\label{sec:questions}

\subsection{{(NCNS) for compactly generated open subgroups of $\Comm(\propF{p})$}}
\label{subsec:NCNS}

For an infinite tdlc group $L$, being topologically simple always implies $L$ has no non-trivial compact normal subgroup. Thus, the following problem is a weaker form of Question~\ref{qCM} (reformulated in terms of subgroups of $\Comm(\propF{p})$):

\begin{Problem} \label{problem-Comm-no-cpt-normal}
Let $\propF{}$ be a free pro-$p$ group of finite rank. Find a compactly generated open subgroup $L$ of $\Comm(\propF{})$ which has no non-trivial compact normal subgroup.
\end{Problem}

As before, we abbreviate 'no non-trivial compact normal subgroup' as (NCNS). Using some general arguments, we can reformulate Problem \ref{problem-Comm-no-cpt-normal} as follows: 

\begin{Proposition}
\label{firstreformulation}
Let $\propF{p}$ be a non-abelian free pro-$p$ group of finite rank.
The following are equivalent:
\begin{itemize}
\item[(a)] $\Comm(\propF{p})$ has a compactly generated open subgroup with (NCNS).
\item[(b)] There exists a compactly generated group $L$ with the following properties:
 \begin{enumerate}
	\item \label{item-loc-free-pro-p} $L$ has a compact open subgroup isomorphic to $\propF{p}$;
	\item \label{item-no-cpt-normal} $L$ has (NCNS);
	\item \label{item-trivial_QZ} $L$ has no non-trivial discrete normal subgroup.
\end{enumerate}
 \end{itemize}
\end{Proposition}
\begin{proof}``(a)$\Rightarrow$(b)'' Suppose that  $L\subseteq \Comm(\propF{p})$ is a compactly generated open subgroup with (NCNS). By Lemma \ref{lem-nocptnormal-overgroup}(\ref{item-lemma-NCNS-overgroup}), we can assume that $L$ contains $\propF{}$, in which case $L$ satisfies (\ref{item-loc-free-pro-p}) and
(\ref{item-no-cpt-normal}), and it remains to check (\ref{item-trivial_QZ}). Since $\propF{p}$ has trivial quasi-center
and $L$ is open in $\Comm(\propF{p})$, by Lemma \ref{lem-Comm-trivialQZ}(1) the group $L$ also has trivial quasi-center. On the other hand,
a discrete normal subgroup is always contained in the quasi-center, so $L$ satisfies (\ref{item-trivial_QZ}).
\skv

``(b)$\Rightarrow$(a)'' Now suppose that $L$ has properties (\ref{item-loc-free-pro-p}), (\ref{item-no-cpt-normal}) and (\ref{item-trivial_QZ}).
By (\ref{item-loc-free-pro-p}), we can apply Proposition \ref{prop-univ-Comm} which yields a continuous homomorphism $\psi : L \to \Comm(\propF{})$ with open image and discrete kernel. Property (\ref{item-trivial_QZ}) now implies that $\ker(\psi)$ is trivial. Hence the image of $L$ in $\Comm(\propF{})$ is isomorphic to $L$ as a topological group and is therefore a solution to Problem \ref{problem-Comm-no-cpt-normal} by
(\ref{item-no-cpt-normal}). 
 \end{proof}

Even leaving condition (\ref{item-trivial_QZ}) aside, we do not  know if there exists a compactly generated group $L$ satisfying (\ref{item-loc-free-pro-p}) and (\ref{item-no-cpt-normal}). 

\skv

Using the equality $\Comm(\propF{p})=\AComm(\propF{p})$ established in Proposition \ref{Comm=AComm}, we can reformulate Problem \ref{problem-Comm-no-cpt-normal} in a rather different way:

\begin{Proposition}[Reformulation of Problem \ref{problem-Comm-no-cpt-normal}]
\label{prop:reformulation}
The following are equivalent: \begin{enumerate}
	\item \label{item-prob-cpt-norm} $\Comm(\propF{p})$ has a compactly generated open subgroup with (NCNS).
	\item \label{item-reformulation} There exist finite  collections $\{\propU{p}_i\}_{i=1}^r$ of open subgroups of $\propF{p}$ and
	 automorphisms $\phi_i\in \Aut(\propU{p}_i)$ for $1\leq i\leq r$
	such that no nontrivial normal subgroup of $\propF{}$ is contained in every $\propU{p}_i$
	and invariant under every $\phi_i$.
\end{enumerate}
\end{Proposition}

\begin{proof}
We will use the following notation: for every finite collection of
	automorphisms $\calA=\{\phi_i\in \Aut(\propU{p}_i)\}_{i=1}^r$ as above (where each $\propU{p}_i$ is open in $\propF{p}$)
we define
$L_{\mathcal{A}}$ to be the subgroup of $\Comm(\propF{})$ generated by $\propF{}$ and $\mathcal{A}$. Then by definition $L_{\mathcal{A}}$ is compactly generated and open in $\Comm(\propF{})$. 

``(\ref{item-reformulation})$\Rightarrow$ (\ref{item-prob-cpt-norm})''. Suppose (\ref{item-reformulation}) holds for some finite set 
$\mathcal{A}=\{\phi_i\in \Aut(\propU{p}_i)\}_{i=1}^r$, and let $\propU{p} = \cap   \propU{p}_i$. The normal core of $\propU{p}$ in $L_{\mathcal{A}}$ is normal in $\propF{}$, contained in every $\propU{p}_i$ and invariant under every $\varphi \in \mathcal{A}$. Hence by our hypotheses this normal core is trivial. By Lemma \ref{lem-nocptnormal-overgroup}(\ref{item-lemma-NCNS-all-cpct}) this implies {that} $L_{\mathcal{A}}$ has (NCNS). {We can indeed invoke Lemma \ref{lem-nocptnormal-overgroup}(\ref{item-lemma-NCNS-all-cpct}) here since a free pro-$p$ group is torsion-free.}

\skv
``(\ref{item-prob-cpt-norm})$\Rightarrow$(\ref{item-reformulation})''. Now suppose (\ref{item-prob-cpt-norm}) holds. Ordered by inclusion, the subsets $\mathcal{A}$ as above form a directed set, the subgroups $(L_{\mathcal{A}})_\mathcal{A}$ form a directed system, and Proposition \ref{Comm=AComm} says that $\Comm(\propF{}) = \bigcup_\mathcal{A}  L_{\mathcal{A}}$ is their directed union. Hence if $L$ is a compactly generated open subgroup  of $\Comm(\propF{})$, then $L \subseteq L_{\mathcal{A}}$ {for some $\mathcal{A}$}. If we take $L$ satisfying (\ref{item-prob-cpt-norm}), then $L_{\mathcal{A}}$ also satisfies (\ref{item-prob-cpt-norm}) by Lemma \ref{lem-nocptnormal-overgroup}. Hence (\ref{item-reformulation}) holds.
\end{proof}

\subsection{Further questions}

In this subsection we discuss several additional questions motivated by the results of the paper. 

\subsection*{Subgroups invariant under the automorphism group of finite index subgroups}

Recall from Observation~\ref{obs:abstractcharacteristic} (applied with $m=p$) that if $F$ is a free group and $H$ is a normal subgroup of $F$ of index $p$, then no non-trivial subgroup of $H$ can be characteristic in both $H$ and $F$ (at least for some values of $\rk(F)$). We do not know if the analogous property holds for free pro-$p$ groups. More generally, one can ask the following:

\begin{Question}
\label{q:superinvariant}
Let $\propF{}$ be a non-abelian free pro-$p$ group of finite rank. Can one find a finite collection $\calU=\left\lbrace \propU{p}_1, \ldots, \propU{p}_n \right\rbrace $ of open subgroups of $\propF{}$ including $\propF{}$ itself such that the only subgroup of $\propF{}$ which is contained in $\propU{p}_i$ and characteristic in $\propU{p}_i$ for every $i$ is the trivial subgroup? 
\end{Question}

The difference between the reformulation of Problem \ref{problem-Comm-no-cpt-normal} from Proposition \ref{prop:reformulation} and Question~\ref{q:superinvariant} is that in Proposition \ref{prop:reformulation} the requirement is invariance under finitely many {automorphisms}, while in Question~\ref{q:superinvariant} the requirement is invariance under the whole $\Aut(\propU{p}_i)$ for every $i$. Since the automorphism group of a free pro-$p$ group is not topologically finitely generated with respect to the $A$-topology \cite{Ro}, these two properties are indeed different. In particular, a positive answer to Question~\ref{q:superinvariant} may not have an immediate impact with respect to Problem \ref{problem-Comm-no-cpt-normal}. 

\subsection*{Quotients of $\Comm(F)$}

\skv
Recall that the subgroup $\AComm(G)$ is normal in $\Comm(G)$ whenever $G$ is finitely generated.
If $F$ is a non-abelian free group of finite rank, by Theorem~\ref{thmA} every proper quotient of $\Comm(F)$ is a quotient of $\Comm(F)/ \AComm(F)$ and $\Comm(F)/ \AComm(F)$ is infinite.

\begin{Question}\label{q:CommACommquotient}
What else can be said about the quotient $\Comm(F)/\AComm(F)$?
\end{Question}

As a starting point, one can ask whether this quotient is non-abelian or whether it is finitely generated.

 \subsection*{Simplicity of $\Comm(\propF{p})$}

Let $\propF{p}$ be a non-abelian free pro-$p$ group of finite rank. Recall that by Corollary~\ref{Comm-prop-almost-simple},
$\Comm(\propF{p})$ is monolithic and $\mon(\Comm(\propF{p}))$ is abstractly simple, but we do not know how large $\mon(\Comm(\propF{p}))$ is
and lack an explicit description for it.

\begin{Question}
\label{q:commprop}$\empty$
\begin{enumerate}[(a)]
\item\label{commsimple} Is $\mon(\Comm(\propF{p}))=\Comm(\propF{p})$? Equivalently, is $\Comm(\propF{p})$ abstractly simple?
\item If the answer to \ref{commsimple} is negative, is the quotient $\Comm(\propF{p})/\mon(Comm(\propF{p}))$ finite? 
Equivalently, is $\Comm(\propF{p})$ virtually abstractly simple?
\end{enumerate}
\end{Question}

Note that $\mon(\Comm(\propF{p}))$ is contained in $\SComm(\propF{p})$ since $\SComm(\propF{p})$ is normal in $\Comm(\propF{p})$.
Also recall that the quotient $\Comm(\propF{p})/\SComm(\propF{p})$ is finite. Therefore Question~\ref{q:commprop}\ref{commsimple}
has a positive answer if and only if $\SComm(\propF{p})$ is abstractly simple and $\SComm(\propF{p})=\Comm(\propF{p})$.

On the other hand, one can ask whether $\mon(\Comm(\propF{p}))$ equals the group $G$ from Proposition~\ref{Comm-prop-contains-A-closure}
(we know that $\mon(\Comm(\propF{p}))$ contains this $G$). This seems unlikely, but answering this question in the negative would still be interesting as it might help find a more explicit description of $\mon(\Comm(\propF{p}))$.
\skv

One result used in the proof of Theorem~\ref{thmA} was the fact that for a free group $F$ with $\rk(F)\geq 3$, the group $\SAut(F)$ is perfect.

\begin{Question}
\label{q:scomm perfect}
Let $\propF{p}$ be a free pro-$p$ group of finite rank at least $3$.
\begin{itemize}
\item[(a)] Is $\SAut(\propF{p})$ abstractly perfect?
\item[(b)] Is $\SAut(\propF{p})$ topologically perfect with respect to the $A$-topology?
\end{itemize}
\end{Question} 

{Even a positive answer to Question~\ref{q:scomm perfect}(b) would have an interesting consequence for $\SComm(\propF{p})$, namely,
it would imply that $\SComm(\propF{p})$ is topologically simple with respect to the $\Aut$-topology defined
in \cite{BEW}.\footnote{The $\Aut$-topology on $\Comm(\propF{p})$ is the strongest topology which makes all the maps $\iota_U:\Aut(\propU{p})\to\Comm(\propF{p})$, with $\propU{p}$ open in $\propF{p}$, continuous with respect to the $A$-topology.}
Indeed, if $\SAut(\propF{p})$ is topologically perfect with respect to the $A$-topology, the analogue of Lemma~\ref{lem:abstractnormalclosure} 
remains true if $SA(F)$ is replaced by $\SAut(\propF{p})$ and the normal subgroup $N$ is assumed closed (with minimal changes to the proof),
and then one can deduce that $\SComm(\propF{p})$ with the $\Aut$-topology is topologically simple arguing as in 
Proposition~\ref{Comm-prop-contains-A-closure}.
}

 \subsection*{Dependency on the rank for $\Comm(\propF{p})$}

Now recall that Theorem~\ref{thm-depend-parameters} completely determines the isomorphism class of the $p$-commensurator $\Comm_p(F)$
as a function of $p$ and $\rk(F)$. Our next question asks whether the analogue of this theorem holds for the commensurators of free pro-$p$ groups,
for a fixed $p$:

\begin{Question}
\label{q:commensurable} Given integers $d,e>1$, let $\propF{p}_d$ and $\propF{p}_e$ be free pro-$p$ groups of ranks $d$ and $e$, respectively.
For which values of $d$ and $e$ are the groups $\Comm(\propF{p}_d)$ and $\Comm(\propF{p}_e)$ isomorphic as abstract groups?
\end{Question}

It is natural to expect the same answer as in  Theorem~\ref{thm-depend-parameters}. Indeed,
by the Schreier formula $\propF{p}_d$ and $\propF{p}_e$ are virtually isomorphic if and only if $\frac{d-1}{e-1}=p^j$ for some $j\in\dbZ$.
Thus if the latter condition on $d$ and $e$ holds, $\Comm(\propF{p}_d)$ and $\Comm(\propF{p}_e)$  are definitely isomorphic, even as topological
groups. If the condition does not hold, it is not hard to show that $\Comm(\propF{p}_d)$ and $\Comm(\propF{p}_e)$ cannot be isomorphic as topological groups. Hence showing that every isomorphism between $\Comm(\propF{p}_d)$ and $\Comm(\propF{p}_e)$ is necessarily continuous would solve Question \ref{q:commensurable}.

 \subsection*{On the local structure of $Q_n$}

Our last question deals with the tdlc groups $Q_n$ constructed in Section~\ref{sec-compact-generated}. Recall that the groups $Q_n$ are abstractly simple, and the main question left open in Section~\ref{sec-compact-generated} is whether the compact open subgroup $F(n)$ of $Q_n$ is a 
free pro-$p$ group. A related problem is to determine the type of the groups $Q_n$ according to the classification in \cite{CRW-Part2}.

Following the notation in \cite{CRW-Part2}, let $\calS$ denote the class of non-discrete topologically simple compactly generated tdlc groups. According to \cite[\S~1.4,~\S~6]{CRW-Part2}, any abstractly simple group $S$ in $\calS$ belongs to exactly one of the following three classes, which can be distinguished by the isomorphism type of any compact open subgroup:
\begin{enumerate}
\item (Locally h.j.i.) Some (and hence every) compact open subgroup of $S$ is hereditarily just-infinite.
\item (Micro-supported) $S$ has a minimal strongly proximal micro-supported action on the Cantor set.
\item (NPF type) Any group not belonging to the other two classes. 
\end{enumerate}
Omitting the precise definition of NPF type, we just mention that $S$ is of NPF type if and only if 
\begin{itemize}
\item[(i)] $S$ has an infinite non-open compact locally normal subgroup and
\item[(ii)] no two infinite compact locally normal subgroups commute elementwise.
\end{itemize}

Here a subgroup is called locally normal if it has open normalizer. It would be reasonable to expect NPF type to be the `generic' case.  However, at the time of writing, there are no examples of groups in $\calS$ that have been proven to be of NPF type. Many topological Kac--Moody groups over finite fields are conjectured to have NPF type. 

It is straightforward to show that if there exists $S\in\calS$ containing a non-abelian free pro-$p$ group $\propF{p}$ as an open subgroup, then $S$ has NPF type. So if the compact open subgroup $F(n)$ of $Q_n$ happens to be free pro-$p$, then $Q_n$ would in particular produce examples of groups of NPF type. In any case, it is natural to ask the following:  

\begin{Question}\label{problem:Qntype}
Do the groups $Q_n$ have NPF type, at least for sufficiently large $n$? 
\end{Question}

\end{document}